\documentclass[a4paper,reqno]{amsart}
\usepackage{CJKutf8}   

\usepackage{amsmath,amsthm,amssymb,amsfonts,amsbsy}
\usepackage{enumitem,color,graphicx,cite,xcolor}
\usepackage[all,pdf]{xy}

\usepackage{geometry}  
\usepackage[backref=page,hyperindex=true,CJKbookmarks=true,
colorlinks,linkcolor=blue,anchorcolor=red,citecolor=cyan]{hyperref}

\newtheorem{thm}{Theorem}[section]
\newtheorem{cor}[thm]{Corollary}

\newtheorem{prop}[thm]{Proposition}
\newtheorem{lem}[thm]{Lemma}
\newtheorem{rmk}[thm]{Remark}
\newtheorem{qst}[thm]{Question}

\newtheorem{example}[thm]{Example}
\newtheorem{claim}[thm]{Claim}

\theoremstyle{plain}
\newtheorem{maintheorem}{Theorem}

\newtheorem{main-cor}[maintheorem]{Corollary}
\newtheorem{main-prop}[maintheorem]{Proposition}

\numberwithin{equation}{section}

\newcommand{\quand}{\quad\text{and}\quad}

\newcommand{\mcc}{\mathcal{C}}
\newcommand{\mcl}{\mathcal{L}}
\newcommand{\mcf}{\mathcal{F}}

\newcommand{\mtt}{\mathbb{T}^2\to\mathbb{T}^2}
\newcommand{\mrr}{\mathbb{R}^2\to\mathbb{R}^2}
\newcommand{\mrt}{\mathbb{R}^2\to\mathbb{T}^2}

\def\NN{\mathbb{N}}
\def\RR{\mathbb{R}}

\def\TT{\mathbb{T}}
\def\ZZ{\mathbb{Z}}
\def\tildeL{\widetilde{\mathcal{L}}}

\def\e{{\varepsilon}}

\renewcommand*{\backref}[1]{}
\renewcommand*{\backrefalt}[4]{%
	\ifcase #1 (Not cited.)%
	\or        (Cited on page~#2.)%
	\else      (Cited on pages~#2.)%
	\fi}

\address[R. Gu]{School of Mathematical Sciences, 
    Tongji University,
    Shanghai 200092, P.R. China}
\email{rhgu@tongji.edu.cn}

\address[M. Xia]{School of Mathematical Sciences,
    Dalian University of Technology,
    Dalian, Liaoning 116024, P.R. China}
\email{xiamingyang@dlut.edu.cn}

\title[Semi-conjugacy of DA endomorphisms]
{Semi-conjugacy rigidity for endomorphisms derived from Anosov on the 2-torus}
\author[R. Gu and M. Xia]{Ruihao Gu and Mingyang Xia}

\date{\today}

\begin{document}
\begin{CJK}{UTF8}{gbsn}
     
\begin{abstract}
	\begin{sloppypar}
        Let $f$ be a non-invertible partially hyperbolic endomorphism on $\TT^2$ 
        which is derived from a non-expanding Anosov endomorphism. 
        Differing from the case of diffeomorphisms derived from Anosov automorphisms, 
        there is no a priori semi-conjugacy between $f$ and its linearization on $\TT^2$.
        We show that $f$ is semi-conjugate to its linearization 
        if and only if 
        $f$ admits a partially hyperbolic splitting with two $Df$-invariant subbundles. 
		
        Moreover, if we assume that $f$ has an unstable subbundle,
        then the semi-conjugacy is exactly a topological conjugacy, 
        and the center Lyapunov exponents of the periodic points of $f$ coincide 
        with its linearization. 
        In particular, $f$ is an Anosov endomorphism 
        and the conjugacy is smooth along the stable foliation.
		
        For the case that $f$ has a stable subbundle, 
        there is still some rigidity in its stable Lyapunov exponents. 
        However, we also give examples which admit a partially hyperbolic splitting 
        with center subbundle 
        but the semi-conjugacy is indeed non-injective.  
        
        Finally, we present some applications under the volume-preserving assumption.
	\end{sloppypar}
\end{abstract}
	
\maketitle

\section{Introduction}

    
    The Smale's classification conjecture \cite{Smale1967} of Anosov diffeomorphisms and expanding endomorphisms gives rise to numerous important works \cite{Franks1969,F70,Gro81,Manning1974,Newhouse1970,Shub1969}.
    It is well known that an expanding endomorphism is always conjugate to an algebraic model of an infra-nilmanifold \cite{F70,Gro81,Shub1969}, 
    while the conjugacy classification for Anosov diffeomorphisms is widely open. 
    A celebrated work of Franks \cite{F70} completely classifies, up to semi-conjugacy, the homotopic class of Anosov automorphisms on the torus. 
    Such systems are called \textit{derived from Anosov} (abbr. \textit{DA}).  
    Precisely, let $A$ be a hyperbolic toral automorphism and $f$ be a diffeomorphism homotopic to $A$, 
    then there exists a continuous surjection $h$ homotopic to identity such that $h\circ f=A\circ h$. 
    Here, $h$ can be obtained by projecting a semi-conjugacy between lifts of $f$ and $A$ on the universal cover.

    DA diffeomorphisms have been studied for several decades since Smale \cite{Smale1967} introduced them for a non-trivial basic set.
	Later, Shub \cite{Shub1971} and Ma\~{n}\'{e} \cite{M78} constructed DA diffeomorphisms on $\TT^4$ and $\TT^3$, respectively, 
    to illustrate that uniform hyperbolicity is not a necessary condition for topological transitivity in a robust manner. 
	Their examples are some of the most well-known DA examples, which are \textit{partially hyperbolic}.  
    In a certain sense, partially hyperbolic diffeomorphisms on $\TT^3$ 
    correspond to partially hyperbolic endomorphisms on $\TT^2$.
    The partially hyperbolic DA endomorphisms on $\TT^2$, called \emph{Ma\~{n}\'{e}-type}, are the main objects concerned in the present paper.
 
    Benefiting from the semi-conjugacy to hyperbolic automorphisms, 
    partially hyperbolic DA diffeomorphisms on $\TT^3$, as one of the three types of partially hyperbolic diffeomorphisms on 3-manifolds with solvable fundamental group \cite{HP15}, have been studied extensively. 
    They have nice geometric structure including dynamical coherence and leaf conjugacy to linear part \cite{BBI2004,FPS2014,Ha13,Potrie2015}. 
    Moreover, they have abundant dynamical properties, such as 
    robust minimality of strong foliations \cite{HUY22},
    uniqueness of measures of maximal entropy \cite{U12,BFSV12}, 
    uniqueness of equilibrium states \cite{CFT19},
    flexibility of center Lyapunov exponents \cite{PT2014,BKR2022,CS2022}, 
    rigidity of center Lyapunov exponents \cite{GS20,HS21}, 
    ergodicity with respect to the volume \cite{HU14,GS20} even Bernoulli in certain settings \cite{PTV2018} and so on.

    However, the situation is different for the toral endomorphisms. 
    There may not exist a semi-conjugacy, on the torus, relating the original endomorphism and its hyperbolic linearization, though there is still a semi-conjugacy on the level of the universal cover space \cite{AH94}.
    Indeed, Franks mentions in \cite{F70}  that his proof for the existence of the semi-conjugacy between the DA diffeomorphism and its linearization fails in the non-invertible case. 
    Precisely, in \cite[Question in Section 2]{F70}, Franks asks the following question.

    \begin{qst}\label{qst-franks}
        Let $A:M\to M$ be a hyperbolic infra-nilmanifold endomorphism and $f:M\to M$ be a local diffeomorphism homotopic to $A$. Is there a semi-conjugacy (homotopic to the identity ${\rm Id}_M$) between $f$ and $A$?
    \end{qst}

    \begin{rmk}
    Here we briefly explain the reason of Franks' method being invalid in the non-invertible case. 
    Let $A$ be a hyperbolic endomorphism of $\TT^2$, $f$ be a local diffeomorphism on $\TT^2$ homotopic to $A$, and $A,F:\RR^2\to\RR^2$ be their lifts, respectively.
    The key step in Franks' work, see \cite[Proposition 2.1]{F70}, is constructing an isomorphism $T:Q\to Q$, where 
    \[Q:=\big\{ H\in C^0(\RR^2)\ |\ H(x+n)=H(x) \ {\rm for \ any}\  x\in\RR^2\ {\rm and}\  n\in\ZZ^2 \big\},\]
    and $T(H):=A^{-1}\circ H\circ F$. 
    However, when $A$ is non-invertible, 
    $T$ is not an isomorphism on $Q$.
    To see this, one needs to notice that the period of $T^{-1}(H)=A\circ H\circ F^{-1}$ is \emph{a priori} only $A\ZZ^2$-type. 
    Indeed, let $x\in\RR^2$ and $n^*\in A\ZZ^2$, that is, 
    $n^*=An_0$ for some $n_0\in\ZZ^2$.
    Then one has 
    \[
    A\circ H\circ F^{-1}(x+n^*)
    =A\circ H\big(F^{-1}(x)+n_0\big)
    =A\circ H\circ F^{-1}(x).
    \] 
    Note that the first equation requires $n^*$ to be $A\ZZ^2$-type.
    Hence, $T^{-1}(Q)$ may not be $Q$, and $T$ may not be an isomorphism on $Q$.

    \end{rmk}

    In the present paper, we would like to respond to the question of Franks 
    and show that there is no \emph{a priori} semi-conjugacy, in the homotopic class of identity, 
    between the original endomorphism and its hyperbolic linearization on the $2$-torus. 
    Moreover, we will explore
    \emph{what rigidity the semi-conjugacy on the torus would impose on the non-invertible case of derived from Anosov systems.}
	More precisely, we are going to focus on Ma\~{n}\'{e}-type DA endomorphisms on $\TT^2$. 

    Throughout this paper, 
    by \emph{DA} we mean partially hyperbolic and derived from Anosov; 
    by \emph{endomorphisms} we mean non-invertible local diffeomorphisms which are not expanding.
    Differing from the diffeomorphism case, a partially hyperbolic endomorphism has no a priori invariant dominated splitting in the sense of the direct sum of subbundles. Instead of this, we define the partial hyperbolicity by using the language of \emph{cone-fields}. For conciseness, we leave the precise definitions in Subsection \ref{subsec ph}.
    We also note that the absence of invariant dominated splitting is an obstacle for generalizing the results or the analysis of partially hyperbolic diffeomorphisms to the endomorphism case, as we will see in the next subsection.


\subsection{The semi-conjugacy and special property}

    Anosov endomorphisms are of course DA systems.
    Differing from Anosov diffeomorphisms and expanding maps, 
    every Anosov endomorphism $f$ of a closed Riemannian manifold $M$ is not structurally stable \cite{ManePugh1975,Pr76}, 
    i.e., in any $C^1$-neighborhood of $f$, there exists an endomorphism which cannot be conjugate to $f$. 
    Reviewing the proofs in \cite{ManePugh1975,Pr76},  
    an obvious obstruction for the existence of conjugacy is lacking of an invariant unstable bundle. 
    Indeed, an Anosov endomorphism on the torus is conjugate to its linearization if and only if it admits an invariant unstable bundle \cite{AH94,S95}.
    
    An Anosov endomorphism with invariant unstable bundle is called \textit{special} 
    (see the definition in Subsection \ref{subsub sec se}). 
    In fact, the proof in \cite{ManePugh1975} implies that the special property is not $C^r$-open ($r\geqslant 1$). 
    Moreover, Przytycki \cite{Pr76} $C^1$-slightly perturbs a linear Anosov endomorphism on the $3$-torus 
    such that the unstable directions of the perturbation at a certain point 
    contain a curve which is homeomorphic to an interval.

    The non-special property is generic in partially hyperbolic endomorphisms \cite{CM22,MT16}.
    We call a DA endomorphism \textit{special}
    if it admits an invariant partially hyperbolic splitting into subbundles (see Subsection \ref{subsec ph} in detail).
    To establish a connection between the existence of semi-conjugacy and the special property for DA endomorphisms, we first give the following result.

    \begin{maintheorem}\label{main-thm-scu}
        Let $f:\mtt$ be a $C^1$-smooth DA endomorphism with a hyperbolic linearization $A:\mtt$.
        If $f$ is special, then it is semi-conjugate to $A$.
    \end{maintheorem}

    \begin{rmk}
        The proofs of Theorem \ref{main-thm-scu} and the following results in this section 
        need the dynamical coherence, i.e., the center bundle on the universal cover space is integrable. 
        DA endomorphisms on $\TT^2$ are actually dynamically coherent \cite{HH21}, 
        see also \cite{HSW19} in the setting of absolutely partially hyperbolic endomorphisms 
        (see the definition in Subsection \ref{subsec ph}) on $\TT^2$. 
        However, it is still unknown if it holds or not for high-dimensional case. 
        Indeed, for DA diffeomorphisms or absolutely partially hyperbolic diffeomorphisms,  
        there are similar results on $\TT^3$ and open questions on $\TT^d\ (d\geqslant 4)$ 
        \cite{BBI2004,FPS2014,Potrie2015,HHU16}. 
    \end{rmk}

    It seems that the special property is also a necessary condition for the existence of semi-conjugacy, 
    because it is true for the Anosov case as mentioned above. 
    Indeed, considering a non-special Anosov endomorphism $f$ on $\TT^2$, 
    there is a point $x$ with at least two transverse unstable manifolds $\mcf^u_1(x)$ and $\mcf^u_2(x)$ \cite{Pr76}. 
    If $f$ is conjugate via $h$ to its linearization $A$,
    $h$ will map $\mcf^u_1(x)$ and $\mcf^u_2(x)$ to the unstable manifold of $A$ at $h(x)$. 
    It contradicts the fact that $h$ is a homeomorphism. 
    However, this trick does not work for semi-conjugacy or DA endomorphisms. 

    Besides the difference between Anosov endomorphisms and DA endomorphisms, 
    there is an interesting observation restricted to the context of DA endomorphisms. 
    \begin{itemize}
        \item When a DA endomorphism admits a stable bundle (i.e., it is uniformly contracting on a subbundle),
    called \textit{sc-DA} 
    (see the definition in Subsection \ref{subsubsec DA}), 
    by a similar argument to the Anosov case, it can be easily shown that the semi-conjugacy guarantees the special property,  
    see Corollary \ref{2 cor semi-conj. Fc}. 
    \item When a DA endomorphism admits an unstable cone-field (i.e., it is uniformly expanding on a cone-field), 
    called \textit{cu-DA} 
    (see also Subsection \ref{subsubsec DA}), 
    the proof will be much more complicated. 
    Indeed, the study of the $cu$-DA endomorphism needs an intuitively stronger rigidity of the semi-conjugacy as follows: 
    the existence of semi-conjugacy between $f$ and its linearization implies that \textit{the center Lyapunov exponents of $f$ on the periodic points coincide with its linearization}. 
    \end{itemize}
    Then we conclude that a DA endomorphism (for both cases of $sc$-DA and $cu$-DA) is semi-conjugate to its linearization if and only if it is special. 
    
    In what follows, we concentrate on the rigidity phenomenon 
    that is the main body of the paper.


\subsection{Rigidity of the semi-conjugacy} 

    The rigidity issue in smooth dynamics focuses on the phenomenon that ``weak equivalence implies strong equivalence''. 
    A typical type of rigidity question is the following. 
    Let Anosov local diffeomorphisms $f,g:M\to M$ be conjugate via a homeomorphism $h:M\to M$, that is, $h\circ f=g\circ h$. 
    It is clear that a necessary condition of $h$ being smooth is the matrices $Df^{\pi(p)}(p)$ and $Dg^{\pi(p)}(h(p))$ being conjugate for every periodic point $p$ of $f$ with period $\pi(p)$. 
    The rigidity means that such a necessary condition is also sufficient for the conjugacy $h$ being smooth. 
    Besides matching the derivatives of return maps at periodic points, one can also consider more necessary conditions of the smooth conjugacy $h$, 
    such as matching other periodic data (the Lyapunov exponents or Jacobian), matching the Lyapunov exponents with respect to the volume, and so on. 
    There are many results on this type of rigidity, e.g., \cite{CV23,dL87,dL92,G08,M23,SY19,GRH23}.

    In this subsection, we will see that the semi-conjugacy guarantees partially matching periodic data. 
    For DA endomorphisms on $\TT^2$, we will present the rigidity 
    which the existence of the semi-conjugacy on the torus imposes on the non-invertible dynamics, 
    with respect to the two cases of Ma\~{n}\'{e}-type: $cu$-DA and $sc$-DA. 
    Actually, the study of the differences between these two types of DA endomorphisms 
    seems to be a subject of independent interest as well.

For the DA endomorphisms with unstable cone-fields, we have the following result. 
    Denote by Per$(f)$ the periodic points set of $f$,
    and $\lambda^s(p,f)$ the stable Lyapunov exponent at $p\in {\rm Per}(f)$. 
    
\begin{maintheorem}\label{main-thm-cu}
	Let $f:\mtt$ be a $C^{1}$-smooth 
    $cu$-DA endomorphism semi-conjugate to its hyperbolic linearization $A:\mtt$. 
    Then $f$ is a special Anosov endomorphism.  
    In particular, the semi-conjugacy is an injection, hence, the semi-conjugacy is a topological conjugacy.  
    Moreover, $\lambda^s(p,f)=\lambda^s(A)$, for any $p\in {\rm Per}(f)$.
\end{maintheorem}

\begin{rmk}\label{rmk AGGS}
    The recent work about the rigidity of Anosov endomorphisms on the existence of topological conjugacy \cite{AGGS23} shows that a $C^{1+}$-smooth Anosov endomorphism $g$ on $\TT^2$ is conjugate to its linearization $A$ if and only if $g$ is special
    if and only if its stable Lyapunov exponents at periodic points are the same as $A$. 
    In particular, the $C^{1+}$-regularity assumption is used only for proving that the constant periodic stable Lyapunov exponent implies the existence of conjugacy and the special property.
    A key tool in \cite{AGGS23} is the $A$-preimage set of every point becoming dense exponentially, see Lemma \ref{2 lem preimage dense linear}. 
    Note that $g$ inherits this property, since the conjugacy is in fact H\"older continuous\cite{KH95}. 
    However, the semi-conjugacy relating $f$ and $A$ in Theorem \ref{main-thm-cu} is a priori non-injective. 
    And the method in the present paper is more general.
\end{rmk}

For the DA endomorphisms with stable subbundles, we have the following result.

\begin{maintheorem}\label{main-thm-sc}
    Let $f:\mtt$ be a  $C^1$-smooth 
    $sc$-DA endomorphism semi-conjugate to its hyperbolic linearization $A:\mtt$ via a surjection $h:\mtt$.  
    Then $f$ is special. Moreover, the closure of $h$-injective points set
    \[
    \Lambda:=\overline{\big\{x\in\TT^2 ~|~ h^{-1}\circ h(x)=\{x\}\big\}}
    \]
    is invariant and satisfies
    \begin{enumerate}
         \item $\overline{{\rm Per}(f|_{\Lambda})}=\Lambda$;
         \item $\lambda^s(p,f)=\lambda^s(A)$, 
         for any $p\in {\rm Per}(f|_{\Lambda})$.
    \end{enumerate}
\end{maintheorem}

\begin{rmk}
    There is also motivation for the investigation of $sc$-DA endomorphisms 
    for their relation to multivalued expanding mappings on branched manifolds. 
    We refer to \cite[page 76]{Pr87}.
\end{rmk}

    Moreover, related to the result of $sc$-DA endomorphisms, 
    we also have the following proposition indicating the non-negligible gap 
    between the semi-conjugacy and topological conjugacy.

\begin{main-prop}\label{main-prop}
    For any Anosov endomorphism  $A:\mtt$, 
    there exists a $C^\infty$-smooth 
    $sc$-DA endomorphism $g_0:\mtt$ semi-conjugate to $A$  
    such that every semi-conjugacy $h:\mtt$ homotopic to ${\rm Id}_{\TT^2}$ is indeed non-injective.
\end{main-prop}

Combining the results above, we have some reformulations.
    Precisely, Theorems \ref{main-thm-scu} and \ref{main-thm-cu} imply the following corollary.
    Note that when $f$ is special, 
    there will be a natural way to define the map $h:\mtt$ 
    as the projection of a priori semi-conjugacy between the lifted maps on $\RR^2$.

\begin{main-cor}\label{main-cor-cu}
	Let $f:\mtt$ be a $C^{1+\alpha}$-smooth 
    $cu$-DA endomorphism and $A:\mtt$ be its hyperbolic linearization. 
    Then the following are equivalent:
  \begin{enumerate}
      \item $f$ is semi-conjugate to $A$ via $h:\mtt$;
      \item $f$ is conjugate to $A$ via $h:\mtt$;
      \item $f$ is special;
      \item $f$ is Anosov and $\lambda^s(p,f)=\lambda^s(A)$, for any $p\in {\rm Per}(f)$.
  \end{enumerate}
    In particular, each of them implies that the conjugacy $h$ is $C^{1+\alpha}$-smooth along the stable foliation.
\end{main-cor}		

\begin{rmk}\label{rmk C1+}
    We require $C^{1+\alpha}$-regularity of $f$ in Corollary \ref{main-cor-cu}. 
    Especially, for proving ``(4)$\implies$(1),(2),(3)" and the smoothness of $h$, 
    we will apply the main result of \cite{AGGS23}  mentioned in Remark \ref{rmk AGGS}, 
    where it needs the $C^{1+}$-smoothness to apply the Liv\v{s}ic Theorem (see \cite[Section 4 and Section 5]{AGGS23} for more details).
\end{rmk}

\begin{rmk}
    An interesting rigidity result in Corollary \ref{main-cor-cu} is that if the semi-conjugacy exists, 
    it is actually a conjugacy and smooth along the stable foliation. 
    Such a phenomenon as lower-regularity automatically implying higher-regularity is usually called ``bootstrap''. 
    To see recent works on bootstrap, we refer to  
    \cite{G17,KSW23} for Anosov diffeomorphisms,  
    \cite{GKS23} for partially hyperbolic diffeomorphisms,
    and \cite{GRH23b} for Anosov flows. 
\end{rmk}

    From Theorems \ref{main-thm-scu} and \ref{main-thm-sc}, we have the following corollary immediately.

\begin{main-cor}\label{main-cor-sc}
    Let $f:\mtt$ be a $C^{1}$-smooth 
    $sc$-DA endomorphism and $A:\mtt$ be its hyperbolic linearization. 
    Then the following are equivalent:
    \begin{enumerate}
      \item $f$ is semi-conjugate to $A$ via $h:\mtt$;
      \item $f$ is special.
    \end{enumerate}
    In particular, each of them implies that 
    the closure of $h$-injective points set
     \[
     \Lambda:=\overline{\big\{x\in\TT^2~ |~  h^{-1}\circ h(x)=\{x\}\big\}}
     \]
    is invariant and satisfies
     \begin{enumerate}
         \item $\overline{{\rm Per}(f|_{\Lambda})}=\Lambda$;
         \item $\lambda^s(p,f)=\lambda^s(A)$, for any $p\in {\rm Per}(f|_{\Lambda})$.
     \end{enumerate}
\end{main-cor} 

    Comparing this result with Corollary \ref{main-cor-cu}, there are the following remark and question.

\begin{rmk}\label{1 rmk sc}
    Based on Proposition \ref{main-prop}, 
    we further give an example of $sc$-DA endomorphism with periodic stable Lyapunov exponents 
    being the same as its linearization on an invariant set saturated by the stable foliation,
    but it is not semi-conjugate to its linearization (see Proposition \ref{5 prop counter-example}).
\end{rmk}

\begin{qst}\label{qst 1 sc-smooth}
    Under the assumption of Corollary \ref{main-cor-sc},  
    it is known that $h$ preserves the stable foliation and is a homeomorphism restricted to each stable leaf 
    (see Proposition \ref{2 prop foliation on R2}). 
    There is still a question whether the semi-conjugacy is smooth along each stable leaf or not.
    Note that even if $f$ is $C^2$, there is no Liv\v{s}ic Theorem \cite{L72} in this situation,  
    which is important to the rigidity issue 
    when one wants to get the smoothness of the conjugacy from the periodic data, see for example \cite{dL92,G08,AGGS23}.
\end{qst}


\subsection{Applications for the volume-preserving case}  

    In what follows, we would like to give some applications. 
    Recall that Theorem \ref{main-thm-cu} implies that a special $cu$-DA endomorphism must be Anosov, while Proposition \ref{main-prop} asserts that 
    a special $sc$-DA endomorphism may not have to be Anosov. 
    A direct reason behind this difference is that the injective points set $\Lambda$ in Theorem \ref{main-thm-sc} may not be hyperbolic for the $sc$-DA case, while it must be hyperbolic in the setting of $cu$-DA (see Lemma \ref{6 lem attractor} and Remark \ref{rmk cs no hyper} for more details).  
    Thus, to show more rigidity for the $sc$-DA case, we introduce some assumptions to control the dynamics on the center direction, in particular, we focus on the volume-preserving condition. 
    
    {\color{black} In this paper, $f:\mtt$ is called \textit{volume-preserving}, 
    if $f$ admits an invariant smooth measure $m$, 
    i.e., the density  $\rho:\TT^2\to \RR_+$ of $m$ is smooth, with $dm=\rho d{\rm Leb}$, 
    and ${\rm Leb}$ is the Lebesgue measure of $\TT^2$. 
    Such a measure $m$ will be referred to as a \textit{smooth volume measure}.} 

    First, parallel to Corollary \ref{main-cor-cu}, we have the following result as a corollary of Theorem \ref{main-thm-sc}.

\begin{main-cor}\label{main-cor-scvp} 
    Let $f:\mtt$ be a  $C^{1+\alpha}$-smooth  $sc$-DA endomorphism and $A:\mtt$ be its hyperbolic linearization.  
    If $f$ is volume-preserving, 
    then the following are equivalent: 
    \begin{enumerate} 
        \item $f$ is semi-conjugate to $A$ via $h:\mtt$; 
        \item $f$ is conjugate to $A$ via $h:\mtt$; 
        \item $f$ is special;
        \item $f$ is Anosov and $\lambda^s(p,f)=\lambda^s(A)$, for any $p\in {\rm Per}(f)$. 
    \end{enumerate} 
    In particular,
    each of items implies that the conjugacy $h$ is $C^{1+\alpha}$-smooth 
    along the stable foliation.  
\end{main-cor}

\begin{rmk}
    As we mentioned in Remark \ref{rmk C1+}, the only two parts of Corollary \ref{main-cor-scvp} requiring $C^{1+\alpha}$-regularity are ``(4) $\implies$ (1),(2),(3)''  and the smoothness of $h$.
 \end{rmk}

    So the $sc$-DA case (Corollary \ref{main-cor-scvp}) and the $cu$-DA case (Corollary \ref{main-cor-cu}) 
    can be unified in the volume-preserving setting. 
    Moreover, since a $C^{1+\alpha}$ volume-preserving Anosov endomorphism is ergodic \cite{MT16},
    as an immediate corollary, we also obtain the ergodicity of DA endomorphisms as follows.

\begin{main-cor}\label{main-cor-ergodic}
    Let $f:\mtt$ be a special $C^{1+\alpha}$-smooth DA endomorphism.   
    If $f$ is volume-preserving, then it is ergodic with respect to the volume measure.
\end{main-cor}

\begin{qst}
    Is a $C^{1+\alpha}$-smooth volume-preserving DA endomorphism on $\TT^2$  ergodic with respect to the volume without the special assumption?
    Note that a volume-preserving $C^{1+\alpha}$-smooth DA diffeomorphism $g$ on $\TT^3$ is always special and semi-conjugate to its linearization,  
    so Corollary \ref{main-cor-ergodic} can be seen as an analogue of the main results of \cite{GS20, HU14} in the non-invertible case.
\end{qst}

    Finally, we consider special DA endomorphisms with the same periodic Jacobian as the linearizations.
    We can get that these DA endomorphisms are smoothly conjugate to their linearization. 
    We refer to \cite{GRH23} for a Jacobian rigidity result for expanding maps,
    where the authors introduce a method called ``matching function'' (see also \cite{GRH23b} for this technique in detail) that is not used for the endomorphism case here. 
    For $r>1$, denote by 
    \[r_*=\Big\{ 
        \begin{array}{lr}
        r-1+{\rm Lip}, \  r\in\NN  \\  r, \ r\notin  \NN \ {\rm or}\ r=+\infty
        \end{array}.
    \]
 
\begin{main-cor}\label{main-cor-Jac-rigidity}
    Let $f:\mtt$ be a $C^r$-smooth $(r>1)$ DA endomorphism.  
    Assume that $f$ is semi-conjugate to its hyperbolic linearization $A$, 
    and for any $n\in\NN$ and $p\in{\rm Fix}(f^n)$, 
    \[{\rm Jac}(Df^n(p))=|{\rm det}(A)|^n.\] 
    Then $f$ is a volume-preserving Anosov endomorphism and $C^{r_*}$-conjugate to $A$.
\end{main-cor}

\begin{rmk}\label{1 rmk vp and jacobian}
    It is known that when $g$ is a $C^{1+\alpha}$ transitive Anosov diffeomorphism, 
    $g$ {\color{black}admitting an invariant measure which is absolutely continuous with respect to Lebesgue measure}
    is equivalent to 
    $g$ having the same periodic Jacobian as its linearization (\cite[Theorem 4.14]{B75}). 
    However, for $C^{1+\alpha}$ transitive Anosov endomorphisms, these two conditions are not equivalent in general.  
    In fact, when $f$ is a transitive and non-invertible Anosov endomorphism, 
    the Jacobian condition will lead to an invariant measure {\color{black}with H\"older density}
    by the same method as \cite[Theorem 4.14]{B75} employing Liv\v{s}ic Theorem. 
    {\color{black} But the converse is not true, even if $f$ preserves a smooth measure.}
    And we give a counterexample in Example \ref{7 example}. 
\end{rmk}


\subsection{Comparison to the case of DA diffeomorphisms}

    Other main motivations of this paper are inspired 
    by several interesting results for DA diffeomorphisms
    on $\TT^3$ \cite{HU14,GS20,HS21}
    and the related differences (also similarities) corresponding to the case of DA endomorphisms. 
    Here we would like to make them more clear. 
    
    Let $g:\TT^3\to\TT^3$ be a $C^{1+\alpha}$-smooth DA diffeomorphism 
    with a partially hyperbolic splitting $T\TT^3=E^s_g\oplus E^c_g\oplus E^u_g$, 
    and $L:\TT^3\to \TT^3$ be its hyperbolic linearization 
    with a partially hyperbolic splitting $T\TT^3=L^{ss}\oplus L^{cs}\oplus L^u$. Let $f:\mtt$ be a DA endomorphism with linear part $A$.
    
\begin{enumerate}
    \item In \cite{HU14}, Hammerlindl and Ures proved that 
        if $g$ is volume-preserving (or transitive) and $su$-integrable, 
        namely, the bundles $E^s_g$ and $E^u_g$ are jointly integrable,  
        then it is topologically conjugate to $L$. 
        This phenomenon is similar to ``$(3) \Rightarrow (2)$'' of both Corollaries \ref{main-cor-cu} and \ref{main-cor-scvp}, 
        if one sees the special property as the integrability of the dominating direction $E^2_f$ (for definition see \eqref{eq. dominating direction}) 
        joint with the direction induced by the generation set 
            \[P(x)=\left\{y\in\TT^2\ |\ f^k(y)=f^k(x),\ {\rm for}\ {\rm some}\ k\in\NN \right\}.\] 
        It's worth mentioning that
        this set $P(x)$ can be regarded as the strong stable ``leaf'' of $x$. 
        But in Corollary \ref{main-cor-cu}, 
        we will not assume that $f$ is volume-preserving or transitive.
   
    \item  In \cite{GS20}, Gan and Shi proved that 
        if $g$ is $su$-integrable and conjugate to $L$, 
        then $g$ is Anosov and its periodic points admit the same center Lyapunov exponents as $L$. 
        This result has the same taste as Theorems \ref{main-thm-cu} and \ref{main-thm-sc}, 
        if one keeps seeing $P(x)$ as the strong stable ``leaf'' of $f$. 
        However, we would not assume that 
        $f$ is special (corresponding to the $su$-integrability of $g$) in advance, 
        or there is a conjugacy between $f$ and $A$. 
        Here we only need a semi-conjugacy on the torus which always exists in the case of DA diffeomorphisms. 
   
    \item In \cite{HS21}, Hammerlindl and Shi further proved that 
        the existence of $su$-leaves of $g$, in fact the boundaries of open $su$-accessible classes, 
        implies that $g$ is Anosov and admits the same periodic center Lyapunov exponents as $L$. 
        Comparing with \cite{GS20}, 
        they need not \emph{a priori} the existence of conjugacy or $su$-integrability on the whole $\TT^3$. 
        Here, in Theorems \ref{main-thm-cu} and \ref{main-thm-sc}, 
        the semi-conjugacy between $f$ and $A$ actually guarantees that 
        there is some submanifold (not \emph{a priori} the ambient manifold)
        which is tangent to the dominating direction $E^2_f$ and independent of the negative orbits, 
        see also Corollary \ref{2 cor semi-conj. Fc} and Proposition \ref{2 prop injective point set}. 
\end{enumerate}
    
    We also point out that 
    the $C^{1+\alpha}$-regularity of $g$ is necessary in the proofs of \cite{HU14,GS20,HS21}. 
    However, 
    the $C^1$-regularity is enough for most results in the present paper, except for  
    \begin{itemize}
        \item applying the Liv\v{s}ic Theorem 
        to get that the constant periodic stable Lyapunov exponent of $f$ implies $f$ being special in Corollaries \ref{main-cor-cu} and \ref{main-cor-scvp},
        and to show the smoothness of the conjugacy in Corollaries \ref{main-cor-cu}, \ref{main-cor-scvp} and \ref{main-cor-Jac-rigidity};
        
        \item applying the Hopf's argument to prove the ergodicity in Corollary \ref{main-cor-ergodic}.
    \end{itemize}

    Finally, considering the special property as a kind of joint integrability, we concentrate on its rigidity.
    Based on the work of \cite{GS20}, Gogolev and Shi \cite{GoS23} recently extend a wide range of the rigidity on integrability in the setting of Anosov diffeomorphisms on higher-dimensional tori. 
    Besides the hyperbolic case, there are also works on the rigidity of integrability for partially hyperbolic diffeomorphisms \cite{AV10,FRH05}. 
    Let us explain more about them for comparison with our results here. 
 
    Let $L:\TT^4\to \TT^4$ be a totally irreducible automorphism with exactly two eigenvalues of absolute value one. 
    Then there is no Franks' semi-conjugacy in this partially hyperbolic situation either. 
    Applying the KAM technique, Rodriguez Hertz \cite{FRH05} shows that 
    for a $C^{22}$-small perturbation $g$ which is also partially hyperbolic, 
    if $g$ is $su$-integrable, then $g$ is topologically conjugate to $L$ and the conjugacy is smooth along the center foliation. 
    Note that this result could be compared with Corollary \ref{main-cor-cu}. 
    In the symplectic setting,
    Avila and Viana \cite{AV10} show that 
    the $su$-integrability implies a volume-preserving conjugacy between $g$ and $L$ by using the Invariance Principle. 
    Corollary \ref{main-cor-Jac-rigidity} displays a similar phenomenon
    if one replaces the equivalent condition of semi-conjugacy by the special property. 
    We do not use the KAM theory and the Invariance Principle in our proofs.


    \vspace{8pt}\noindent
    \textbf{Organization of this paper.}
    The rest of the paper will be organized as follows:
        In Section \ref{sec-pre}, we collect some definitions and facts about DA endomorphisms on $\TT^2$, 
        and show directly that the existence of the semi-conjugacy $h$ on $\TT^2$ implies the integrability of the center. 
        Meanwhile, we describe the structure of $h$-injective points for further preparation.
    Next we prove Theorem \ref{main-thm-scu} in Section \ref{sec-special-to-semi} 
    by showing the existence of a semi-conjugacy which commutes with deck transformations on $\RR^2$. 
    The rigidity results implied by the existence of the semi-conjugacy on $\TT^2$ will be obtained in two steps.
        In Section \ref{sec-small-LyaExp}, 
        we analyze the periodic data inside the closure of $h$-injective points set 
        for both cases of $sc$-DA and $cu$-DA in an almost unified manner, 
        see Proposition \ref{4 prop s-rigidity} for the main rigidity of Lyapunov exponents.
    For the case of $sc$-DA, 
    we continue to complete the proofs of Theorem \ref{main-thm-sc} 
    in Section \ref{sec-DA-sc}, 
    while we also give the counterexamples mentioned in Proposition \ref{main-prop} and Remark \ref{1 rmk sc}.
        Later in Section \ref{sec-DA-cu}, 
        we present the proofs of Theorem \ref{main-thm-cu} and Corollary \ref{main-cor-cu} for the $cu$-DA case 
        by showing the existence of negative center Lyapunov exponents.
    Finally, we prove Corollary \ref{main-cor-scvp} and Corollary \ref{main-cor-Jac-rigidity} 
    in Section \ref{sec-volume-preserving} 
    with extended discussions on the volume-preserving systems.

    


\section{Preliminaries}\label{sec-pre}

\subsection{Partial hyperbolicity}\label{subsec ph}

    It is natural for endomorphisms to define partial hyperbolicity in terms of cone families.
    Let $f$ be an endomorphism on a compact Riemannian manifold $M$, for a continuous bundle splitting $TM=E_1\oplus E_2$ and a constant $\alpha>0$, 
    there are two related cone-fields $\mcc^{E_1}_\alpha$ and $\mcc^{E_2}_\alpha$ defined by
     \begin{align*}
        \mcc^{E_1}_\alpha
        &=\bigcup\limits_{x\in M}\mcc^{E_1}_\alpha(x)
        =\bigcup\limits_{x\in M}\big\{ v\in T_xM~\big|~
        v=(v^{E_1},v^{E_2}),|v^{E_2}|\leqslant\alpha|v^{E_1}|\big\},\\
        \mcc^{E_2}_\alpha
        &=\bigcup\limits_{x\in M}\mcc^{E_2}_\alpha(x)
        =\bigcup\limits_{x\in M}\big\{ v\in T_xM~\big|~
        v=(v^{E_1},v^{E_2}),|v^{E_1}|\leqslant\alpha|v^{E_2}|\big\}.
    \end{align*}

    {\color{black}
    A continuous splitting (may not be $Df$-invariant) $TM={E_1}\oplus {E_2}$ is called a \textit{dominated splitting} of $f$ denoted by $TM={E_1}\oplus_{<} {E_2}$, if there are constants $\alpha>0$ and $k\in\NN$
	such that
    \begin{align}
        Df^{k}(\mcc^{E_2}_\alpha)\subseteq {\rm int}\big(\mcc^{E_2}_{\alpha}\big), \label{eq. invariant cone}
    \end{align}
    where the interior of a cone ${\rm int}\big(\mcc^{E_2}_{\alpha}(x)\big)=\big\{ v\in T_xM~\big|~
        v=(v^{E_1},v^{E_2}),|v^{E_1}|<\alpha|v^{E_2}|\big\}\cup \{0_x\}$. 
    Note that \eqref{eq. invariant cone} also allows us to assume that $E_1$ is $Df$-invariant 
    (see for example \cite{CP15,HH21} or \eqref{eq. domination bundle}).
        
    Moreover, an endomorphism $f$ is called \emph{partially hyperbolic} if 
	there exist a dominated splitting $TM={E_1}\oplus_< {E_2}$ 
    and constants $\alpha>0$ and $k\in\NN$ such that
        \begin{align}\label{PH-conditions}
        |Df^k(v^{E_1})|<\frac{1}{2}
        \quad {\rm or} \quad
        |Df^k(v^{E_2})|>2,
        \tag{$\blacklozenge$}
        \end{align}
    for any $x\in M$ and unit vectors $v^{E_1}\in E_1(x)$ and $v^{E_2}\in \mcc^{E_2}_\alpha(x)$. }  
    In particular, there are two cases:
    \begin{itemize}
        \item We call $f$ an $sc$-partially hyperbolic endomorphism 
        if the first inequality of \eqref{PH-conditions} holds. 
        In this case, the cone-fields $\mcc^{E_1}_\alpha$ and $\mcc^{E_2}_\alpha$ will be called the \emph{stable cone-field} and \emph{center cone-field}, respectively.
        
        \item We call $f$ a $cu$-partially hyperbolic endomorphism 
        if the second inequality of \eqref{PH-conditions} holds. 
        In this case, the cone-fields $\mcc^{E_1}_\alpha$ and $\mcc^{E_2}_\alpha$ will be called the \emph{center cone-field} and \emph{unstable cone-field}, respectively.
    \end{itemize}

\begin{rmk}
    There are some remarks on the definition of partially hyperbolic endomorphisms involving the cone-field language. We refer the readers to \cite[Section 2]{CP15} for more details.{\color{black}
    \begin{enumerate}
        \item The domination \eqref{eq. invariant cone} is equivalent to the existence of constants $\alpha>0$ and $k\in\NN$ satisfying
    \begin{align}
        |Df^k(v^{E_1})|< \frac{1}{2}|Df^k(v^{E_2})|,\label{eq. domination}
    \end{align}
        for all $x\in M$ and unit vectors $v^{E_1}\in{E_1}(x)$ and $v^{E_2}\in\mcc^{E_2}_\alpha(x)$. 
        
        \item As usual, one can also define the partial hyperbolicity by choosing constants $C>1$ and $0<\lambda<1$ such that the domination 
        \[|Df^n(v^{E_1})|\leqslant C\lambda^n |Df^n(v^{E_2})|,\quad \forall\ n\in \NN,\]
        and the hyperbolicity
        \[ |Df^n(v^{E_1})|\leqslant C\lambda^n \quad {\rm or} \quad |Df^n(v^{E_2})|\geqslant C^{-1}\lambda^{-n},\quad \forall\ n\in \NN,   \]
        hold for any $x\in M$ and unit vectors $v^{E_1}\in{E_1}(x)$ and $v^{E_2}\in \mcc^{E_2}_\alpha(x)$.
        This definition is equivalent to inequalities of \eqref{eq. domination} and\eqref{PH-conditions}, up to taking $k$ via $(C,\lambda)$ or vice versa. 
       
        \item The definition of partial hyperbolicity is independent of the metric in the sense that another equivalent metric will just lead to different constants $\alpha$ (the size of cone-fields) and $k\in\mathbb{N}$.
    \end{enumerate}  }
\end{rmk}

    Furthermore, we call $f$ an \emph{absolutely partially hyperbolic endomorphism} 
        {\color{black} if it is partially hyperbolic and \eqref{eq. domination} holds 
        for any $x,y\in M$ and unit vectors $v^{E_1}\in E_1(x)$ and $v^{E_2}\in\mcc^{E_2}_\alpha(y)$.}
    In particular, we call $f$ an \emph{expanding endomorphism} 
        {\color{black} if it has trivial bundle $E_1$ in the dominated splitting,
        and there exists $k\in\NN$ such that $|Df^k(v)|>1$, for any $x\in M$ and unit vector $v\in T_xM$.}


\subsubsection{Special endomorphisms}\label{subsub sec se}
    Notice that there exists a $Df$-invariant bundle $E^1_f\subset \mcc^{E_1}_\alpha$  given by
     \begin{align}
        E^1_f(x)=\bigcap_{i\in\NN}Df_0^{-i}\big(\mcc^{E_1}_\alpha(f^i(x))\big),
        \label{eq. domination bundle}
     \end{align} 
    where $f^{-i}_0(f^i(x))=x$.  
    Similarly, for a given negative orbit $\Tilde{x}=(\dots,x_{-2},x_{-1},x)$ of $x\in M$, 
    there is a $Df$-invariant  direction (along this orbit $\Tilde{x}$) called \emph{dominating direction}
    \begin{align}\
        E^2_f(x,\Tilde{x})=\bigcap_{i\in\NN}Df^i\big(\mcc^{E_2}_\alpha(x_{-i})\big). \label{eq. dominating direction}
    \end{align}
    We refer to \cite{Pr76} for more properties, such as the continuity of directions with respect to orbits. 

    It should be noticed that there are generally cone-fields $\mcc^{E_2}_\alpha(x)$ 
    rather than $Df$-invariant bundles that are independent of the negative orbits, 
    unless the partially hyperbolic endomorphism is special. 
    
    Precisely, a partially hyperbolic endomorphism $f:M\to M$ is called \emph{special}, 
    if there exists a $Df$-invariant partially hyperbolic splitting 
    \begin{align}
        TM= E_f^1\oplus_< E_f^2,\label{eq. 1. special}
    \end{align}
    that is, the dominating direction of a point is independent of its negative orbits.


\subsubsection{Derived from Anosov endomorphisms}\label{subsubsec DA}
    If both of these inequalities in (\ref{PH-conditions}) hold, 
    $f$ will be an \emph{Anosov endomorphism}.
    A partially hyperbolic endomorphism $f$ is called 
	\emph{derived from Anosov} (abbr. \emph{DA endomorphism}) 
	if its linear part $f_*$ is an Anosov endomorphism.
	More precisely,  
    there are also the following cases corresponding to the two inequalities in (\ref{PH-conditions}):
    \begin{itemize}
       \item When the DA endomorphism $f$ is  $sc$-partially hyperbolic, we call $f$ \emph{$sc$-DA}. In this case, $f$ admits a stable cone-field which by \eqref{eq. domination bundle} actually gives a $Df$-invariant uniformly contracting subbundle that will be called the \emph{stable bundle} and denoted by $E^s_f$.
       \item  When the DA endomorphism $f$ is  $cu$-partially hyperbolic, we call $f$ \emph{$cu$-DA}. In this case, $f$ admits an unstable cone-field, and by \eqref{eq. domination bundle} the center cone-field gives a $Df$-invariant subbundle that will be called the \emph{center bundle} and denoted by $E^c_f$.
    \end{itemize}

\subsection{Semi-conjugacy}\label{subsec semi-conj}
    Throughout this subsection, we will use the following notations.
    Let $f:\mtt$ be a $C^1$-smooth DA endomorphism and $A:\mtt$ be its linearization. 
    Let $F:\mrr$ be a lift of $f$ via $\pi:\mrt$.  
    In short, we still denote the lift of $A:\mtt$ via $\pi$ by $A:\mrr$. 
    And we denote the stable/unstable bundles and foliations of $A$ by $L^{s/u}$ and $\mcl^{s/u}$ on $\TT^2$, $\widetilde{L}^{s/u}$ and $\tildeL^{s/u}$ on $\RR^2$, respectively.

    Recall that  $\mcl^{s/u}$  is minimal, 
    namely, for every $x\in\TT^2$, the leaf $\mcl^{s/u}(x)$ is dense in $\TT^2$. 
    Then there is the following property. 

\begin{prop}\label{2 prop minimal foliation}
    For any $x\in\RR^2$ and $k\in \NN$, the set
    \[\bigcup_{n\in A^k\ZZ^2}\tildeL^{\sigma}(x+n)\]
    is dense in $\RR^2$, where $\sigma=s,u$.
\end{prop}
\begin{proof}
    It suffices to verify for $x=0$ 
    since the foliation $\tildeL^{\sigma}$ consists of translations of 
    the leaf $\tildeL^{\sigma}(0)$, for $\sigma=s,u$. 
    Note that the lifted map of the linearization $A:\mrr$ is an Anosov diffeomorphism with no trivial bundles in the partially hyperbolic splitting.
    Thus, for a given $k\in\NN$, 
    one can get $A^k\{\tildeL^{\sigma}(0)+\ZZ^2\}$ is dense in $\RR^2$.
    Then it follows that the set $\{\tildeL^{\sigma}(0+A^k\ZZ^2)\}$ is dense in $\RR^2$.
\end{proof}

    Focusing on the $A$-preimage set, there is the following useful estimate.

\begin{lem}[\cite{AGGS23}, Proposition 2.10]\label{2 lem preimage dense linear}
    There exist $C_A>0$ and $0<\gamma<1$ such that for any $x\in\TT^2$, $k\in\NN$, 
    the $k$-preimage set of $x$ 
        \[ P_k(x):=\{\ y\in\TT^2\ |\ A^k(y)=x\ \} \]
    is $C_A\cdot\gamma^k$-dense in $\TT^2$. 
    Here one can take $\gamma=|{\rm det}(A)|^{-\frac{1}{2}}$.
\end{lem}

\begin{rmk}
    When $|{\rm det}(A)|=1$, the conclusion obviously holds for large constant $C_A$. 
    The interesting case of this lemma is $|{\rm det}(A)|>1$, 
    which is exactly applicable to the linearization of DA endomorphism 
    that we concern in this paper.
\end{rmk}

\subsubsection{The semi-conjugacy and foliations on the universal cover}
    There are some basic facts for the dynamics on the universal cover, 
    especially, some properties related to the semi-conjugacy.

    We first include the following classical result for toral endomorphisms.
\begin{prop}[\cite{AH94}, Theorem 8.2.1]\label{2 prop semi-conj in R2}
    Let $F$ and $A$ be as in Section \ref{subsec semi-conj}.
    There exists a unique surjection $H:\mrr$ satisfying the following: 
    \begin{enumerate}
        \item $H\circ F=A\circ H$;
        \item $H$ is uniformly continuous;
        \item There exists $C_H>0$ such that $\| H-{\rm Id}_{\RR^2} \|_{C^0}<C_H$.
    \end{enumerate}
\end{prop}

    Then we collect useful properties for the semi-conjugacy on $\RR^2$. 
    These properties of DA endomorphisms 
    are adaptations of the ones of DA diffeomorphisms, Anosov endomorphisms 
    and $cu$-DA endomorphisms in \cite{Ha13,H13,HH21,HS21,AGGS23}. 
    For completeness, 
    we give proofs briefly based on the original ones.

    We start by introducing a few notations for the following results. 
    For the map $T_n:\RR^2\to \RR^2$ with $T_n(x)=x+n$,
    we say $H$ \emph{commutes with the deck transformations}
    if \[H\circ T_n=T_n\circ H,  \quad\forall\ n\in\ZZ^2.\]
    Note that the semi-conjugacy in Proposition \ref{2 prop semi-conj in R2} 
    may not commute with the deck transformations, 
    otherwise it can descend to $\TT^2$ and induce a semi-conjugacy between $f$ and $A$. 
    
    The map $T_{-n}\circ H\circ T_n$ still keeps some good properties, 
    for instance, it preserves the stable foliation and has the following asymptotic behavior.

\begin{prop}[\cite{AGGS23}]\label{2 prop approach property of H}
    Let $H$ and $C_H$ be as in Proposition \ref{2 prop semi-conj in R2}.
    Then 
    \begin{enumerate}
        \item For any $x\in\RR^2$ and $n\in\ZZ^2$, $H(x+n)-n\in \tildeL^s\big(H(x)\big)$.
        \item For any $x\in\RR^2$, $k\in\NN$, and $n^*\in A^k\ZZ^2$, 
        one has 
            \[
            \Big| H(x+n^*)-H(x)-n^* \Big|<2C_H\cdot \|A|_{L^s}\|^k.
            \]
    \end{enumerate}
\end{prop}

\begin{proof}
    By Proposition \ref{2 prop semi-conj in R2}, 
    one has $\| H-{\rm Id}_{\RR^2} \|_{C^0}<C_H$.
    Then, for any $k\in\NN$,
    \begin{align*}
        &d\Big(A^k(H(x+n)-n),A^k\circ H(x)\Big)\\
        &=d\Big(H(F^k(x+n))-A^k(n),H\circ F^k(x)\Big)\\
        &=d\Big(H\big(F^k(x)+A^k(n)\big)-A^k(n),H\circ F^k(x)\Big)\\
        &\leqslant 
        d\Big(H\big(F^k(x)+A^k(n)\big)-A^k(n),F^k(x)\Big)+d\Big(F^k(x),H\circ F^k(x)\Big)\\
        &<2C_H,
    \end{align*}
    where $d$ is the natural Euclidean distance on $\RR^2$.
    This uniform upper bound indicates that 
    there is no ``unstable component'' for the difference between $H(x+n)-n$ and $H(x)$, 
    that is, they are in the same stable leaf for any $n\in\ZZ^2$.
    
    Moreover, for any $n^*\in A^k\ZZ^2$, one has
        \[F^{-k}(x+n^*)-F^{-k}(x)=A^{-k}(n^*)\in\ZZ^2,\]
    so $H \big(F^{-k}(x)+A^{-k}(n^*)\big)$ and $H\circ F^{-k}(x)+A^{-k}(n^*)$ are in the same stable leaf.
    Note that
    \begin{align*}
        H(x+n^*)&=A^k\circ H\circ F^{-k}(x+n^*)=A^k\circ H \big(F^{-k}(x)+A^{-k}(n^*)\big),\\
        H(x)+n^*&=A^k\circ \big( H\circ F^{-k}(x)+A^{-k}(n^*)\big).
    \end{align*} 
    Hence, one can get 
        \[
        \Big| H(x+n^*)-H(x)-n^* \Big|<2C_H\cdot \|A|_{L^s}\|^k.
        \]
\end{proof}

    Note that  $f$ is a partially hyperbolic endomorphism if and only if $F$ is a partially hyperbolic diffeomorphism, by a similar argument to the Anosov case in \cite{ManePugh1975}.
    Hence, the lifted diffeomorphism $F:\RR^2\to \RR^2$ 
    admits an invariant partially hyperbolic splitting  
    \begin{align}
       T\RR^2=E^1\oplus_< E^2.\label{eq. Fsplitting}
    \end{align}
    One can also get the invariant splitting \eqref{eq. Fsplitting} by lifting the cone-fields of $f$ to $T\RR^2$ and using the formulas similar to  \eqref{eq. dominating direction} and \eqref{eq. domination bundle}.
    
    If $f$ is $sc$-DA, then $DF$ is contracting on $E^1$, 
    meanwhile, $E^1$ is uniquely integrable. 
    We denote the integral foliation on $\RR^2$ by  $\mcf_F^s$, the stable foliation. 
    Similarly, if $f$ is $cu$-DA, we can define the unstable foliation $\mcf_F^u\subset \RR^2$ 
    which is tangent to the expanding bundle $E^2$. 
    We call $\mcf_F^{s}$ and $\mcf_F^{u}$ the \emph{hyperbolic foliations}, 
    which are the stable and unstable foliations in the cases of $sc$-DA and $cu$-DA, respectively. 
    We refer to \cite{Pesinbook04} for the unique integrability of these hyperbolic foliations on $\RR^2$. 
    
    Two cases are a little different for the situation on $\TT^2$. 
    The stable foliation $\mcf^s_f$ of $sc$-DA always exists and is the projection of $\mcf^s_F$.
    However, the unstable foliation $\mcf^u_f$ for a $cu$-DA $f$ exists only if $f$ is special.
    
    Following the works of leaf conjugacy by Hammerlindl 
    (\cite{Ha13,H13}), 
    we know the lifted diffeomorphism $F:\RR^2\to \RR^2$ is dynamically coherent 
    (without a priori quasi-isometric property).
    Moreover, we have some common propositions for the cases of $sc$-DA and $cu$-DA as follows.

    \vspace{5pt}\noindent
    \textbf{Notation.}
    The following notations 
    include upper indexes $*^s,*^2$ for $sc$-DA and $*^1,*^u$ for $cu$-DA when we emphasize the hyperbolic part, respectively, 
    and only include upper indices $*^1,*^2$ for DA systems when we only point out the domination on $\RR^2$.

\begin{prop}[\cite{H13,HH21}]\label{2 prop foliation on R2}
    The bundle $E^i$ is uniquely integrable to a foliation $\mcf^i$ for $i=1,2$. 
    Denote the center foliation by $\mcf^c_F$. Then the following properties hold.
    \begin{enumerate}
        \item $\mcf^1$ and $\mcf^2$ admit the Global Product Structure, i.e., 
        for any $x,y\in\RR^2$, the manifold $\mcf^1(x)$ intersects with $\mcf^2(y)$ exactly once;
        \item The  foliation $\mcf^i$ is quasi-isometric, i.e., 
        there exist $C_1,C_2>0$ such that \[d_{\mcf^i}(x,y)<C_1|x-y|+C_2, \quad \forall\ x\in\RR^2\  {\rm and}\ y\in \mcf^i(x). \] 
   \end{enumerate}
\end{prop}

\begin{proof}
    By \cite{HH21}, this proposition holds for the $cu$-DA case. 
    Specifically, the dynamical coherence follows from \cite[Theorem A]{HH21}, 
    the Global Product Structure is given by \cite[Proposition 2.10]{HH21} 
    and the quasi-isometry of the foliations $\mcf^u_F$ and $\mcf^c_F$ are respectively given by \cite[Proposition 2.12(2)]{HH21} and \cite[Lemma 3.3]{HH21}. 
 
    Now we consider the $sc$-DA $f:\mtt$.
    The proofs for the dynamical coherence, the Global Product Structure, and the quasi-isometry of the foliation $\mcf^s_F$ are similar to the proofs in \cite{HH21}. 
    Here we briefly explain how it works for the $sc$-DA case. 
    Consider the inverse map $F^{-1}:\mrr$ of the lift of $sc$-DA $f$. 
    Note that $F^{-1}$ is partially hyperbolic diffeomorphism with invariant unstable bundle and is semi-conjugate to the hyperbolic diffeomorphism $A^{-1}:\mrr$. 
    Since the argument in \cite[Theorem A, Proposition 2.10 and Proposition 2.12(2)]{HH21} are completely in the universal cover $\RR^2$, 
    these properties hold for $F^{-1}$ and $\mcf^{\,c/u}_{F^{-1}}$. 
    Hence, we get that these properties also hold for $F$ and $\mcf^{\,s/c}_{F}$. 
    Then, it remains to prove the quasi-isometry of $\mcf^c_F$.

    The quasi-isometry of $\mcf^c_F$ in the $sc$-DA case is not as direct as above.   
    Note that the proof of \cite[Lemma 3.3]{HH21}, 
    i.e., the quasi-isometry of $\mcf^c_F$ in the $cu$-DA case, 
    needs $\mcf^c_F$ commuting with the $\ZZ^2$-actions. 
    This does not \emph{a priori} hold in the $sc$-DA case. 
    However, combining the argument in \cite{H13} and \cite{HH21}, $\mcf^c_F$ is also quasi-isometric in the $sc$-DA case. 
    Note that though the paper \cite{H13} deals with the $3$-nilmanifold case, the argument can be adapted for the $2$-torus. 
    We refer to the proof of \cite[Proposition 2.10]{HH21} where the authors of both \cite{H13, HH21} explain how the argument in \cite{H13} works for the $2$-torus as well.
   
    Now we give a brief proof of the quasi-isometry of $\mcf^c_F$ in the $sc$-DA case. 
    Denote by $\pi^{s/u}:\RR^2\to \tildeL^{s/u}(0)$ the projections along foliations $ \tildeL^{u/s}$, respectively. 
    The following properties are immediate from \cite{H13,HH21} to the $cu$-partially hyperbolic diffeomorphism $F^{-1}$.
    \begin{enumerate}
        \item (\cite[Lemma 2.3]{HH21}, \cite[Lemma 4.10]{H13}) 
        There exists $R>0$ such that for any center segment $J^c$, 
        the interval $\pi^s(J^c)$ has length at most $R$. 
        (This can be done by using the argument in \cite[Lemma 2.3]{HH21} or \cite[Lemma 4.10]{H13} to the $cu$-partially hyperbolic diffeomorphism $F^{-1}$.)
        
        \item (\cite[Corollary 2.5]{HH21}) There exists $R>0$ such that for any stable segment $J^s$, 
        the interval $\pi^u(J^s)$ has length at most $R$. 
        (This can be done by using the argument in \cite[Corollary 2.5]{HH21} to the $cu$-partially hyperbolic diffeomorphism $F^{-1}$.)
        
        \item (\cite[Lemma 2.9]{HH21}, \cite[Lemma 4.7]{H13}) 
        For any $C>0$, there exists $l>0$ such that any center segment $J^c$ of length greater than $l$ 
        is such that $\pi^u(J^c)$ has length greater than $C$. 
        (This can be proved by the same proof of \cite[Proposition 2.8 and Lemma 2.9]{HH21}, where one just needs to replace \cite[Corollary 2.5]{HH21} by the property given by the above item.) 
   \end{enumerate}

    Then, by modifying the proof of \cite[Lemma 4.13]{H13}, 
    we can get the following. 
\begin{claim}\label{2 claim in 2.6}
        For any $M>0$, there exists $l>0$ such that any center segment $J^c$ of length greater than $l$ has endpoints $p$ and $q$ with $|\pi^u(p)-\pi^u(q)|>M$.
\end{claim}
    \begin{proof}
       Let $J^c$ be a center curve with length great enough and $\alpha:[0,1]\to \RR^2$  be its parametrization. 
        Let $R$ be as in the items (1) and (2) above. 
        Taking $C=M+2R$ and $l$ in the item (3) above, 
        we get that there are $s,t\in[0,1]$ satisfying
        \begin{align}
             \pi^u\alpha(t)-\pi^u\alpha(s)>C=M+2R.\label{eq. claim 2.7.1}
        \end{align}
        Without loss of generality, we assume $s<t$
        (see FIGURE \ref{quasi-isometry}).  
        Here we identify $\tildeL^u(0)$ with $\RR$.
        
    \begin{figure}[htbp]
		\centering
		\includegraphics[width=8.6cm]{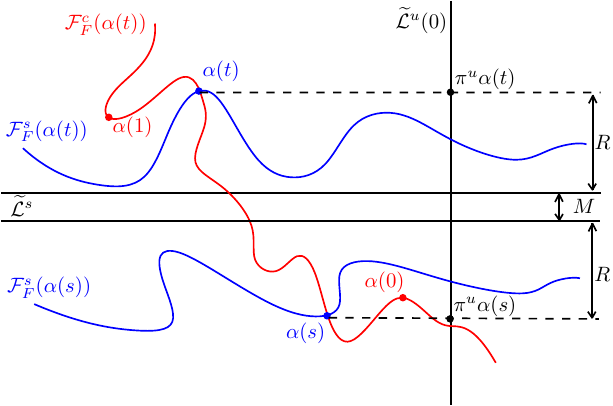}
		\caption{The projections of endpoints of $J^c$ on $\tildeL^u$ have distance bigger than $M$.}	
        \label{quasi-isometry}
	\end{figure}

    By the above item (2), for any $x\in\RR^2$, the complement $\RR^2 \setminus \mcf^s_F(x)$
    has connected components $B_-(x)$ and $B_+(x)$ satisfying:
    \begin{align*}
        H_R^-(x):=\big\{y\in\RR^2\ | \ \pi^u(y)<\pi^u(x)-R \big\} \subset B_-(x),\\
        H_R^+(x):=\big\{y\in\RR^2\ | \ \pi^u(y)>\pi^u(x)+R \big\} \subset B_+(x).
    \end{align*}
 
    Combining the item (2) with the Global Product Structure, 
    we get that the curve $\alpha((t,1])$ is located in the component $B_+(\alpha(t))$ of $\RR^2 \setminus \mcf^s_F(\alpha(t))$, 
    namely, $\alpha((t,1])\subset B_+(\alpha(t))$. 
    Hence, we have $\alpha((t,1])\subset \RR^2\setminus H^{-}_R(\alpha(t))$ which implies $\pi^u\alpha(1)-\pi^u\alpha(t)>-R$. 
    Similarly, we have $\alpha([0,s))\subset \RR^2\setminus H^{+}_R(\alpha(s))$ which implies $\pi^u\alpha(s)-\pi^u\alpha(0)>-R$. 
    Then by \eqref{eq. claim 2.7.1}, we get 
    \begin{align*}
    \pi^u\alpha(1)-\pi^u\alpha(0)
    &=
    \big( \pi^u\alpha(t)-\pi^u\alpha(s) \big)+
    \big( \pi^u\alpha(1)-\pi^u\alpha(t) \big)+
    \big( \pi^u\alpha(s)-\pi^u\alpha(0) \big)\\
    &>
    (M+2R)+(-R)+(-R)=M.    
    \end{align*}

    This ends the proof of Claim \ref{2 claim in 2.6}.
    \end{proof}

    Now we show that Claim \ref{2 claim in 2.6} is sufficient to get the quasi-isometry of $\mcf^c_F$. 
    Indeed, let constants $M$ and $l$ be as in Claim \ref{2 claim in 2.6}, 
    and $J^c$ be a center curve with endpoints $x,y$ and with length $L=kl+r$ 
    where $k\in\NN$ and $0\leqslant r<l$. 
    Then by Claim \ref{2 claim in 2.6}, $|\pi^u(x)-\pi^u(y)|\geqslant kM$. 
    Thus, there is a constant $\beta>0$ determined by the angle between $\tildeL^s$ and $\tildeL^u$ such that $d(x,y)\geqslant k\beta M$. 
    Then for $L\geqslant 2l$, one has $L-r\geqslant\frac{L}{2}$ and 
        \[d(x,y)\geqslant k\beta M
        =\frac{M}{l}\beta kl
        = \frac{M}{l}\beta (L-r)
        \geqslant\frac{M}{2l}\beta L.\]
    Hence, $d_{\mcf^c_F}(x,y)=L\leqslant\frac{2l}{M\beta}d(x,y)$ for $L\geqslant 2l$. 
    In particular, for any $L>0$, one has 
    \[ d_{\mcf^c_F}(x,y)=L\leqslant\frac{2l}{M\beta}d(x,y)+2l.\]
    This concludes the quasi-isometry of $\mcf^c_F$,
    and ends the proof of Proposition \ref{2 prop foliation on R2}.
\end{proof}

    Now we can get more information about the semi-conjugacy $H$ 
    restricted to the leaf of these foliations 
    with the nice properties established above.

\begin{prop}[\cite{HH21,HS21}]\label{2-3 prop foliation on R2}
    Let $H$ be as in Proposition \ref{2 prop semi-conj in R2}. Then 
    \begin{enumerate}
        \item The semi-conjugacy $H$ 
        maps $\mcf^1(x)$ and $\mcf^2(x)$ onto $\tildeL^s(H(x))$ and $\tildeL^u(H(x))$, respectively.
        \item The semi-conjugacy $H$ 
        restricted to a hyperbolic leaf 
        \[  H:\mcf_F^{s/u}(x) \to \tildeL^{s/u}(H(x)) \] 
        is a homeomorphism, respectively.
        \item For any $x\in\RR^2$, the preimage $H^{-1}(x)$ consists of 
        either a single point or a compact segment inside a center leaf. 
        Moreover, there exists a uniform upper bound for the length of $H^{-1}(x)$ for every $x\in\RR^2$.
   \end{enumerate}
\end{prop}
\begin{proof}
    It essentially follows the proof dealing with the $cu$-DA case in \cite{HH21}, 
    while the details in \cite{HS21} for DA diffeomorphisms are strongly related as well. 
    Just as in the last proposition, we briefly show the proof for the lifted $sc$-DA map.
    
    By the fact that the semi-conjugacy $H$ is close to the identity and uniformly continuous, 
    one can get $H$ restricted to the hyperbolic leaf 
        \[ H:\mcf_F^{s}(x) \to \tildeL^{s}(H(x)), \] 
    is a homeomorphism.
    Then, according to the Global Product Structure, 
    it follows that $H$ preserves the center leaf, 
    and maps two distinct center leaves to the different lines in $\tildeL^{u}$, 
    namely, $H$ is a bijection on the space of center leaves.
    Thus, $H$ is surjective when restricted to center leaves.
    
    Let us equip ``$\prec$'' the order 
    corresponding to the relative orientation on each leaf of $\mcf^c_F$. 
    Without loss of generality, we assume that $F$ preserves this orientation. 
    Now we prove that the semi-conjugacy $H$ is monotonic along center leaves.
    If not, taking $x\prec y\prec z$ along a center leaf $J^c$ with $H(x)=H(z)\neq H(y)$, 
    the forward iterates of $A$ will expand the distances 
    between $H(y)$ and $H(x)$, $H(z)$, respectively. 
    By the fact that $H$ is a finite distance from the identity, one has 
    \[
        d_{\mcf_F^c}(F^n(x),F^n(y))\to \infty 
        \quand  
        d_{\mcf_F^c}(F^n(y),F^n(z))\to \infty .
    \]
    And the order $F^n(x)\prec F^n(y)\prec F^n(z)$ implies
        \[ d_{\mcf_F^c}(F^n(x),F^n(z))\to \infty.\]
    Using facts that center leaves are quasi-isometrically embedded and $H$ is uniformly continuous, 
        \[ |A^nH(x)-A^nH(z)|\to \infty,\] 
    which contradicts $H(x)=H(z)$. 
    
    Based on the monotonicity of $H$ along center leaves, 
    for any $x\in\RR^2$, the non-empty set $H^{-1}(x)$ 
    which is a compact subset of some center leaf $J^c$ will be connected, 
    hence, it consists of either a single point or a compact segment inside the center leaf $J^c$. 
    Denote the two endpoints of $H^{-1}(x)$ by $x_1$ and $x_2$. 
    Since $\mcf^c_F$ is quasi-isometric (Proposition \ref{2 prop foliation on R2}), 
    one has 
        \[  d_{\mcf^c_F}(x_1,x_2)\leqslant C_1|x_1-x_2|+C_2. \]
    It follows from  $\|H-{\rm Id}_{\RR^2}\|_{C^0}<C_H$ that 
        \[  d_{\mcf^c_F}(x_1,x_2)\leqslant 2C_HC_1+C_2.       \] 
    Then we obtain $C\triangleq2C_HC_1+C_2$, the uniform upper bound of the length of $H^{-1}(x)$.
\end{proof}

    As a direct corollary, we can project the properties from the ones on the universal cover, 
    when assuming the existence of a semi-conjugacy on $\TT^2$.

\begin{cor}\label{2 cor semi-conj. Fc}
    Let $f:\mtt$ be a $C^1$-smooth DA endomorphism 
    and semi-conjugate to an Anosov endomorphism $A:\mtt$ via $h:\mtt$. 
    Then $f$  has the center bundle $E^c_f\subset T\TT^2$ 
    which is uniquely integrable to the center foliation $\mcf^c_f$. 
    Moreover, for any $x\in\TT^2$, the preimage $h^{-1}(x)$ consists of 
    either a single point or a compact segment with a uniformly upper bounded length inside the center leaf.
\end{cor}
\begin{proof}
    It suffices to prove the existence of the center bundle $E^c_f$. 
    If so, the unique integrability of $E^c_f$ and the property of $h^{-1}$ follow immediately
    from Proposition \ref{2 prop foliation on R2} and Proposition \ref{2-3 prop foliation on R2}.

    When $f$ is $cu$-DA, the existence of the $Df$-invariant center bundle $E^c_f$ is guaranteed 
    by iterating backward the center cone-field along positive orbits. 
    
    Now let $f$ be $sc$-DA. 
    Since $f$ is semi-conjugate to $A$, 
    the lifted semi-conjugacy $H:\RR^2\to\RR^2$ satisfies $H(x+n)=H(x)+n$, for any $x\in\RR^2$ and $n\in\ZZ^2$. 
    Hence, one has
    \begin{align*}
        H\big( \mcf^c_F(x+n)  \big) 
        &= \tildeL^{u}\big(H(x+n)\big)=\tildeL^{u}\big(H(x)+n\big)\\
        &=\tildeL^{u}\big(H(x)\big)+n=H\big(\mcf^c_F(x)\big)+n
        =H\big(\mcf^c_F(x)+n\big).
    \end{align*}
    Since $H$ is homeomorphic along $\mcf^s_F$, 
    it follows from the Global Product Structure
    that \[\mcf^c_F(x+n) =\mcf^c_F(x)+n,\] for any $x\in\RR^2$ and $n\in\ZZ^2$. 
    Note that the center direction of $f$ is continuous with respect to $f$-orbits, 
    the natural projection of the set of $F$-orbits is dense in $\TT^2_f$. 
    Hence, one can get the existence of the $Df$-invariant center bundle $E^c_f$. 
\end{proof}

\subsubsection{The set of injective points}

    Let $H$ be as in Proposition \ref{2 prop semi-conj in R2}, 
    we now focus on preparing some technical properties for the set of injective points 
    associated with the surjection $H$.
    
    Throughout this subsection, 
    we assume that $f:\mtt$ is semi-conjugate to $A:\mtt$.
    Then $H$ commutes with the deck transformations.
    Denote by \[h:=\pi\circ H\circ \pi^{-1}:\mtt\]
    the semi-conjugacy between $f$ and $A$. 
    We define \[\Gamma:=\big\{x\in\RR^2\ |\ H^{-1}\circ H(x)=\{x\} \big\},\] 
    and denote by $\Bar{\Gamma}$ the closure of $\Gamma$ 
    with respect to the topology of $\RR^2$. 
    
    By Corollary \ref{2 cor semi-conj. Fc}, 
    the center foliation $\mcf^c_f$ exists 
    and can be equipped with an order ``$\prec$'' as well. 
    Moreover, for any $x\in\TT^2$, 
    we can define the endpoints of the set $h^{-1}\circ h(x)$ by $x^-$ and $x^+$ 
    such that $x^-\preceq x^+$, that is,  $x^-\prec x^+$ or $x^-=x^+$. 
    Here $x^-=x^+$ if and only if $h^{-1}\circ h(x)=\{x\}$.

    Following the spirit of the work \cite{HS21}, 
    we describe the structure of $\Bar{\Gamma}$ as follows. 

\begin{prop}[\cite{HS21}]\label{2 prop injective point set}
    Let $H$ be as in Proposition \ref{2 prop semi-conj in R2}. 
    Assume $H$ commutes with the deck transformations.
    Then one has the following.
    \begin{enumerate}
        \item $\Bar{\Gamma}$ is $F^{\pm 1}$-invariant.
        
        \item $\Bar{\Gamma}+\ZZ^2=\Bar{\Gamma}$. 
            In particular, the set $\Lambda:=\pi(\bar{\Gamma})$ is $f^{\pm 1}$-invariant 
        and 
        \begin{align}
            \Lambda
            =\overline{\big\{x\in\TT^2~ |~  h^{-1}\circ h(x)=\{x\}\big\}} 
            =\bigcup_{x\in\TT^2}\{x^+, x^-\}.    \label{eq. 2. Lambda}
        \end{align}
        \item $\Bar{\Gamma}$ is $\mcf_F^{s/u}$-saturated and $f|_{\Lambda}$ is special.  
            In particular, $\Lambda$ is $\mcf_f^{s}$-saturated when $f$ is $sc$-DA, 
            and $\Lambda$ is $\mcf_f^{u}$-saturated when $f$ is $cu$-DA.
        
        \item Given $x\in \Lambda$, if $x=x^{\sigma}$, 
            then $y=y^{\sigma}$ for any $y\in \mcf_f^{s/u}(x)\subset \Lambda$ and $\sigma=+,-$.
        
        \item $\overline{{\rm Per}(f|_{\Lambda})}=\Lambda$.
    \end{enumerate}
\end{prop}
\begin{proof}
    For the first item, 
    let $x\in\Gamma$. Then 
        \[  H^{-1}\circ H\big(F^{\pm 1} (x)\big)
        =  H^{-1}\circ A^{\pm 1}\circ  H(x)= F^{\pm 1} \circ H^{-1}\circ H(x)
        = \{F^{\pm 1} (x) \} .\]
    Hence, $F^{\pm 1}(x)\in \Gamma$. 
    It follows from the continuity that $\Bar{\Gamma}$ is $F^{\pm 1}$-invariant.

    For the second item, 
    let $x\in\Gamma$ and $n\in \ZZ^2$. 
    We claim that $H^{-1}\circ H(x+n)=\{x+n\}$. 
    Indeed, let $y\in H^{-1}\circ H(x+n)$, 
    then $H(y)=H(x+n)=H(x)+n$, that is, $H(y-n)=H(x)$. 
    Since $x\in\Gamma$, one has $y=x+n$. 
    Hence, $\Gamma+\ZZ^2=\Gamma$ and $\Bar{\Gamma}+\ZZ^2=\Bar{\Gamma}$. 

    Now the set $\Lambda:=\pi(\bar{\Gamma})$ satisfies  $\pi^{-1}(\Lambda)=\Bar{\Gamma}$  
    and is $f^{\pm 1}$-invariant by the first item. 
    Indeed, let $x\in \Lambda$. Then as a set $f^{\pm 1}(x)=\pi\circ F^{\pm 1} \circ \pi^{-1}(x)$.  
    Since $ \pi^{-1}(x)\subset \Bar{\Gamma}$ and $\Bar{\Gamma}$ is $F^{\pm 1}$-invariant, 
    one has $F^{\pm 1} \circ \pi^{-1}(x)\subset \Bar{\Gamma}$ and $f^{\pm 1}(x)\subset \Lambda$. 

    Note that $\Lambda=\pi(\Bar{\Gamma})=\overline{\pi(\Gamma)}$ 
    since $\Bar{\Gamma}$ is $\ZZ^2$-periodic. 
    It is clear that 
    \begin{align}
       \pi(\Gamma)
       =\big\{x\in\TT^2~ |~  h^{-1}\circ h(x)
       =\{x\}\big\},    \label{eq. 2. Lambda set 1} 
    \end{align}
    since $H$ commutes with the deck transformations. 
    On the other hand, by Proposition \ref{2-3 prop foliation on R2},
    \begin{align}
        \bar{\Gamma}=
        \bigcup_{x\in\RR^2} \big\{ y\in \RR^2\ |\ y\  \text{is an endpoint of} \ H^{-1}\circ H(x)\big\}, 
        \label{eq. 2. Lambda set 2}
    \end{align}
    since for any $x\in \RR^2\setminus \Gamma$, every endpoint of $H^{-1}\circ H(x)$  
    can be approached by points in  $\Gamma$. 
    Combining \eqref{eq. 2. Lambda set 1} and \eqref{eq. 2. Lambda set 2}, one just has \eqref{eq. 2. Lambda}.  

    For the third item, 
    by Proposition \ref{2-3 prop foliation on R2}, 
    $\Bar{\Gamma}$ is $\mcf^{s/u}_F$-saturated. 
    Let us consider $f|_{\Lambda}$ for the cases of $sc$-DA and $cu$-DA, respectively.
\begin{itemize}
    \item In the $sc$-DA case, 
        it follows from Corollary \ref{2 cor semi-conj. Fc} that $f$ is special on the whole $\TT^2$. 
        And the foliation $\mcf^s_f$ exists and is the projection of $\mcf^s_F$. 
        Hence, $\Lambda$ is $\mcf^s_f$-saturated.
    \item In the $cu$-DA case, 
        let $x\in\Gamma$ and $n\in\ZZ^2$. Then 
        \begin{align*}
              H\big( \mcf^u_F(x+n)-n  \big) &=  H\big( \mcf^u_F(x+n) \big)-n \\
              &= \tildeL^u \big( H(x)+n \big)-n
              =  \tildeL^u \big( H(x)\big)=H\big(\mcf^u_F(x)\big). 
        \end{align*}
        Since $x\in\Gamma$, one has $\mcf^u_F(x+n)-n = \mcf^u_F(x)$ for any $n\in\ZZ^2$. 
        By the continuity of $\mcf^u_F$,  
            $\mcf^u_F(z+n) = \mcf^u_F(z)+n$ for any $z\in \bar\Gamma$ and $n\in\ZZ^2$.
        It follows that $E^u_F(z+n)=DT_n(E^u(z))$ for any $z\in \bar\Gamma$ and $n\in\ZZ^2$. 
        Recall that $T_n:\RR^2\to \RR^2$ is the deck transformation.
        Hence, \[T_{\Lambda}=E^c_f\oplus D\pi(E^u_F)\] is an invariant hyperbolic splitting,  
        $f|_{\Lambda}$ is special and $\Lambda$ is $\mcf^u_f$-saturated. 
\end{itemize}

    For the fourth item, 
    we just prove for $\sigma=+$ and $\mcf^u_f$, 
    since other cases are similar. Note that we do not know $f|_{\TT^2\setminus \Lambda}$ is special yet,
    so it will be convenient to deal with this case in $\RR^2$. 
    We equip the same orientation to $\mcf^c_F$. 
    Assume that there exist $x=x^+\in\bar\Gamma$ and $y\in \mcf^u_F(x)$ with $y\neq y^+$.
    Then one has $y\prec y^+$.  
    Let $z = \mcf^u_F(y^+)\cap \mcf^c_F(x^+)$, 
    one has $x^+\prec z$ by 
    the holonomy map along unstable foliation.
    However, \[H(\mcf^u_F(z))=H(\mcf^u_F(y^+))=H(\mcf^u_F(y))=H(\mcf^u_F(x^+)),\] 
    which implies $H(x^+)=H(z)$. It contradicts the choice of $x^+$.

    For the fifth item, 
    let $U$ be an open set with respect to the topology of $\Lambda$. 
    Then $h(U)$ is an open set in $\TT^2$ 
    (see a precise proof in the next independent Lemma \ref{2 lem Cantor property}).
    It follows that there exists $p'\in{\rm Per}(A)\cap h(U)$. 
    Hence, the extreme point(s) of $H^{-1}(p')$ will be periodic in $U$.
\end{proof}

    To end this section, we give a further property of the semi-conjugacy $h$ along center leaves, 
    which can be regarded as a Cantor-like property.
    
\begin{lem}\label{2 lem Cantor property}
    Let $\Lambda$ be given in Proposition \ref{2 prop injective point set} 
    and $h$ be the semi-conjugacy between $f$ and $A$. 
    Then one has the following:
    \begin{enumerate}
        \item For any $x\in\TT^2$ and $y\in \mcf^c_f(x)$ with $x^+\prec y$ (or $y \prec x^-$), 
            there exists $z\in\mcf^c_f(x)$ such that 
            $h^{-1}\circ h(z)=\{z\}$ and $x^+\prec z\prec y$ (or $y\prec z\prec x^-$).
        \item Let $x\in\TT^2$ and $y_n\in \mcf^c_f(x)$. 
            If $y_n\to x^{\sigma}$ as $n\to+\infty$, 
            then $y_n^\sigma\to x^\sigma$ where $\sigma=+$ or $-$. 
            In particular, if $y_n\prec x^-$ or $x^+\prec y_n$, 
            then $y_n^{-\sigma}\to x^\sigma$ as $n\to+\infty$.
        \item For any $x\in \Lambda$ and $\e>0$, 
            the $\e$-ball of $x$, $B_\e(x)\subset \TT^2$ satisfies that 
            the set $h\big(B_\e(x)\cap \Lambda \big)$ has non-empty interior on $\TT^2$.
    \end{enumerate}
\end{lem}

\begin{proof}
    The existence of center leaf $\mcf^c_f(x)$ for every $x\in \TT^2$ is guaranteed 
    by Corollary \ref{2 cor semi-conj. Fc}. 

    For the first item, 
    let $x=x^+$ and $y\in\mcf^c_f(x)$ such that $x^+\prec y$. 
    By contradiction, we assume that for every $z\in\mcf^c_f(x)$ with $x^+\prec z\prec y$, 
    the set $h^{-1}\circ h(z)$ is a closed interval.
    From this assumption, such two adjacent intervals intersect with each other, hence, the union of these intervals is an interval which we denote by $I$. In particular, for every point $z\in I\subset \mcf^c_f(x)$, one has $h(z)=h(z_0)$, where $z_0\in \mcf^c_f(x)$ is a given point such that $x^+\prec z_0\prec y$.
    Then
    one can take $z_n\in I$ such that 
    $z_n\to x^+$ with $x^+\prec z_n$ and $h(z_n)=h(z_0)$.
    Since $h$ is continuous, 
    \[h(x^+)=\lim_{n\to +\infty}h(z_n)=\lim_{n\to +\infty}h(z_0)=h(z_0).\]
    It contradicts the fact that $x^+\prec z$.

    For the second item, 
    we just prove the case of $\sigma=+$, and the other case is similar. 
    Let $x\in\TT^2$ and $y_n\in \mcf^c_f(x)$ with $y_n\to x^+$.    
    Taking a subsequence, one has $x^+\prec y_n$ or $y_n\prec x^+$ for all $n\in\NN$.  
    
    If $x^+\prec y_n$ for all $n\in\NN$, by the first item, there exists $z_n\in\mcf^c_f(x)$ such that 
    $x^+\prec z_n\prec y_n$ and $h^{-1}\circ h(z_n)=\{z_n\}$. 
    Here, one can assume that 
        \[x^+\prec  y_{n+1}^- \preceq y_{n+1}^+\prec z_n\prec  y_n^-.\] 
    Note that $y_n \to x^+$, one has $z_n\to x^+$.  
    Hence, $y_{n+1}^+\to x^+$ and $y_{n+1}^-\to x^+$ as $n\to +\infty$.   

    Let $y_n\prec x^+$ for all $n\in\NN$. 
    If $h^{-1}\circ h(x)$ is an interval, 
    then for large $n$, we have $y_n\in h^{-1}\circ h(x)$, hence, $y_n^+=x^+$. 
    If $h^{-1}\circ h(x)=\{x\}$, by the first item again, 
    there exists $z_n\in\mcf^c_f(x)$ such that 
    \[y_n\prec z_n\prec x^- \quand h^{-1}\circ h(z_n)=\{z_n\}.\] 
    Hence, one can assume that 
        \[y_n\preceq y_n^+\prec z_n\prec y_{n+1}^- \preceq y_{n+1}^+\prec   x^-=x^+.\] 
    It follows from $y_n\to x=x^+=x^-$ that $y^-_{n+1}\to x^+$ and $y^+_{n+1}\to x^+$ as $n\to +\infty$.

    Finally, we prove the third item. 
    Denote by $\mcl^c$ the foliation corresponding to $H(\mcf^c_F)$ on $\RR^2$. 
    Note that $H$ is homeomorphic along each leaf of the hyperbolic foliation $\mcf^{s/u}_F$, 
    and the center foliation $\mcf^c_F$ and hyperbolic one $\mcf^{s/u}_F$ admit the product structure. 
    Hence, the third item follows from that for every $x\in \Lambda$ and $\e>0$, 
    the $\e$-ball $B_\e^c(x)$ on center leaf $\mcf^c_f(x)$
    satisfies that $h\big(B_\e^c(x)\cap \Lambda \big)$ has non-empty interior on $\mcl^c(h(x))$.
  
    Indeed, let $\e>0$ and $x=x^+$ (or symmetrically $x=x^-$). 
    By the second item, 
    there exists $y_n\in \mcf^c_f(x) \cap\Lambda$ such that $y_n=y_n^+\to x$ as $n\to \infty$. 
    In particular, there exists $y_0=y_0^+\in B_\e^c(x)\cap \Lambda$ and $x\prec y_0$. 
    Since $h$ is continuous and monotonic along the center leaf,  
    $h\big( [x,y_0] \cap \Lambda\big)=h\big( [x,y_0])$ 
    is an interval contained in $h\big(B_\e^c(x)\cap \Lambda \big)$.
\end{proof}

\section{The special property implies semi-conjugacy}\label{sec-special-to-semi}

    In this section, we will prove Theorem \ref{main-thm-scu}. 
    Let $f:\mtt$ be a $C^1$-smooth special DA endomorphism with the hyperbolic linearization $A:\mtt$. We will show that   
    the existence of a $Df$-invariant dominated splitting into direct sum of subbundles implies the existence of a semi-conjugacy.
    The property holds for both $sc$-DA and $cu$-DA cases, 
    but their proofs have some differences. 
    We would like to present some common lemmas and prove for these two cases independently. 
    Specifically, Theorem \ref{main-thm-scu} will follow from 
    Corollary \ref{sc-H-semi-h-2} and Corollary \ref{cu-H-semi-h-2}.

    Since $f$ is special, the foliations $\mcf^1$ and $\mcf^2$ 
    given in Proposition \ref{2 prop foliation on R2} are now $\ZZ^2$-periodic, 
    that is, $\mcf^i(x+n)=\mcf^i(x)+n$, for any $x\in\RR^2$, $n\in \ZZ^2$ and $i=1,2$.  
    Let $H:\mrr$ be the semi-conjugacy between the lifted maps $F:\mrr$ and $A:\mrr$.

    Fix $x_0\in \RR^2$, the map 
        \[ H_n:\mcf^2(x_0+n) \to \tildeL^u\big(H(x_0)+n\big)\] 
    given by $H_n(x)=H(x-n)+n$ is well-defined. 
    In particular, let $x_0\in \RR^2$ be a fixed point of $F$. 
    Note that Fix$(F)\neq \emptyset$. 
    Indeed, let $0\in\RR^2$ be the unique fixed point of $A$. 
    By Proposition \ref{2-3 prop foliation on R2}, $H^{-1}(0)$ is 
    either a single point (hence, a fixed point of $F$), 
    or a compact local center leaf which admits at least one fixed point.

    Moreover, we consider the following map 
    \[ 
     \overline{H}:\bigcup_{n\in\ZZ^2}\mcf^2(x_0+n)\to \bigcup_{n\in\ZZ^2}\tildeL^u\big(H(x_0)+n\big) 
    \] 
    given by $\overline{H}(x)=H_{n_x}(x)$, for any $x\in \mcf^2(x_0+n_x)$.
    
\begin{lem} \label{3 lem Hnm}
    If $\mcf^2(x_0+n)=\mcf^2(x_0+m)$ for some $n\neq m\in\ZZ^2$, then $H_n=H_m$. 
\end{lem}

\begin{rmk}
    We can see $x_0+n\notin\mcf^2(x_0)$ for any $n\in \ZZ^2$, so we can actually define $\overline{H}$ without Lemma \ref{3 lem Hnm}.
    Indeed, if $f$ is $cu$-DA, then $\mcf^2(x_0)$ is a one-dimensional unstable manifold, it cannot be a circle. 
    When $f$ is $sc$-DA, if there is a $n_0\in \ZZ^2$ such that $x_0+n_0\in\mcf^2(x_0)$, 
    then the one-dimensional leaf $\mcf^2(x_0)$ module $\ZZ^2$ cannot be dense in a fundamental domain 
    which contradicts Proposition \ref{sc-H-semi-h-1}. 
    However, Lemma \ref{3 lem Hnm} can be of independent interest as it can be generalized for higher-dimensional $\mcf^2$ 
    (based on the argument of the proof without restriction on the dimension).  
\end{rmk}

\begin{proof}[Proof of Lemma \ref{3 lem Hnm}]
    Let $x\in \mcf^2(x_0+n)=\mcf^2(x_0+m)$. 
    By Proposition \ref{2 prop approach property of H}, 
    \begin{align}
        H_n(x),\  H_m(x)\in \tildeL^s\big(H(x)\big).\label{eq. 3. Hmns} 
    \end{align}
    On the other hand, $x_0+m-n\in\mcf^2(x_0)$ 
    since $f$ is special and $x_0+m\in \mcf^2(x_0+m)=\mcf^2(x_0+n)$. 
    We claim that 
    \begin{align}
        H(x_0)+m-n\in \tildeL^u(H(x_0)). \label{eq. 3. Hmn1} 
    \end{align}
    Otherwise, if $H(x_0)+m-n\notin \tildeL^u(H(x_0))$, 
    one will get two contradictory facts:
    \begin{itemize}
        \item $d\big(H(x_0)+k(m-n),\     
            \tildeL^u (H(x_0))     
            \big)\to +\infty$, as $k\to +\infty$;
        \item Since $f$ is special, for every $k\in\NN$,
            one has $x_0+k(m-n)\in \mcf^2(x_0)$.
            Moreover, by $\|H-{\rm Id}_{\RR^2}\|_{C^0}<C_H$ 
            (Proposition \ref{2 prop semi-conj in R2}), one can get
            \begin{align*}
                H(x_0)+k(m-n) &\in B_{2C_H}\Big( H\big(x_0+k(m-n)\big) \Big)\\ 
                &\subset  B_{2C_H}\big(H(\mcf^2(x_0))\big) = B_{2C_H}\big( \tildeL^u(H(x_0))\big).
            \end{align*}   
    \end{itemize}
    By \eqref{eq. 3. Hmn1}, it follows
    \begin{align*}
        H_n(x)=H(x-n)+n&\in H\big(\mcf^2(x_0+m)-n \big)+n\\
        &= H\big(\mcf^2(x_0+m-n)\big)+n= H\big(\mcf^2(x_0)\big)+n\\
        &= \tildeL^u(H(x_0))+n=\tildeL^u\big( H(x_0)+m-n\big)+n\\
        &=\tildeL^u(H(x_0)+m).
    \end{align*}
    This implies $\tildeL^u(H(x_0)+m)=\tildeL^u(H(x_0)+n)$.  
    Thus, combining with \eqref{eq. 3. Hmns} and the Global Product Structure, one has $H_n(x)=H_m(x)$. 
\end{proof}

    For short, we denote the domain of the map $\overline{H}$ by 
    \[\mathbf{F}:=\bigcup_{n\in\ZZ^2}\mcf^2(x_0+n).\]
    We are going to prove that the map $\overline{H}$ is equal to $H$ on a large set, 
    see also \cite{MT19} for this method to deal with special Anosov endomorphisms on nilmanifolds. 
    We mention in advance that 
    when $f$ is $sc$-DA, $\Bar{\mathbf{F}}=\RR^2$ and $\overline{H}$ can be extended to $\RR^2$. 
    However, when $f$ is $cu$-DA, 
    we have no \emph{a priori} $\Bar{\mathbf{F}}=\RR^2$ and this will be an obstacle.

\begin{lem}\label{3 lem barH general property}
    The map $\overline{H}$ (defined on $\mathbf{F}$) satisfies:
    \begin{enumerate}
        \item $\overline{H}\circ F= A\circ \overline{H}$;
        \item There exists $C_{\overline{H}}>0$ 
            such that $\| \overline{H}-{\rm Id}_{\mathbf{F}} \|_{C^0}<C_{\overline{H}}$;
        \item $\overline{H}$ is uniformly continuous;
        \item $\overline{H}(x+n)= \overline{H}(x)+n$, for any $x\in \mathbf{F}$ and $n\in\ZZ^2$.
    \end{enumerate}
\end{lem}
\begin{proof}
    For the first item,  let $x\in\mathbf{F}$.  
    There exists $n_x\in\ZZ^2$ with $x\in \mcf^2(x_0+n_x)$. 
    Since $x_0\in {\rm Fix}(F)$,
        \[F(x)\in \mcf^2(F(x_0+n_x))=\mcf^2(F(x_0)+An_x))=\mcf^2(x_0+An_x),\] 
    and 
    \begin{align*}
        \overline{H}\circ F(x)&=H(F(x)-An_x)+An_x
        =H\circ F(x-n_x)+An_x\\&=A\circ H(x-n_x)+An_x
        =A\circ (H(x-n_x)+n_x)=A\circ\overline{H}(x).
    \end{align*}

    The second item follows from the corresponding facts of the map $H$ by taking $C_{\overline{H}}=C_H$. 
    Now we continue to prove the third item. 
    
    By the definition and the uniform continuity of $H$ on $\RR^2$, 
    $\overline{H}$ is uniformly continuous along each leaf of $\mcf^2$, 
    namely, for given $\e>0$, 
    there exists $\delta>0$ such that for any $n\in\ZZ^2$ 
    and any $x,y\in \mcf^2(x_0+n)$ with $d(x,y)<\delta$, 
    then $d(\overline{H}(x),\overline{H}(y))<\e$. 
    From the Global Product Structure and the uniform angle between bundles $E^1$ and $E^2$, 
    the proof remains to show the uniform continuity along $\mcf^1$, the transverse foliation of $\mcf^2$.
    
    By contradiction, we assume that 
    \textcolor{black}{there exists $\e_0$ such that for any $\delta_k> 0$ there are} 
    two points 
    \[ x_k\in \mcf^2(x_0+n_k) 
    \quand 
    y_k\in \mcf^2(x_0+m_k) \cap \mcf^1(x_k)\] 
    for some $n_k, m_k\in\ZZ^2$, such that
    $d(x_k,y_k)<\delta_k$ and $d\big(\overline{H}(x_k),\overline{H}(y_k)\big)\geqslant \e_0$.
    
    Note that $\overline{H}(y_k)\in \tildeL^s\big(\overline{H}(x_k)\big)$. 
    Indeed, by \textcolor{black}{Proposition \ref{2-3 prop foliation on R2}}, one has 
    \begin{align*}
    \tildeL^s\big(\overline{H}(y_k)\big)
        =\tildeL^s\big(H(y_k)\big)
        &=H\big(\mcf^1(y_k)\big)\\
        &=H\big(\mcf^1(x_k)\big)
        =\tildeL^s\big(H(x_k)\big)=\tildeL^s\big(\overline{H}(x_k)\big).
    \end{align*}
    Since $d\big(\overline{H}(x_k),\overline{H}(y_k)\big)\geqslant \e_0$ with
    $$
    \overline{H}(x_k)\in\tildeL^u(H(x_0)+n_k)
    \quand
    \overline{H}(y_k)\in \tildeL^u(H(x_0)+m_k)\cap \tildeL^s\big(\overline{H}(x_k)\big),
    $$ 
    by the Global Product Structure of $\tildeL^s$ and $\tildeL^u$,
    these two parallel lines in $\RR^2$
    $$\tildeL^u(H(x_0)+n_k)\quand\tildeL^u(H(x_0)+m_k)$$
    satisfy that
    there is $\eta_0>0$ relying only on $\e_0$ and $A$ such that 
        \[ 
        d\big( \tildeL^u(H(x_0)) +n_k-m_k ,\ \tildeL^u(H(x_0))  \big) =
        d\big( \tildeL^u(H(x_0)+n_k) ,\ \tildeL^u(H(x_0)+m_k)      \big)
        \geqslant \eta_0 . 
        \]
    Moreover, for any $l\in\NN$, 
    \begin{align}
          d\big(\tildeL^u(H(x_0)) +l(n_k-m_k) ,\ \tildeL^u(H(x_0))     \big) \geqslant l\eta_0 . 
          \label{eq. 3. to infinity}
    \end{align}
    
    On the other hand, since $d(x_k,y_k)<\delta_k$, one has 
        \[   d\big( \mcf^2(x_0+n_k) ,\ \mcf^2(x_0+m_k)   \big) 
        =   d\big( \mcf^2(x_k) ,\ \mcf^2(y_k)  \big) <\delta_k,\]
    equivalently,
        \[d\big( \mcf^2(x_0)+n_k-m_k ,\ \mcf^2(x_0)  \big)  <\delta_k.\]
    In particular, for any $l\in \NN$, 
        \[d\big( \mcf^2(x_0)+l(n_k-m_k) ,\ \mcf^2(x_0)+(l-1)(n_k-m_k)   \big)  <\delta_k, \]
    which implies 
    \begin{align}
        d\big( \mcf^2(x_0)+l(n_k-m_k) ,\ \mcf^2(x_0)  \big)  <l\delta_k. 
        \label{eq. 3. to small 1}
    \end{align}
    Since $\|H-{\rm Id}_{\RR^2}\|_{C^0}<C_H$, one has
       $d\big( \mcf^2(x_0) ,\  \tildeL^u(H(x_0)) \big) <C_H$.
    Combining with \eqref{eq. 3. to small 1}, one has
    \begin{align}
        d\big( \tildeL^u(H(x_0))+l(n_k-m_k) ,\ \tildeL^u(H(x_0))  \big)  <l\delta_k+2C_H.
        \label{eq. 3. to small 2}
    \end{align}
    
    For a given $\eta_0$, one can choose $l\in\NN$ such that $l\eta_0>3C_H$. 
    Moreover, let $\delta_k$ be small enough such that $l\delta_k<C_H$. 
    Then \eqref{eq. 3. to small 2} contradicts \eqref{eq. 3. to infinity}.

    Finally, we prove the fourth item. 
    Since $\mcf^2$ is $\ZZ^2$-periodic, for any $n\in \ZZ^2$, one has
        \[x+n\in \mcf^2(x_0+n_x)+n=\mcf^2(x_0+n_x+n).\]
    On the one hand, by the definition of $\overline{H}$, one has 
    \begin{align*}
        \overline{H}(x+n)\in \tildeL^u\big(H(x_0)+n_x+n\big)
        =\tildeL^u\big(H(x_0)+n_x\big)+n
        =\tildeL^u\big(\overline{H}(x)\big)+n,
    \end{align*}
    where the last equality comes from 
    $\overline{H}(x)\in \tildeL^u\big( H(x_0)+n_x\big)$.
    On the other hand, 
    according to Proposition \ref{2 prop approach property of H} 
    (without the assumption that $f$ is special and $x\in\mathbf{F}$) 
    and $x+n\in \mcf^2(x_0+n_x+n)$, 
    \begin{align*}
        \overline{H}(x+n) = H\big(x+n-(n_x+n)\big)+(n_x+n) 
        \in  \tildeL^s \big(H(x)\big)+n=\tildeL^s \big(\overline{H}(x)\big)+n,
    \end{align*}
    where the first inclusion and the last equality are 
    exactly based on Proposition \ref{2 prop approach property of H}.
    Thus, 
    \[ \overline{H}(x+n)\in 
    \tildeL^u\big(\overline{H}(x)+n\big) \cap \tildeL^s \big(\overline{H}(x)+n\big) 
    =\big\{ \overline{H}(x)+n\big\},\]
    which implies $\overline{H}(x+n)= \overline{H}(x)+n$.
\end{proof}

    Now we split the proof of Theorem \ref{main-thm-scu} as the following two cases.
\begin{prop}\label{sc-H-semi-h-1}
    Under the assumption of Theorem \ref{main-thm-scu}, 
    if $f$ is $sc$-DA, then
    \[ \overline{\bigcup_{n\in\ZZ^2}\mcf^2(x_0+n)}
    =\overline{\bigcup_{n\in\ZZ^2}\mcf_F^c(x_0+n)}=\RR^2. \]
\end{prop}
\begin{proof}
    Here $f$ is  $sc$-DA, $\mcf^2(x_0+n)=\mcf^c_F(x_0+n)$ and
        \[ H(\bigcup_{n\in\ZZ^2}\mcf_F^c(x_0+n))=
        \bigcup_{n\in\ZZ^2}\tildeL^u\big(H(x_0+n)\big).\]
    By Proposition \ref{2 prop minimal foliation}, one has
        \[  \bigcup_{n\in\ZZ^2}\tildeL^u\big(H(x_0)+n\big)   \]
    is dense in $\RR^2$. 
    And according to Proposition \ref{2 prop approach property of H}, 
    one has $n^*\in A^k\ZZ^2$ such that    
    $\tildeL^u\big(H(x_0)+n^*\big)$ and $\tildeL^u\big(H(x_0+n^*)\big)$ are close, 
    so one can get
        \[\overline{\bigcup_{n\in\ZZ^2}\tildeL^u\big(H(x_0+n)\big)}= \overline{\bigcup_{n\in\ZZ^2}\tildeL^u\big(H(x_0)+n\big)}=\RR^2.\]
    Thus, by Proposition \ref{2-3 prop foliation on R2}, 
    $H$ is homeomorphism restricted to $\mcf_F^s$,
    and the conclusion follows.
\end{proof}

    There is a direct corollary of Proposition \ref{sc-H-semi-h-1} 
    and the fourth item of Lemma \ref{3 lem barH general property}.

\begin{cor}\label{sc-H-semi-h-2}
    Under the assumption of Theorem \ref{main-thm-scu}, 
    if $f$ is $sc$-DA,
    then $\overline{H}$ can uniquely continuously extend to be $H$. 
    Moreover, $H(x+n)=H(x)+n$, for any $x\in\RR^2$ and $n\in\ZZ^2$. 
    In particular, $h:=\pi\circ H\circ \pi^{-1}$ is well-defined on $\TT^2$
    and a semi-conjugacy between $f$ and $A$.
\end{cor}

    Now we deal with the case of $cu$-DA. 
    Recall that  $\Bar{\Gamma}$, the closure of the $H$-injective set
    \[\Gamma:=\big\{x\in\RR^2\ |\ H^{-1}\circ H(x)=\{x\} \big\},\] 
    is an $F$-invariant set.
   
\begin{prop}\label{cu-H-semi-h-1}
    Under the assumption of Theorem \ref{main-thm-scu}, 
    if $f$ is $cu$-DA, then 
    \begin{enumerate}
        \item $\Bar{\Gamma} \subset \bar{\mathbf{F}}$, the closure of set $\mathbf{F}$;
        \item $\overline{H}$  uniquely continuously extends to $\Bar{\Gamma}$ 
            such that  $\overline{H}|_{\Bar{\Gamma}}=H|_{\Bar{\Gamma}}$ and $H(\Bar{\Gamma})=\RR^2$;
        \item $H(x+n)=H(x)+n$, for any $x\in \Bar{\Gamma}$ and $ n\in\ZZ^2$;
        \item $\Bar{\Gamma}+\ZZ^2=\Bar{\Gamma}$.
    \end{enumerate}
\end{prop}
\begin{proof}
    To prove the first item,  
    applying the same method in proof of Proposition \ref{sc-H-semi-h-1}, one has
        \[H(\mathbf{{F}})=
        H(\bigcup_{n\in\ZZ^2}\mcf_F^u(x_0+n))=
        \bigcup_{n\in\ZZ^2}\tildeL^u\big(H(x_0+n)\big)\]
    is dense in $\RR^2$, 
    by Proposition \ref{2 prop minimal foliation} and Proposition \ref{2 prop approach property of H}. 
    Note that it does not imply $\Bar{\mathbf{F}}=\RR^2$ 
    since $H$ may not be homeomorphic along $\mcf^1$ in this case. 
    However, $\mathbf{F}$ is dense in $\Bar{\Gamma}$. 
    Indeed, let $x\in \Bar{\Gamma}$. 
    Since $H|_{\mcf_F^c(x)}$ is not locally constant at $x$, 
    for every short open segment $J^c$ containing $x$ on $\mcf^c_F(x)$, 
    there exists a non-trivial closed interval $I^c\subset J^c$ 
    such that $H(I^c)$ has positive length and $H^{-1}\circ H(I^c)=I^c$.
    Since $H(\mathbf{{F}})$ is dense in $\RR^2$, there exists $n^*\in\NN$ such that the leaf
        \[\tildeL^u\big(H(x_0+n^*)\big)={H}\big(\mcf_F^u(x_0+n^*)\big)\]
    intersects with $H(I^c)$ at a point $z$. 
    Thus, the set $H^{-1}(z)\subset I^c\subset  J^c$ and the leaf $\mcf_F^u(x_0+n^*)$ intersects $J^c$. 
    This implies $\Bar{\Gamma} \subset \bar{\mathbf{F}}$.

    Recall that by Proposition \ref{2-3 prop foliation on R2},
        $\bar{\Gamma}= \bigcup_{x\in\RR^2} 
        \{ y\in \RR^2\ |\ y\  \text{is an endpoint of} \ H^{-1}\circ H(x)\}$.
    Thus, $H(\Bar{\Gamma})=\RR^2$.
    The second item follows from the uniqueness of $H$ in Proposition \ref{2 prop semi-conj in R2} 
    and the fact that $\bar\Gamma$ is $F^{\pm 1}$-invariant. 
    Indeed, for every $x\in\bar\Gamma$ and $k\in\ZZ$,
    \begin{align}
        d\big(A^k\circ H(x),A^k\circ\overline{H}(x)\big)= d\big(H\circ F^k(x),
        \overline{H}\circ F^k(x)\big) \leqslant 2C_H,
        \label{eq. 3. barH 1}
    \end{align}
    where $\|H-{\rm Id}_{\RR^2}\|_{C^0}\leqslant C_H$ 
    and $\|\overline{H}-{\rm Id}_{\bar\Gamma}\|_{C^0}\leqslant C_{\overline{H}}=C_H$. 
    By \eqref{eq. 3. barH 1}, $H(x)=\overline{H}(x)$ for all $x\in\bar\Gamma$.
   
    The third item is a corollary of the first two items. 
    By the fourth item of Lemma \ref{3 lem barH general property}, 
    for any $x\in \Bar{\Gamma}$ and $n\in\ZZ^2$, one has that $x+n\in \Bar{\Gamma}+\ZZ^2\subset \Bar{\mathbf{F}}$, hence, $\overline{H}(x+n)= \overline{H}(x)+n=H(x)+n$. 
    However, we do not know a priori that $x+n\in\Bar{\Gamma}$ nor $\overline{H}(x+n)=H(x+n)$. To prove this,  it suffices to show that $\Bar{\mathbf{F}}$ is an $F$-invariant set. Then by the same argument of the second item, one has that $\overline{H}|_{\Bar{\mathbf{F}}}=H|_{\Bar{\mathbf{F}}}$, hence, $\overline{H}(x+n)=H(x+n)$. 
    
    Now we prove the $F$-invariance of $\Bar{\mathbf{F}}$, which is in fact a corollary of the first item.
	Note that there exist $n_i\in\ZZ^2 \ (1\leqslant i\leqslant l={\rm deg}(A))$ such that $\ZZ^2= \bigcup_{1\leqslant i\leqslant l}(n_i+A\ZZ^2)$ and 
     \begin{align}
         	\mathbf{F}=\bigcup_{1\leqslant i\leqslant l} \big(\mcf^u_F(x_0+n_i)+A\ZZ^2\big). \label{eq. mathbfF}
     \end{align}
    By the proof of the first item, $\Bar{\Gamma}\subset \overline{\mcf^u_F(x_0+n_i)+A\ZZ^2}$ for every $1\leqslant i\leqslant l$. 
    Hence, we can take \[x\in \bar{\Gamma}\subset \overline{\mcf^u_F(x_0+n_i)+A\ZZ^2}\cap \overline{\mcf^u_F(x_0)+A\ZZ^2}.\]
    It follows that there are $m^1_j\in A\ZZ^2$ and $m^2_j\in A\ZZ^2$ such that when $j\to+\infty$,
   \[ \mcf^u_F(x_0)+m^1_j\to \mcf^u_F(x) \quad {\rm and} \quad \mcf^u_F(x_0+n_i)+m^2_j\to \mcf^u_F(x).\]
    Therefore, $\mcf^u_F(x_0)+m^1_j-m^2_j \to  \mcf^u_F(x_0+n_i)$ as $j\to+\infty$. 
    Note that $(m^1_j-m^2_j)\in A\ZZ^2$, so
    \begin{align}
    \overline{\mcf^u_F(x_0+n_i)+A\ZZ^2}\subset\overline{\mcf^u_F(x_0)+A\ZZ^2}, \quad \forall\ 1\leqslant i\leqslant l. \label{eq. foliation dense}
    \end{align}
    It follows from \eqref{eq. mathbfF} and \eqref{eq. foliation dense} that $F(\Bar{\mathbf{F}})=\overline{\mcf^u_F(x_0)+A\ZZ^2}=\Bar{\mathbf{F}}$.

    For the fourth item, we first claim that for any $z\in\Gamma$ and $n\in \ZZ^2$, 
    we have $z+n\in \bar{\Gamma}$. 
    Indeed, if there exist $z_0\in \Gamma$ and $n_0\in\ZZ^2$ 
    such that $z_0+n_0\notin \bar{\Gamma}$, 
    we can take $y\in \partial\Bar{\Gamma}\subset \Bar{\Gamma}$ 
    with $y\neq z_0+n_0$ such that $H(y)=H(z_0+n_0)$. 
    By the third item, \[H(y-n_0)=H(y)-n_0=H(z_0+n_0)-n_0=H(z_0)+n_0-n_0=H(z_0).\]
    By the injection of $H$ on $\Gamma$, 
    it follows $y-n_0=z_0$, which is an obvious contradiction.
    Then, for any $x\in \bar{\Gamma}$ and $n\in\ZZ^2$, 
    there exists $z_k\in \Gamma$ such that $z_k\to x$. 
    Note that $z_k+n \in \Bar{\Gamma}$, one has $z_k+n\to  x+n\in \Bar{\Gamma}$.     
\end{proof}

\begin{cor}\label{cu-H-semi-h-2}
    Under the assumption of Theorem \ref{main-thm-scu},
    if $f$ is $cu$-DA,
    then $H(x+n)=H(x)+n$, for any $x\in \RR^2$ and $ n\in\ZZ^2$. 
    In particular, $h:=\pi\circ H\circ \pi^{-1}$ is  well-defined on $\TT^2$
    and a semi-conjugacy between $f$ and $A$.
\end{cor}
\begin{proof}
    By the third item of Proposition \ref{cu-H-semi-h-1}, 
    we need only show $H(x+n)=H(x)+n$ for $x\in \Bar{\Gamma}^c$, the complement of $\Bar{\Gamma}$. 
    By Proposition \ref{2-3 prop foliation on R2}, 
    there exists $y\in \Bar{\Gamma}$  such that 
    the curve $[x,y)\subset \mcf_F^c(x)\cap \Bar{\Gamma}^c$ and $H(x)=H(y)$. 
    Then by the fourth item of Proposition \ref{cu-H-semi-h-1},
    one has that $y+n\in \Bar{\Gamma}$ and 
    \[ [x+n,y+n)\in  \mcf_F^c(y+n)\cap \Bar{\Gamma}^c,\] for any $n\in\ZZ^2$. 
    Hence, $H(x+n)=H(y+n)=H(y)+n=H(x)+n$.
\end{proof}

\section{Rigidity of Lyapunov exponents}\label{sec-small-LyaExp}

    Let $f:\mtt$ be a $C^1$-smooth DA endomorphism and $A:\mtt$ be its linearization. 
    Let $f$ be semi-conjugate to $A:\mtt$
    and $\Lambda$ be given in Proposition \ref{2 prop injective point set}. 
    Note that $\Lambda$ is $\mcf^{s/u}_f$-saturated and $f|_{\Lambda}$ is special. 
    Also, $f$ has uniquely integrable center bundle 
    (Corollary \ref{2 cor semi-conj. Fc}).

    For $\{x_n\}\subset \Lambda$, $x\in\Lambda$ and $\sigma=+\ {\rm or}\ -$, 
    we define $x_n\to_\sigma x$ (see FIGURE \ref{Defto}) as $x_n$ 
    converging to $x$ along the $\sigma$-direction of the center foliation $\mcf^c_f$. 
    Precisely, by the local product structure, 
    let $x'_n$ be the unique intersection of local leaves $\mcf^c_f(x)$ and $\mcf^{s/u}_f(x_n)$. 
    Since $\Lambda$ is $\mcf^{s/u}_f$-saturated, we have $x'_n\in\Lambda$.
    Then we can define as follows:
    \begin{itemize}
        \item 
        We say $x_n\to_+ x$ if $x\prec x'_n$ and $d(x_n,x)\to 0$ as $n\to +\infty$;
        \item 
        We say $x_n\to_- x$ if $x'_n\prec x$ and $d(x_n,x)\to 0$ as $n\to +\infty$.
    \end{itemize}  
   \begin{figure}[htbp]
		\centering
		\includegraphics[width=11cm]{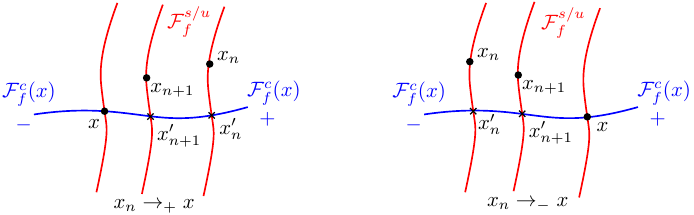}
		\caption{The definition of $x_n\to_{\pm}x$.}	
        \label{Defto}
	\end{figure}
    Similarly, we can define \[x_n\to_\sigma^A x,\] as $x_n$ 
    converging to $x$ along the $\sigma$-direction of the corresponding foliation $\mcl^c:=h(\mcf^c_f)$.
    And we assume that $h$ preserves the orientation on each leaf.

    Let $F:\mrr$ be a lift of $f$ via $\pi:\mrt$ and admit the invariant partially hyperbolic splitting given as \eqref{eq. Fsplitting}
        \[T\RR^2=E^1\oplus_< E^2.\] 
    Denote by $\lambda^1(p)$ the Lyapunov exponent of point $p\in\TT^2$ with respect to the bundle  
        \[E_1:=D\pi(E^1)\subset T\TT^2.\] 
    Recall that $E^1$ is uniquely integrable, so is $E_1$ (Proposition \ref{2 prop foliation on R2}). 
    We denote the integral foliation by $\mcf_1$,
    and denote the infimum and supremum of 
    the set \[\big\{\lambda^1(p) \ |\ p\in {\rm Per}(f|_{\Lambda})\big\},\] 
    by $\lambda_-$ and $\lambda_+$, respectively. 
    
    We mainly obtain the following rigidity result for DA endomorphisms on $\TT^2$.

\begin{prop}\label{4 prop s-rigidity}
    Let $f:\mtt$ be a  $C^1$-smooth 
    DA endomorphism semi-conjugate to its hyperbolic linearization $A:\mtt$. 
    Let $\Lambda$ be given in Proposition \ref{2 prop injective point set}.  
    If $\lambda_-<0$, then $\lambda^1(p)=\lambda^1(q)$ for any $p,q\in {\rm Per}(f|_{\Lambda})$. 
\end{prop}

    Before giving the proof of Proposition \ref{4 prop s-rigidity}, 
    we first give an adapted metric with respect to the periodic data on the set $\Lambda$. 
    Recall that $\Lambda$ is exactly the collection of the endpoints of $h^{-1}\circ h(x)$ for all $x\in \TT^2$.

\begin{prop}\label{4 prop adapted metric}
    Let $f:\mtt$ be a  $C^1$-smooth 
    DA endomorphism semi-conjugate to its hyperbolic linearization $A:\mtt$. 
    Let $\Lambda$ be given in Proposition \ref{2 prop injective point set}. 
    Then for any $\delta>0$ there exists a Riemannian metric $\|\cdot\|$ such that for every $x\in\Lambda$,
        \[ \lambda_--\delta<{\rm log}\|Df|_{E_1(x)}\|<\lambda_++\delta. \] 
\end{prop}
\begin{proof}
    The existence of an adapted metric we desired 
    follows from the Lemma \ref{4 lem measure shadowing}
    which can be viewed as a non-uniformly shadowing property.

\begin{lem}\label{4 lem measure shadowing}
    Let $\mu$ be an ergodic measure with ${\rm supp}(\mu)\subset \Lambda$.
    Then $\mu$ can be approached by periodic measures 
    supported on $\Lambda$ with respect to the weak-star topology. 
    In particular, there exists a sequence of periodic points $p_n\in\Lambda$ 
    such that \[\lim_{n\to +\infty} \lambda^1(p_n)=\lambda^1(\mu).\]
    In particular, $\lambda_-\leqslant \lambda^1(\mu)\leqslant \lambda_+$.
\end{lem}
\begin{proof}[Proof of Lemma \ref{4 lem measure shadowing}]
    We just prove in the case of $sc$-DA, as the other case is quite similar.

    For a given ergodic measure $\mu$ supported on $\Lambda$,
    by the Poincar\'e recurrence theorem, 
    we can take a recurrent $\mu$-typical point $x$. Since $\mu$ is supported on $\Lambda$,  $x\in \Lambda$. 
    Then we can assume that $x=x^+$. 
    And the case of $x=x^-$ follows the similar argument.
    For convenience, we assume that $f$ preserves the orientation on center leaves.

    For a given negative $f$-orbit $x_{-n},\dots,x_{-1},x$ of $x$, 
    we have $x_{-i}=x^+_{-i}$ for any $1\leqslant i\leqslant n$. 
    To simplify, we still denote by $\mcf^c_f(x,\e)$ 
    the positive-orientation component of the local center leaf of $x^+$, 
    namely, $y\in \mcf^c_f(x,\e)$  
    implies $y\in\mcf^c_f(x)$ with $d_{\mcf^c_f}(x,y)\leqslant \e$ and $x\prec y$.
    And we denote the component of $f^{-i}(\mcf^c_f(x,\e))$ containing $x_{-i}$ 
    by $f_0^{-i}(\mcf^c_f(x,\e))$ for any $0\leqslant i\leqslant n$.
  
    We first claim that the local center leaf of $x$ has
    \emph{topological contraction} (see also \cite{PS00}) property 
    along the negative orbit.

\begin{claim}\label{4 claim topological contracting leaf}
    Given $x=x^+\in\Lambda$. 
    For any given $\e>0$ there exists $\eta>0$ such that 
    for every negative $f$-orbit $x_{-n},\dots,x_{-1},x$ of $x$, 
    one has 
        \[ f_0^{-i}\big(\mcf^c_f(x,\eta)\big) \subset \mcf^c_f(x_{-i},\e), 
        \quad \forall\ 0\leqslant i\leqslant n.\]
\end{claim}
\begin{proof}[Proof of Claim \ref{4 claim topological contracting leaf}]
    By contradiction, 
    we assume that there exists $\e_0>0$ such that 
    for any $\eta_m>0$ there exists an $f$-orbit $x_{-n_m},\dots,x_{-1},x$ such that 
    the lengths of local leaves 
    \begin{align}
           |f^{-i}_0\big(  \mcf^c_f(x,\eta_m) \big)|\leqslant \e_0,
        \quad \forall\ 0\leqslant i\leqslant n_m-1,\label{eq. 4. claim4.4}
    \end{align}
    and 
        \[\e_0\leqslant |f^{-n_m}_0\big(  \mcf^c_f(x,\eta_m) \big)|\leqslant C\e_0,\] 
    where $C$ only depends on $f$. 
    Note that we can take $\eta_m\to 0$, hence, $n_m\to +\infty$.
 
    Let $I_m:=f^{-n_m}_0\big(  \mcf^c_f(x,\eta_m) \big)$. 
    By the compactness of $\Lambda$ and the continuity of $\mcf_f^c$, 
    we can assume that $x_{-n_m}\to z\in\Lambda$ and $I_m\to I\subset \mcf^c_f(z)$ 
    with length $\e_0\leqslant |I| \leqslant C\e_0$.
    Moreover, $|f^n(I)|\leqslant C\e_0$ for any $n\in\NN$. 
    Indeed, for given $n\geqslant 1$, 
    by \eqref{eq. 4. claim4.4},
    when $m$ is big enough, $|f^n(I_m)|\leqslant \e_0$ and $f^n(I_m)\to f^n(I)$ as $m\to +\infty$.  
    Hence, by the uniform continuity of $h$,
    there exists a constant $L_0>0$ such that  
    \begin{align}\label{eq. 4 topo contracting}
        |h\circ f^n(I)|< L_0, \quad \forall\ n\in\NN. 
    \end{align}
    However, note that $x_{-n_m}=x^+_{-n_m}$,
    we get $z=z^+$ and every $y\in I$ satisfies $z\prec y$, 
    then we have $|h(I)|>0$. 
    Thus, by $h(I)\subset \mcl^u(h(z))$,
    we further have 
        \[ |h\circ f^n(I)|=|A^n\circ h(I)|\to +\infty \quad{\rm as}\quad n\to +\infty.\] 
    This contradicts \eqref{eq. 4 topo contracting}, 
    and ends the proof of Claim \ref{4 claim topological contracting leaf}.
\end{proof}
  
    From the topological contraction property, 
    we prove Lemma \ref{4 lem measure shadowing} as follows.
 
    Note that for $x=x^+$, we can assume there exists a preimage sequence $x_{-n_k}\to_+ x$. 
    Indeed, let $\bar{x}=h(x)$. 
    Applying Lemma \ref{2 lem preimage dense linear} for choosing $A$-preimages of $\bar{x}$ in small balls which are close to $\bar{x}$ in its positive direction (with respect to the orientation of $h(\mcf^c_f)$), one has 
    $$\bar{x}_{-n_k}\to^A_+\bar{x}\quad{\rm as}\quad k\to +\infty,$$ 
    where $A^{n_k}(\bar{x}_{-n_k})=\bar{x}$. 
    Let $x_{-n_k}=x_{-n_k}^+$ be the positive boundary of $h^{-1}(\bar{x}_{-n_k})$.  
    Then one has $f^{n_k}(x_{-n_k})=x$, and by Lemma \ref{2 lem Cantor property}, one can get 
    $$x_{-n_k}\to_+ x\quad{\rm as}\quad k\to +\infty.$$
    For a recurrence point $x_{-n_k}$ 
    and its related pseudo periodic $f$-orbit 
        \[x_{-n_k},\ x_{-n_k+1},\ \dots,\ x_{-1},\ x,\] 
    the corresponding sequence
        \[ h(x_{-n_k}),\ h(x_{-n_k+1}),\ \dots,\ h(x_{-1}),\ h(x) \]
    gives a pseudo periodic $A$-orbit. 
    By the shadowing property of Anosov endomorphisms (see \cite{AH94}), 
    for any $\e>0$, there exist $x_{-n_k}$ 
    and a periodic point $p'_k$ of $A$ with $A^{n_k}(p_k')=p_k'$ 
    such that \[d\big(A^i(p_k'),h(x_{-i})\big)<\e, \quad  \forall \ 0\leqslant i\leqslant n_k.\] 
    Moreover, when $x_{-n_k}\to_+ x$ one has $h(x_{-n_k})\to^A_+ h(x)$ and  $p_k'\to^A_+ h(x)$.

    Let $p_k=p^-_k\in\Lambda$ be the negative endpoint of $h^{-1}(p_k')$. 
    Then one has $f^{n_k}(p_k)=p_k.$
    By Lemma \ref{2 lem Cantor property}, for any $\e>0$, when $x_{-n_k}$ is sufficiently close to $x$, 
    one can get  $p_k\to_+ x$ with $d(p_k,x)<\e$.
    Let $z_k$ be the unique intersection point of local leaves $\mcf^s_f(p_k)$ and $\mcf^c_f(x)$ 
    and $w_k$ be the unique intersection point of local leaves $\mcf^c_f(p_k)$ and $\mcf^s_f(x)$. 
    For $1\leqslant i\leqslant n_k$, we denote the components of $f^{-i}(z_k)$ and $f^{-i}(w_k)$ 
    on local leaves of $x_{-i}$ by $z_k^{-i}$ and $w_k^{-i}$, respectively. 
    Similarly we denote by $p_k^{-i}$ the component of $f^{-i}(p_k)$ on local leaf $\mcf^s_f(z_k^{-i})$. See FIGURE \ref{shadow}.

     \begin{figure}[htbp]
		\centering
		\includegraphics[width=14.5cm]{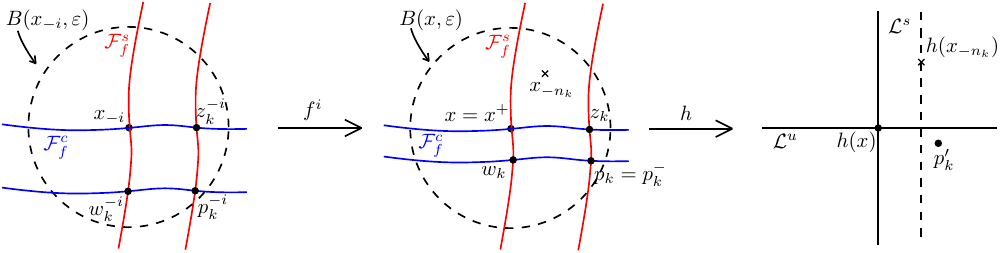}
		\caption{The shadowing property on the set $\Lambda$.}	
        \label{shadow}
	\end{figure}
 
    Note that $h$ restricted to stable leaves is homeomorphic. 
    Hence, for given $\e>0$ there exist $x_{-n_k}$ and  $p_k'$ satisfying 
        \[d_{\mcf^s_f}(w_k^{-i},x_{-i})<\e, \ \  \forall \ 0\leqslant i\leqslant n_k.\] 
    On the other hand, for given $\e>0$, 
    let $\eta>0$ be given in Claim \ref{4 claim topological contracting leaf}. 
    When $x_{-n_k}$ is sufficiently close to $x$,
    one has $z_k\in \mcf^c_f(x,\eta)$, hence, 
        \[d_{\mcf^c_f}(z_k^{-i},x_{-i})<\e,\ \ \forall \ 0\leqslant i\leqslant n_k.\] 
    As a result, by choosing $x_{-n_k}$ appropriately, one can get
    $d(p^{-i}_k,x_{-i})<\e$ for all $0\leqslant i\leqslant n_k$. 
    
    Moreover, note that $x$ is $\mu$-typical, the ergodic measure $\mu$ has the following form
        \[\mu=\lim_{n\to+\infty}\frac{1}{n}\sum_{i=0}^{n-1}\delta_{f^i(x)},\] 
    thus, it can be approached by periodic measures 
    \[\mu_k:=\frac{1}{n_k}\sum_{i=0}^{n_k-1}\delta_{f^i(p_k)},\]  
    where $\delta_y$ is Dirac measure of the point $y$. 
    Hence, the sequence of periodic points $p_k\in\Lambda$ satisfies 
        \[\lim_{k\to +\infty} 
        \lambda^1(p_k)=\lim_{k\to +\infty}\int_{y\in \Lambda}
        {\rm log}|Df |_{E_1(x)}|d\mu_k(y)=\int_{y\in \Lambda}{\rm log}|Df |_{E_1(x)}|d\mu(y)
        =\lambda^1(\mu).\]
    
    This ends the proof of Lemma \ref{4 lem measure shadowing}.    
\end{proof}

    Now we prove the existence of the adapted metric $\|~\|$. 
    Let $\delta>0$ and $|\cdot|$ be a Riemannian metric given in advance.
    Note that for every point $x\in\Lambda$, 
    there exists a subsequence of the measures $\{\mu_{x,n}\}_{n\in\NN}$ 
    converging to an invariant measure $\mu$ with ${\rm supp}(\mu)\subset \Lambda$, 
    where 
        \[\mu_{x,n}=\frac{1}{n}\sum_{i=0}^{n-1}\delta_{f^i(x)}.\]
    For short, let $\mu_{x,n}\to \mu$. 
    It follows from Lemma \ref{4 lem measure shadowing} and the ergodic decomposition that 
    \begin{align*}
        {\rm log}|Df^n|_{E_1(x)}|^{\frac{1}{n}}=\int_{\Lambda}{\rm log}|Df|_{E_1(y)}|d \mu_{x,n}(y)
        \to \int_{\Lambda}{\rm log}|Df|_{E_1(y)}|d \mu(y)\in [\lambda_-,\lambda_+].
    \end{align*}
    Thus, there exists $N=N(x,\delta)\in\NN$ such that for every $n\geqslant N$,
        \[ \lambda_--\frac{\delta}{2}<{\rm log}|Df^n|_{E_1(x)}|^{\frac{1}{n}}<\lambda_++\frac{\delta}{2}.  \]
    Note that $f$ is $C^1$ and the bundle $E_1$ is continuous, by fixing $n\geqslant N$, 
    there exists an open neighborhood $U=U(x,n)\subset \Lambda$ of $x$ 
    such that for every $y\in U$,
        \[ \lambda_--\delta<{\rm log}|Df^n|_{E_1(y)}|^{\frac{1}{n}}<\lambda_++\delta.  \]

    By the compactness of $\Lambda$, 
    there exist finite points $x_1,\dots, x_m\in \Lambda$, 
    finite times $n_i\in\NN\ (1\leqslant i\leqslant m)$ 
    and  finite neighborhoods $U_i=U(x_i,n_i)\subset \Lambda\ (1\leqslant i\leqslant m)$ 
    such that $\cup_{i=1}^m U_i=\Lambda$ and 
        \[ \lambda_--\delta<{\rm log}|Df^{n_i}|_{E_1(y)}|^{\frac{1}{n_i}}<\lambda_++\delta,  \]   
    for every $y\in U_i$. 
    Let 
        \[ \|v\|_i :=\prod_{n=0}^{n_i-1}|Df^n(x)v|^{\frac{1}{n_i}}, \]
    for every $v\in E_1(x)$ and every $x\in \Lambda$. 
    Note that the norm $\|\cdot\|_i$ of the one-dimensional bundle $E_1$ actually induces a norm on $T \Lambda$ (see \cite[formula (3.6)]{BKR2022})  and  
         \[ \frac{\|Df(x)\|_i}{\|v\|_i} 
         = \frac{|Df^{n_i}(x)v|^{\frac{1}{n_i}}}{|v|^{\frac{1}{n_i}}}, 
         \quad \forall \ x\in \Lambda \ {\rm and}\  v\in E_1(x). \]
    Hence, for every $x\in U_i$ one has 
        \[ \lambda_--\delta<{\rm log}\|Df|_{E_1(x)}\|_i<\lambda_++\delta. \]

    Let $\{ \phi_i\}_{i=1}^m$ be a unit decomposition subordinated to $\{ U_i\}_{i=1}^m$, 
    namely, $\phi_i$ is $C^{\infty}$ and supported on $U_i$ and $\sum_{i=1}^m\phi_i=1$. 
    For every $x\in\Lambda$ and $v\in E_1(x)$, let 
        \[ \|v\|= \big(    \prod_{i=1}^m e^{\phi_i(x)} \|v\|_i   \big)^{\frac{1}{m}}. \]
    Then $\|\cdot\|$ induces the metric we desired. This ends the proof of Proposition \ref{4 prop adapted metric}.
\end{proof}

    With the help of the adapted metric $\|\cdot\|$, 
    we prove the rigidity of DA endomorphisms now.

\begin{proof}[Proof of Proposition \ref{4 prop s-rigidity}]

    By contradiction, 
    we assume that there exist $p_0,q_0\in{\rm Per}(f|_{\Lambda})$ 
    such that $\lambda^1(p_0)<\lambda^1(q_0)$.
    Then we have 
        \[ -\infty<\inf_{x\in\Lambda}   
        {\rm log}\|Df|_{E_1(x)}\| \leqslant \lambda_-<\lambda_+\leqslant \sup_{x\in\Lambda}
        {\rm log}\|Df|_{E_1(x)}\|<\infty.   \] 
    
    Given $\delta>0$ arbitrarily small, 
    we can choose two periodic points $p,q\in {\rm Per}(f|_{\Lambda})$ such that
        \[   \lambda^1(p)\leqslant \lambda_-+\delta
        \quand  
        \lambda^1(q)\geqslant \lambda_+-\delta. \]
    Moreover, by Proposition \ref{4 prop adapted metric}, 
    we can assume that  $\|\cdot\|$ is the norm with  
        \[ \lambda_--\delta<{\rm log}\|Df|_{E_1(x)}\|<\lambda_++\delta, 
        \quad {\rm for}\ x\in\Lambda, \]
    since Lyapunov exponents of a regular point are independent of the choice of equivalent metrics. 
    Furthermore, by the uniform continuity, 
    we can also take $\eta_0>0$ such that 
    for any $x,y\in\TT^2$ with $d(x,y)\leqslant \eta_0$, it follows
        \[ -\delta<{\rm log}\|Df|_{E_1(x)}\|-{\rm log}\|Df|_{E_1(y)}\|<\delta.\]

    Here we fix the positive constant $\delta$ with
        \[\delta<
        \min\{\frac{-\alpha\lambda_-}{4-2\alpha}, \frac{\alpha(\lambda_+-\lambda_-)}{4}\},\]
    where 
        \[\alpha=\frac{1}{4}{\rm log}_{\|A|_{L^u}\|} |{\rm det}(A)| \in(0,1)\] 
    as in Claim \ref{4 claim time rate}.

    For convenience, we can assume that both $p$ and $q$ are fixed points. 
    Otherwise, we can apply the following proof to $f^{n_0}$, 
    with $n_0$ being a common period of $p$ and $q$. 
    We further assume that $p=p^+$ and consider two different cases: $q=q^+$ and $q=q^-$. 
    Here we mention in advance that the $sc$-DA case and $cu$-DA case have some differences and similarities. 
    In particular,  the following claims 
    (Claim \ref{4 claim preimage dense}, Claim \ref{4 claim start rate}, Claim \ref{4 claim I J in nbhd} 
    and Claim \ref{4 claim time rate}) 
    exhibit the same dynamical properties for both $sc$-DA case and $cu$-DA case, 
    but their proofs are somewhat different.

    Let $H:\RR^2\to \RR^2$ be the semi-conjugacy between $F$ and $A$ 
    given in Proposition \ref{2 prop semi-conj in R2}. 
    Then we denote by $h=\pi\circ H\circ \pi^{-1}:\TT^2\to\TT^2$ the semi-conjugacy between $f$ and $A$.  
    For  $x\in\Lambda$, we can assume that 
    $h$ restricted to leaves $\mcf_1(x)$ and $\mcf_2(x)$ preserves the orientation, respectively.

\begin{claim}\label{4 claim preimage dense}
    Given two points $x=x^\sigma, y=y^\sigma\in \Lambda\ (\sigma=+ \ {\rm or}\  -)$, 
    then for any $\e>0$, 
    there exist $k_\e=k(\e,x,y) \in\NN$ and $x_\e=x_\e^\sigma\in\Lambda$ 
    such that $f^{k_\e}(x_\e)=f^{k_\e}(x)$ and  $d(x_\e,y)<\e$. 
    In particular, $x_\e\to_\sigma y$ as $\e\to 0$.
\end{claim}
\begin{proof}[Proof of Claim \ref{4 claim preimage dense}]
    This claim is just a corollary of 
    Lemma \ref{2 lem preimage dense linear} and Lemma \ref{2 lem Cantor property}. 
    Without loss of generality, 
    we prove the claim in the case of $\sigma=+$.  
    
    Applying Lemma \ref{2 lem preimage dense linear} to points $\bar{x}:=h(x)$ and $\bar{y}:=h(y)$, 
    there exist $\bar{x}_n\in\TT^2$ and $\bar{k}_n\in\NN$ with 
    $A^{\bar{k}_n}(\bar{x}_n)=A^{\bar{k}_n}(\bar{x})$ and $\bar{x}_n\to_+ \bar{y}$ as $n\to +\infty$. 
    Then by Lemma \ref{2 lem Cantor property}, up to taking a subsequence, 
    one can obtain the point $x_n^+$, the positive boundary of $h^{-1}(\bar{x}_n)$, satisfying 
    $f^{\bar{k}_n}(x_n^+)=f^{\bar{k}_n}(x)$ and $x_n^+\to_+ y$ as $n\to +\infty$. 
    This completes the proof of this claim.
\end{proof}

    Let $I\subset \mcf_1(p)$ be a closed curve with endpoints $p$ and $w$. 
    Then by Claim \ref{4 claim preimage dense}, 
    we will find a curve $J_\e$ and $k_\e\in\NN$ such that $f^{k_\e}(J_\e)=f^{k_\e}(I)$ and 
    one of the endpoints of $J_\e$ is $\e$-closed to $q$. 
    Precisely, we need to divide the choices of $J_\e$ into two cases: $cu$-DA and $sc$-DA.

    \begin{figure}[htbp]
		\centering
		\includegraphics[width=12cm]{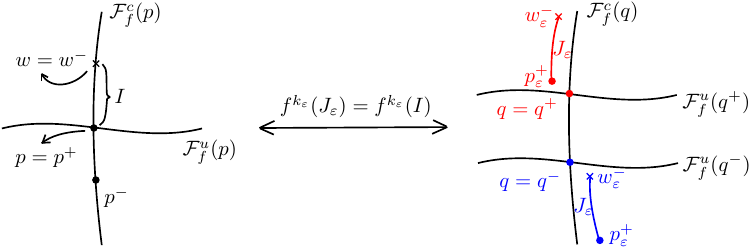}
		\caption{The curve $J_\e$ in the $cu$-DA case.}	
        \label{Jcu}
	\end{figure}

\begin{enumerate}

    \item If $f$ is $cu$-DA, see FIGURE \ref{Jcu}.  
        For small $\delta_p>0$, 
        we can take $I\subset \mcf^c_f(p)$ such that
        \begin{itemize}
            \item $w=w^-$ and  $p\preceq x$ for any $x\in I$;
            \item The curve $I$ is contained in $B_{\eta_0}(p)$, the $\eta_0$-ball of $p$; 
            \item The length of $I$ satisfies  $|I|\geqslant\delta_p$. 
        \end{itemize}
         Given $\e>0$, then
        \begin{enumerate}
            \item If $q=q^+$, applying Claim \ref{4 claim preimage dense} to $p^+$ and $q^+$, 
                there exist 
                $k_\e\in\NN$, $p_\e=p_\e^+\in\TT^2$ and $J_\e\subset \mcf_1(p_\e)=\mcf^c_f(p_\e)$ 
                such that 
                \begin{itemize}
                    \item $f^{k_\e}(p_\e)=f^{k_\e}(p)$ with $d(p_\e,q)<\e$ and $f^{k_\e}(J_\e)=f^{k_\e}(I)$ ;
                    \item For any $z\in J_\e$, 
                        the unique intersection point $z'$ 
                        of local leaves $\mcf^c_f(q)$ and $\mcf^{u}_f(z)$ satisfies $q\prec z'$.
                \end{itemize}
                Here we mention that $J_\e$ may not be contained in $\Lambda$ as the curve $I$ may not do.
            \item If $q=q^-$, applying Claim \ref{4 claim preimage dense} to $w^-$ and $q^-$ 
                we obtain 
                $k_\e\in\NN$, $w_\e\in\TT^2$ and $J_\e\subset \mcf_1(w_\e)=\mcf^c_f(w_\e)$ 
                such that 
                \begin{itemize}
                    \item $f^{k_\e}(w_\e)=f^{k_\e}(w)$ with $d(w_\e,q)<\e$ and $f^{k_\e}(J_\e)=f^{k_\e}(I)$;
                    \item For any $z\in J_\e$, 
                        the unique intersection point $z'$ 
                        of local leaves $\mcf^c_f(q)$ and $\mcf^{u}_f(z)$ satisfies $z'\prec q$.
                \end{itemize}
        \end{enumerate}

    \begin{figure}[htbp]
		\centering
		\includegraphics[width=12cm]{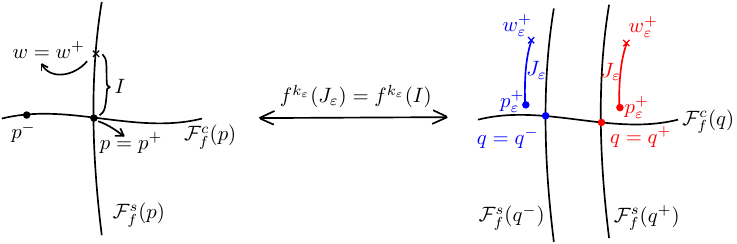}
		\caption{The curve $J_\e$ in the $sc$-DA case.}	
        \label{Jsc}
	\end{figure}

    \item  If $f$ is $sc$-DA, see FIGURE \ref{Jsc}. 
        For small $\delta_p>0$, we can take $I\subset \mcf_1(p)=\mcf^s_f(p)$ such that
        \begin{itemize}
            \item $w=w^+$;
            \item The curve $I$ is contained in $B_{\eta_0}(p)$, the $\eta_0$-ball of $p$; 
            \item The length of $I$ satisfies  $|I|>\delta_p$.
        \end{itemize}
        Given $\e>0$, just like the $cu$-case, then 
        \begin{enumerate}
            \item If $q=q^+$, applying Claim \ref{4 claim preimage dense} to $p^+$ and $q^+$, 
                there exist 
                $k_\e\in\NN$, $p_\e=p_\e^+\in\TT^2$ and $J_\e\subset \mcf_1(p_\e)=\mcf^s_f(p_\e)$ 
                such that 
                \begin{itemize}
                    \item $f^{k_\e}(p_\e)=f^{k_\e}(p)$ with $d(p_\e,q)<\e$ and $f^{k_\e}(J_\e)=f^{k_\e}(I)$;
                    \item The local stable leaf of $J_\e$ intersects with 
                        the local center leaf $\mcf^c_f(q)$ at the unique point $p_\e'$ 
                        satisfying $q^+=q\prec p'_\e$.
                \end{itemize}
                Here $J_\e$ must be contained in $\Lambda$ as the curve $I$ does.
            
            \item If $q=q^-$, applying Claim \ref{4 claim preimage dense} to $p^-$ and $q^-$,   
                there exist $k_\e\in\NN$ and $p_\e^-\in\TT^2$ 
                such that $f^{k_\e}(p^-_\e)=f^{k_\e}(p^-)$ and $p_\e^-\to_- q^-$. 
                Then by the second item of Lemma \ref{2 lem Cantor property} 
                and the forth item of Proposition \ref{2 prop injective point set}, 
                we have $p_\e^+\to_- q^-$. 
                And it is clear that $f^{k_\e}(p^+_\e)=f^{k_\e}(p^+)$. 
                Hence, we obtain 
                $k_\e\in\NN$, $p_\e^+=p_\e\in\TT^2$ and $J_\e\subset \mcf_1(p_\e)=\mcf^c_f(p_\e)$ 
                such that 
                \begin{itemize}
                    \item $f^{k_\e}(p_\e)=f^{k_\e}(p)$ with $d(p_\e,q)<\e$ and $f^{k_\e}(J_\e)=f^{k_\e}(I)$;
                    \item The local stable leaf of $J_\e$ intersects with 
                    the local center leaf $\mcf^c_f(q)$ at the unique point $p_\e'$ 
                    satisfying $p'_\e\prec q=q^-$.
                \end{itemize}
        \end{enumerate}

\end{enumerate}

\begin{claim}\label{4 claim start rate}
    Let $I\subset \mcf_1(p)$ be given as above with length $|I|=\delta_p$. 
    Then there exist constants $\e_0>0$ and $C>1$ such that for every $\e\in(0,\e_0)$,
        \[\frac{|I|}{|J_\e|}\in[C^{-1},C],\]
    where  $J_\e$ is given as above.
\end{claim}
\begin{proof}[Proof of Claim \ref{4 claim start rate}]
    Let $\bar{I}=h(I)$, $\bar{J}_\e=h(J_\e)$, $\bar{p}=h(p)$, 
    $\bar{w}=h(w)$, $\bar{p}_\e=h(p_\e)$, $\bar{w}_\e=h(w_\e)$ and  $\bar{q}=h(q)$. 
    Denote the length of $\bar{I}$ by $\bar{\delta}_p$. 
    Note that \[A^{k_\e}(\Bar{I})=A^{k_\e}(\Bar{J}_\e).\] 
    This implies that 
    $\Bar{J}_\e$ is a rigid translation of $\Bar{I}$, since $A$ is linear.
    Hence, we have $|\Bar{J}_\e|=\Bar{\delta}_p=|\bar{I}|$.
     
    We will prove this claim in $cu$-DA case and $sc$-DA case, respectively.
     
    When $f$ is $cu$-DA, we assume that $q=q^+$, as the case $q=q^-$ is similar. 
    We consider three points $q\prec q_1\prec q_2\prec q_3 \in \mcf^c_f(q)$ such that 
    \begin{itemize}
        \item $h^{-1}\circ h(q_i)=\{q_i\}$ for $i=1,2,3$.
        \item The point $\bar{q}_i=h(q_i)$ satisfies 
            \[\frac{\bar{\delta}_p}{4}
            \leqslant d(\bar{q},\bar{q}_1)
            <\frac{\bar{\delta}_p}{3},\ \ \frac{\bar{\delta}_p}{2}
            \leqslant d(\bar{q},\bar{q}_2)
            <\bar{\delta}_p,\ \  2\bar{\delta}_p\leqslant d(\bar{q},\bar{q}_3)
            <3\bar{\delta}_p. \]
    \end{itemize}
    This can be done by Lemma \ref{2 lem Cantor property}. 
    Without loss of generality, we assume that 
        \[ d(\bar{q},\bar{q}_1)=\frac{\bar{\delta}_p}{4},\ \ d(\bar{q},\bar{q}_2)
        =\frac{\bar{\delta}_p}{2},\ \  d(\bar{q},\bar{q}_3)=2\bar{\delta}_p. \]
    Denote by $\bar{p}_\e'$ and $\bar{w}_\e'$ the unique intersection points of local leaf $\mcl^s(\bar{q})$ 
    with local leaves $\mcl^u(\bar{p}_\e)$ and $\mcl^u(\bar{w}_\e)$, respectively. 
    When $\e_0$ is small enough, one has that for $\e<\e_0$,
        \[ \Bar{q}\prec \bar{p}_\e'\prec  \bar{q}_1 \ 
            \ {\rm and}\ \ 
           \Bar{q}_2\prec \bar{w}_\e'\prec  \bar{q}_3.  \] 
    See FIGURE \ref{q123}. It follows that
        \[ d_{\mcf^c_f}(q_1,q_2) <d_{\mcf^c_f}(p_\e',w_\e')< d_{\mcf^c_f}(q,q_3),  \]
    where $p'_\e$ and $w_\e'$ are the unique intersection points of local leaf $\mcf^c_f(q)$ 
    with local leaves $\mcf_f^u(p_\e)$ and $\mcf_f^u(w_\e)$, respectively.

    \begin{figure}[htbp]
		\centering
		\includegraphics[width=12cm]{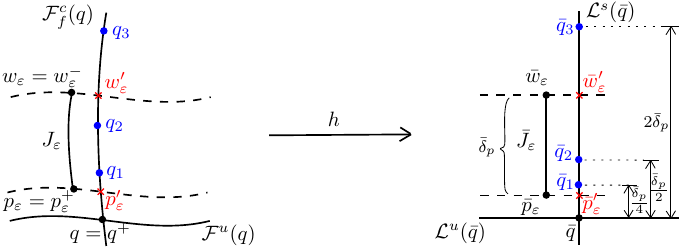}
		\caption{The choice of points $q_1,q_2,q_3$.}	
        \label{q123}
	\end{figure}

    Let $q',q_i'\ (i=1,2,3)$ be the intersections of local leaf $\mcf^c_f(p_\e)$ 
    with local leaves $\mcf_f^u(q), \mcf_f^u(q_i)$, respectively.  
    Then, by the continuity of holonomy map along the unstable leaf $\mcf^u_f(x)$ for $x\in\Lambda$, 
    given a size $\eta>0$ of the local unstable foliation, 
    one can get that $d_{\mcf^u_f}(q',q_3')$ has an upper bound $C_2<+\infty$ 
    and  $d_{\mcf^u_f}(q_1',q_2')$ has a lower bound $C_1>0$ 
    where $C_1$ and $C_2$  only  depend on $\eta$. 
    This implies that the length $|J_\e|$ satisfies, for $\e<\e_0$,
        \[0<C_1\leqslant d_{\mcf^u_f}(q_1',q_2') <|J_\e|<d_{\mcf^u_f}(q',q_3')\leqslant C_2<+\infty.\]  

    Now, let $f$ be $sc$-DA.  
    Note that the semi-conjugacy $h$ restricted to each stable leaf is in fact a homeomorphism.  
    It is clear that there exists $\delta>0$ such that $|J_\e|\geqslant \delta$ for every $\e>0$. 
    Otherwise, there exist $w_\e$ and $q_\e$ such that 
    \[
    d(w_\e,q_\e)\to 0 
    \quand 
    d(\bar{w}_\e,\bar{q}_\e)=|\bar{J}_\e|=\Bar{\delta}_p>0. 
    \]
    It contradicts the continuity of $h$. 
    
    On the other hand, the length of $J_\e$ has an upper bound 
    which is given by the lift on the universal cover. 
    Indeed, let $H$ be the lift of $h$. 
    Then there exists $C_0$ such that $\|H-{\rm Id}_{\RR^2}\|_{C^0}<C_0$. 
    This implies that the lifts $\Tilde{w}_\e,\Tilde{p}_\e\in\RR^2$ of $w_\e, q_\e$ 
    satisfy $d(\Tilde{w}_\e,\Tilde{p}_\e)<2C_0+\Bar{\delta}_p$. 
    Otherwise, 
        \[d(H(\Tilde{w}_\e), H(\Tilde{p}_\e))
        \geqslant d(\Tilde{w}_\e,\Tilde{p}_\e)
            -d(H(\Tilde{w}_\e),\Tilde{w}_\e)-d(H(\Tilde{p}_\e),\Tilde{p}_\e)
        >\Bar{\delta}_p.\] 
    It contradicts the fact that $|\Bar{J}_\e|=\Bar{\delta}_p$.
    Hence, the length $|J_\e|$ is uniformly bounded away from $0$. 
     This ends the proof of Claim \ref{4 claim start rate}.
\end{proof}

    Recall that for given $\delta$, 
    we fix $\eta_0>0$ such that for any $x,y\in\TT^2$ with $d(x,y)\leqslant \eta_0$, 
    one has 
        \[ -\delta<{\rm log}\|Df|_{E_1(x)}\|-{\rm log}\|Df|_{E_1(y)}\|<\delta.\]
    Let $B_{\eta_0}(p)$ and $B_{\eta_0}(q)$ be the  $\eta_0$-balls of $p$ and $q$, respectively. 
    Here we fix a $\mcf_1$-foliation box $D^+$ of $q=q^+$ such that 
\begin{itemize}
    \item $D^+\subset B_{\eta_0}(q) $;
    \item The  boundaries of $D^+$ intersects with local leaf $\mcf^c_f(q)$ 
        at two points $q$ and  $q_D\in\Lambda$ 
        such that $h^{-1}\circ h(q_D)=\{q_D\}$ and $q\prec q_D$;
    \item The boundaries of $D^+$ are leaves of foliation $\mcf_1$ or $\mcf_2$.
\end{itemize}
    This can be done by Lemma \ref{2 lem Cantor property}. 
    Similarly, when $q=q^-$ 
    we can define foliation box $D^-$ as same as above three items but $q_D\prec q$. 

    In the rest of the proof, 
    we will consider the curves $I\subset B_{\eta_0}(p)$ and $J_\e \in D^{\pm}$. 
    Indeed, we have the following claim.

\begin{claim}\label{4 claim I J in nbhd}
    Let $q=q^\sigma \ (\sigma=+ \ {\rm or}\ -)$. 
    There exist $\delta_p>0$ and $\e_1>0$ such that for $I$ with $|I|=\delta_p$ and any $\e\in(0,\e_1)$, 
    one has $f^k(I)\subset B_{\eta_0}(p)$ for every $k\geqslant 0$, and $J_\e \subset D^{\sigma}$.
\end{claim}
\begin{proof}[Proof of Claim \ref{4 claim I J in nbhd}]
    It is clear that when $\delta_p$ is small, we have $I\subset B_{\eta_0}(p)$. 
    For $J_\e$, we just prove the claim for $\sigma=+$, and the proof of the other case is similar. 
   
    Let $\Bar{D}^+=h(D^+)$. Note that $\bar{D}^+$ has non-empty interior on $\TT^2$. 
    Hence, when $\delta_p$ and $\e_1$ are small enough, 
    one has $\Bar{J}_\e\subset {\rm Int}\bar{D}^+$ for $\e<\e_1$. 
    By the construction of $J_\e$ and $D^+$, we obtain $J_\e\subset D^+$.

    Note that we can also take an $\mcf_1$-foliation box $D^+_p$ of $p=p^+$ 
    which satisfies three conditions of $D^+$ but using $p$ to replace $q$. 
    Let $\delta_p$ be small such that $I\subset D^+_p$. 
    It follows from $\Bar{I}\subset \mcl^s(\Bar{p})\cap h(D^+_p) $ 
    that $A^k(\Bar{I})\subset \mcl^s(\Bar{p})\cap h(D^+_p)$ for every $k\geqslant 0$. 
    Since the endpoints of $f^k(I)$ are $f^k(p^+)$ and $f^k(w^-)$, 
    one has  $f^k(I)\subset D^+_p\subset B_{\eta_0}(p)$. This ends the proof of Claim \ref{4 claim I J in nbhd}.
\end{proof}

    Now we take $\delta_p, I$ and $\e_1$ satisfying Claim \ref{4 claim I J in nbhd}. 
    For $q=q^\sigma\ (\sigma=+ \ {\rm or}\ -)$ and $\e<\e_1$, 
    we can define a $\mcf_1$-foliation box $D_\e^\sigma$ just like the construction of $D^\sigma$ as follows:
    \begin{itemize}
        \item $D_\e^\sigma\subset B_{\e}(q) $;
        \item The  boundaries of $D_\e^\sigma$ intersects with 
            local leaf $\mcf^c_f(q)$ at two points $q$ and  $q_\e\in\Lambda$ 
            such that $h^{-1}\circ h(q_\e)=\{q_\e\}$, moreover,  if $q=q^+$, then $q\prec q_\e$ and if $q=q^-$, then $q_\e\prec q$;
        \item The boundaries of $D_\e^\sigma$ are leaves of foliation $\mcf_1$ or $\mcf_2$.
    \end{itemize}
    
    Denote by
    \[K_\e:=\inf\{ k\in\NN\ |\ \exists\  p_\e\in D^\sigma_\e  
    \ {\rm s.t.} \ f^k(p_\e)=f^k(p)\}.  \]  
    Then for $K_\e$ (the specific replacement of $k_\e$), 
    we choose a corresponding curve that is still denoted by $J_\e$ 
    satisfying $f^{K_\e}(I)=f^{K_\e}(J_\e)$.
    Denote by $N_\e\in\NN$
    the first time such that $f^{N_\e}(J_\e)$ is not contained in $D^\sigma$, 
    namely, 
        \[f^{i}(J_\e)\subset D^\sigma, \ \forall\  0\leqslant i\leqslant N_\e-1 
        \ \  {\rm and}\ \  
        f^{N_\e}(J_\e)\cap \big(\TT^2\setminus D^\sigma\big)\neq \emptyset.\]

\begin{claim}\label{4 claim time rate}
    There exist constants $\e_2>0$ and $\alpha\in(0,1)$ such that for every $\e\in(0,\e_2)$, 
    one has 
    \[ \frac{N_\e}{K_\e}\geqslant \alpha.\]
\end{claim}
\begin{proof}[Proof of Claim \ref{4 claim time rate}]
    We just prove this claim for $q=q^+$, and the proof of the other case is similar. 
    By Lemma \ref{2 lem Cantor property}, when $\e_2<\e_1$ is small enough, 
    we have $D_\e^+\subset D^+$ for every $0<\e<\e_2$. 

    Let $\bar{D}^+_\e=h(D_\e^+)$. 
    Note that $\bar{D}^+_\e$ and $\bar{D}^+$ are both $\mcl^{s/u}$-foliation boxes of $\bar{q}$ 
    with non-empty interior on $\TT^2$, 
    since $h$ preserves the leaves. For convenience, 
    we can assume that $\bar{D}^+_\e$ and $\bar{D}^+$ are both squares 
    with side lengths $\bar{\e}$ and $\bar{\eta}_0$, respectively.

    Denote by
        \[\bar{K}_\e:=\inf\{ k\in\NN\ |\ \exists z\in \bar{D}^+_\e  \ {\rm s.t.} \ A^k(z)=A^k(\bar{p})\}.   \]
    It is clear that 
        \[K_\e\leqslant \bar{K}_\e \leqslant {\rm log}_\gamma \frac{\bar{\e}}{2C_A},\] 
    where  $\gamma\in (0,1)$ and $C_A$ are constants given in Lemma \ref{2 lem preimage dense linear}.

    Denote by $\Bar{p}_\e'$ the unique intersection point of 
    local leaves $\mcl^s(\Bar{p}_\e)$ and $\mcl^u(\Bar{q})$. 
    Denote by $\Bar{N}_\e$ the first time such that 
    $A^{\Bar{N}_\e}(\Bar{p}_\e')$ is not contained in $\bar{D}^+_\e$, namely,
        \[A^{i}(\Bar{p}_\e')\in \bar{D}^+_\e, \ \forall\  0\leqslant i\leqslant \Bar{N}_\e-1 
        \ \  {\rm and}\ \  
        A^{\Bar{N}_\e}(\Bar{p}_\e')\notin \bar{D}^+_\e.\]
    It follows from Claim \ref{4 claim I J in nbhd} that $\Bar{J}_\e\subset \Bar{D}^+$.
    Then by the construction of the foliation box $\bar{D}^+_\e$ 
    and $\Bar{J}_\e$ is a curve on a stable leaf, one has 
    \[N_\e\geqslant \Bar{N}_\e\geqslant {\rm log}_{\nu}\frac{\Bar{\eta}_0}{\Bar{\e}},\]
    where $\nu=\|A|_{L^u}\|>1$. 

    As a result, one has
        \[\frac{N_\e}{K_\e} 
        \geqslant \frac{\Bar{N}_\e}{\Bar{K_\e}}
        \geqslant \frac{{\rm ln}\gamma}{{\rm ln} \nu} 
        \cdot \frac{{\rm ln}\Bar{\eta}_0-{\rm ln}\Bar{\e}}{{\rm ln}\Bar{\e}-{\rm ln}2C_A}.\]
    Note that $\Bar{\e}$ tends to $0$ as $\e \to 0$. 
    Hence, when $\e_2$ is small enough, one has 
        \[\frac{N_\e}{K_\e} \geqslant -\frac{1}{2}{\rm log}_\nu \gamma,\]
    for every $\e\in(0,\e_2)$. 
    By Lemma \ref{2 lem preimage dense linear}, 
    one can take $\gamma=|{\rm det}(A)|^{-\frac{1}{2}}$. 
    Then it follows
        \[\alpha:=-\frac{1}{2}{\rm log}_\nu \gamma=\frac{1}{4}{\rm log}_\nu |{\rm det}(A)| \in (0,1),\] 
    since $\nu>|{\rm det}(A)|>1$.
\end{proof}

    Let $0<\e<\min\{\e_0,\e_1,\e_2\}$, 
    where $\e_0, \e_1, \e_2$ are given in 
    Claim \ref{4 claim start rate}, Claim \ref{4 claim I J in nbhd} 
    and Claim \ref{4 claim time rate}, respectively. 
    We are going to calculate the error 
    between the lengths of $f^{K_\e}(I)$ and $f^{K_\e}(J_\e)$ 
    induced by the gap of Lyapunov exponents. 
    
    Note that 
        \[{\rm log}\|Df|_{E_1(x)}\|<\delta+\lambda^1(p), 
        \ \ \forall \ x\in B_{\eta_0}(p),\]
    by Claim \ref{4 claim I J in nbhd}, we have
        \[|f^{K_\e}(I)|\leqslant |I|\cdot {\rm exp}(K_\e(\lambda_-+2\delta)).\]
    On the other hand, by $f^{i}(J_\e)\subset B_{\eta_0}(q)$ for $0\leqslant i\leqslant N_\e-1$, 
    we have
        \[ |f^{N_\e}(J_\e)|\geqslant  |J_\e|\cdot {\rm exp}(N_\e(\lambda_+-2\delta)). \]
    Although we have 
        \[{\rm log}\|Df|_{E_1(x)}\|>\lambda_--\delta, 
        \ \ \forall\ x\in\Lambda,\] 
    and 
        \[ \big|{\rm log}\|Df|_{E_1(x)}\|- {\rm log}\|Df|_{E_1(y)}\| \big| 
        <\delta,\ \ \forall\ x\in \TT^2 \ {\rm and}\ y\in B_{\eta_0}(x), \]
    we cannot calculate $f^{K_\e}(J_\e)$ directly.
    Because  $J_\e$ is not contained in $\Lambda$ for the $cu$-DA case,
    $f^k(J_\e)$ could be not in an $\eta_0$-ball for some $k>N_\e$. 
    
    This leads us to marking a specific iteration. We take $K_0\in [N_\e,K_\e]$ such that
    $K_0$ is the last time $k$ satisfying 
    $f^k(J_\e)$ is not in the  $\eta_0$-ball $B_{\eta_0}(f^k(p_\e))$, 
    namely, 
    \[f^{K_0}(J_\e)\not\subset B_{\eta_0}(f^{K_0}(p_\e)) 
    \quad {\rm and}\quad f^{k}(J_\e)\subset B_{\eta_0}(f^{k}(p_\e)) 
    \quad {\rm for}\quad K_0<k\leqslant K_\e. \]
    Here we claim that when $\e\to 0$ (so $K_\e\to +\infty$), one has $K_\e>N_\e$.
    Otherwise, if $K_\e\leqslant N_\e$, 
    then \[|f^{K_\e}(J_\e)|\geqslant |J_\e|\cdot {\rm exp}(K_\e(\lambda_+-2\delta))\] 
    and 
    \begin{align*} 
        1= \frac{|f^{K_\e}(J_\e)|}{|f^{K_\e}(I)|}
        &\geqslant \frac{|J_\e|}{|I|}\cdot {\rm exp}\big( K_\e(\lambda_+-\lambda_--4\delta)  \big) \\
        &\geqslant C^{-1}{\rm exp}\big( K_\e(\lambda_+-\lambda_--4\delta)  \big) \to +\infty,
    \end{align*}
    as $ \e\to 0$. 
    Here recall that $\delta<\frac{\lambda_+-\lambda_-}{4}$, 
    and $C$ is given in Claim \ref{4 claim start rate}. 
    Therefore, we can assume $K_\e>N_\e$ when $\e$ is small enough.
    Then there are the following possibilities for the number $K_0$.
    \begin{enumerate}
        \item If $K_0$ does not exist (or $K_0<N_\e$), that is,  
            $f^{k}(J_\e)\subset B_{\eta_0}(f^{k}(p_\e))$ for every $k\in[N_\e,K_\e]$. 
            Then one has 
            \begin{align*}
                |f^{K_\e}(J_\e)|
                &=|f^{K_\e-N_\e}\circ f^{N_\e}(J_\e)|   \\
                &\geqslant |J_\e|\cdot {\rm exp}(N_\e(\lambda_+-2\delta))\cdot 
                \prod_{k=N_\e}^{K_\e-1}{\rm exp}(\|Df|_{E_1(f^k(p_\e))}\|-\delta)   \\
                &\geqslant |J_\e|\cdot {\rm exp}\big( N_\e(\lambda_+-2\delta)+(K_\e-N_\e)(\lambda_--2\delta) \big) \\
                &\geqslant |J_\e|\cdot {\rm exp}(K_\e(\lambda_--2\delta))\cdot{\rm exp}(N_\e(\lambda_+-\lambda_-)) \\
                &\geqslant |J_\e|\cdot {\rm exp}\big(K_\e(\alpha\lambda_++(1-\alpha)\lambda_--2\delta\big),
            \end{align*}
            where $\alpha$ is given in Claim \ref{4 claim time rate}. 
            Hence, 
            \begin{align*}
                \frac{|f^{K_\e}(J_\e)|}{|f^{K_\e}(I)|}
                &\geqslant \frac{|J_\e|}{|I|}\cdot {\rm exp}\big( K_\e(\alpha(\lambda_+-\lambda_-)-4\delta) \big) \\
                &\geqslant C^{-1} {\rm exp}\big( K_\e(\alpha(\lambda_+-\lambda_-)-4\delta)  \big).
            \end{align*}
            Since $\delta<\frac{\alpha(\lambda_+-\lambda_-)}{4}$ 
            and $K_\e\to +\infty$ as $\e\to 0$, 
            one has
                \[\frac{|f^{K_\e}(J_\e)|}{|f^{K_\e}(I)|}\to +\infty, \ \ {\rm as}  \ \e\to 0.\]  
            It contradicts the fact that $f^{K_\e}(J_\e)=f^{K_\e}(I)$.
        
        \item  If $K_0=K_\e$, 
            then $f^{K_\e}(J_\e)\not\subset B_{\eta_0}(p)$ since $f^{K_\e}(p_\e)=f^{K_\e}(p)=p$. 
            It contradicts the fact that $f^{K_\e}(J_\e)=f^{K_\e}(I)\subset B_{\eta_0}(p)$, 
            see Claim \ref{4 claim I J in nbhd}.
        
        \item If $K_0\in [N_\e,K_\e)$, 
            then $f^k(J_\e)\subset B_{\eta_0}(f^k(p_\e))$ for $k>K_0$. 
            And $f^{K_0}(J_\e)\not\subset B_{\eta_0}(f^{K_0}(p_\e))$ 
            implies $f^{K_0+1}(J_\e)>\eta_0$, 
            so there exists a constant $R_0>0$ 
            which depends only on $\eta_0$ and $\min_{x\in\TT^2}\|Df|_{E_1(x)}\|$ 
            such that  $f^{K_0+1}(J_\e)>R_0$. Thus, it follows that 
            \begin{align*}
                 |f^{K_\e}(J_\e)|
                 &=|f^{K_\e-K_0-1}\circ f^{K_0+1}(J_\e)|    \\
                 &\geqslant R_0\cdot \prod_{k=K_0+1}^{K_\e-1}{\rm exp}(\|Df|_{E_1(f^k(p_\e))}\|-\delta) \\
                 &\geqslant R_0 \cdot {\rm exp}\big((K_\e-1-K_0)(\lambda_--2\delta)\big).
            \end{align*}
            Since $\lambda_-<0$ and $1+K_0\geqslant N_\e\geqslant \alpha K_\e$, one has
                \[|f^{K_\e}(J_\e)|
                \geqslant R_0\cdot {\rm exp}\big(K_\e(1-\alpha)(\lambda_--2\delta)\big).\]
            Hence,
            \begin{align*}
                \frac{|f^{K_\e}(J_\e)|}{|f^{K_\e}(I)|}
                &\geqslant \frac{R_0}{|I|}\cdot {\rm exp}\big( K_\e(1-\alpha)(\lambda_- -2\delta)
                -K_\e(\lambda_-+2\delta)  \big),\\
                &\geqslant \frac{R_0}{|I|}\cdot {\rm exp}\big( K_\e(-\alpha\lambda_- +(2\alpha-4)\delta)  \big).  
            \end{align*}
            Since $\lambda_-<0$ and $\delta<\frac{-\alpha\lambda_-}{4-2\alpha}$, 
            one has
                \[\frac{|f^{K_\e}(J_\e)|}{|f^{K_\e}(I)|}\to +\infty,\ \ {\rm as}\ \e\to 0. \]
            It contradicts the fact that $f^{K_\e}(J_\e)=f^{K_\e}(I)$.
    \end{enumerate}

    All in all, we obtain the contradiction induced by the exponent gap. 
    Hence, $\lambda^1(p)=\lambda^1(q)$ for any $p,q\in {\rm Per}(f|_\Lambda)$.
    This ends the proof of Proposition \ref{4 prop s-rigidity}.
\end{proof}

\begin{rmk}
    We mention that through the whole proof of Proposition \ref{4 prop s-rigidity}, 
    the condition $\lambda_-<0$ is used 
    only in calculating the error between $|f^{K_\e}(I)|$ and $|f^{K_\e}(J_\e)|$ 
    in the setting of $cu$-DA endomorphisms with $K_0\in[N_\e,K_\e)$.
    The existence of $\lambda_-<0$ will be shown later for the $cu$-DA case in Section \ref{sec-DA-cu}.
\end{rmk}

\section{DA endomorphisms with stable bundles}\label{sec-DA-sc}

    In this section, we prove Theorem \ref{main-thm-sc} 
    which focus on the rigidity of stable Lyapunov exponents of $sc$-DA endomorphisms on $\TT^2$. 
    Moreover, by following Proposition \ref{5 prop counter-example}, 
    we give an example of the $sc$-DA endomorphism on $\TT^2$
    which is strictly semi-conjugate to its linearization.
    Further, even some rigidity assumption on a large invariant set cannot guarantee 
    the existence of semi-conjugacy neither. 
    Notice that this phenomenon is different from the $cu$-DA case 
    where the semi-conjugacy must be injective, 
    hence, the semi-conjugacy is a topological conjugacy (see Theorem \ref{main-thm-cu}). 

    First, we prove Theorem \ref{main-thm-sc}. Recall that $f:\mtt$ is a $C^{1}$-smooth 
    $sc$-DA endomorphism 
    and $A:\mtt$ is its hyperbolic linearization. When $f$ is semi-conjugate to $A$ via $h$, 
    we denote by 
        $\Lambda:=\overline{\big\{x\in\TT^2~ |~  h^{-1}\circ h(x)=\{x\}\big\}}$ 
    the  the closure of $h$-injective points set.

\begin{proof}[Proof of Theorem \ref{main-thm-sc}]
    By Corollary \ref{2 cor semi-conj. Fc} and Theorem \ref{main-thm-scu}, 
    the existence of semi-conjugacy and $f$ being special are equivalent in the $sc$-DA case. 
    Moreover, by Proposition \ref{2 prop injective point set}, 
    we have $\overline{{\rm Per}(f|_{\Lambda})}=\Lambda$. 
    Hence, it suffices to prove the rigidity of stable Lyapunov exponents.
    By Proposition \ref{4 prop s-rigidity},
    we already have $\lambda^s(p,f)=\lambda^s(q,f)\equiv \lambda$ for any $p,q\in {\rm Per}(f|_{\Lambda})$,
    so we prove $\lambda=\lambda^s(A)$. 
    
    Let $\Gamma$ be the lifting of $\Lambda$ on $\RR^2$. 
    By Proposition \ref{4 prop adapted metric}, 
    for any $\e>0$, 
    one can find an adapted metric $\|\cdot\|$ such that
    \[
       \lambda-\varepsilon <\log ||DF|_{E_F^s(x)}||< \lambda+\varepsilon,\quad \forall ~x\in\Gamma.
    \]
    If $|\lambda-\lambda^s(A)|\neq 0$, 
    one can pick $\varepsilon\in(0,|\lambda-\lambda^s(A)|)$.
    Then, on the one hand, for any $y\in\mcf_F^s(x)\subset \Gamma$, 
    \begin{align}\label{LAMBDA}
        {\rm exp}(-k(\lambda+\varepsilon))
        <\frac{d_{\mcf^s_F}\big(F^{-k}(x),F^{-k}(y)\big)}{d_{\mcf^s_F}\big(x,y\big)}<
        {\rm exp}(-k(\lambda-\varepsilon)).
    \end{align}
    On the other hand,
    \begin{align*}
        |H\circ F^{-k}(x)-H\circ F^{-k}(y)|
        =|A^{-k}\circ H(x)-A^{-k}\circ H(y)|
        ={\rm exp}(-k\lambda^s(A))\cdot|H(x)-H(y)|.
    \end{align*}
    Note that for any $k\in\NN$,
    \[
    \big|\big(H\circ F^{-k}(x)-H\circ F^{-k}(y)\big)
    -\big(F^{-k}(x)-F^{-k}(y)\big)\big|<2C_H.
    \]
    where $C_H$ is given in Proposition \ref{2 prop semi-conj in R2}.
    Thus, one can get a contradiction from the quasi-isometric property of $\mcf_F^s$.
\end{proof}

Proposition \ref{main-prop} follows from the following result. 
Furthermore, as mentioned in Remark \ref{1 rmk sc} and the beginning of this section, 
we also construct an $sc$-DA endomorphism which has constant periodic stable Lyapunov exponent in a large invariant set, 
but it is not semi-conjugate to its linearization. 
Precisely, we give the following extended result of Proposition \ref{main-prop}.

\begin{prop}\label{5 prop counter-example}
    Let $A:\mtt$ be an Anosov endomorphism with a non-trivial stable bundle. 
    Then there exists a $C^\infty$-smooth 
    $sc$-DA endomorphism $g_0:\mtt$ such that 
    \begin{enumerate}
        \item $g_0$ is semi-conjugate to $A$, and any semi-conjugacy $h$ 
            homotopic to ${\rm Id}_{\TT^2}$ is not injective.
        \item The set \[\Lambda_{g_0}:=\overline{ \{ x\in\TT^2\ |\ h^{-1}\circ h(x)=x \}}\neq \TT^2,\]
            is a uniformly hyperbolic repeller.
    \end{enumerate}
    Moreover, there exists a $C^\infty$-smooth 
    $sc$-DA endomorphism $g_1:\mtt$ 
    homotopic to $A$ such that
    \begin{enumerate}
        \item The semi-conjugacy $H$ guaranteed by Proposition \ref{2 prop semi-conj in R2} 
            does not commute with the deck transformations, and $g_1$ is not special.
        \item The set \[\Gamma_{g_1}:=\{x\in\RR^2\ |\ H^{-1}\circ H(x)=\{x\} \}\] is invariant 
            and satisfies $\Bar{\Gamma}_{g_1}+\ZZ^2=\Bar{\Gamma}_{g_1}$ and $H(\Gamma_{g_1})=\RR^2$.
        \item  The set $\Lambda_{g_1}:=\pi(\Bar{\Gamma}_{g_1})\subset \TT^2$ is compact, 
            $g_1^{\pm 1}$-invariant and $\mcf^s_{g_1}$-saturated.  
        \item $\overline{{\rm Per}(g_1|_{\Lambda_{g_1}})}=\Lambda_{g_1}$ 
            and $\lambda^s(p,g_1)=\lambda^s(A)$ for any $p\in {\rm Per}(g_1|_{\Lambda_{g_1}})$.
    \end{enumerate}
\end{prop}

\begin{proof}
    The construction of $g_0$ is modification for the DA diffeomorphism in \cite[Section 17.2]{KH95},
    but for completeness and some difference between diffeomorphisms and endomorphisms, we give the proof as follows.
    
    Let $A:\mtt$ be a hyperbolic endomorphism with a non-trivial stable bundle. 
    Assume that $A(0)=0$. Let $U$ be a neighborhood of $0$ 
    and we can give a coordinates $(x,y)$ on $U$ diagonalizing $A$ on $U$ such that 
        \[A(x,y)=(\mu_1x,\mu_2y), \quad \forall\ (x,y)\in U, \]
    where 
        \[
        0<\mu_1<1<\mu_2\quand 
        \mu_1\cdot \mu_2={\rm det}(A)\geqslant 2. 
        \]
    In particular, we assume that $U$ is a square with the edge length $\frac{1}{2}$ centering $0$.

    Let $\phi(t):\RR\to [0,1]$ be $C^\infty$ a bump function such that 
    \begin{itemize}
        \item $\phi(t)=1$ if $|t|\leqslant \frac{1}{8}$ and $\phi(t)=0$ if $|t|\geqslant \frac{1}{4}$;
        \item $\phi(-t)=\phi(t)$;
        \item $-9<\phi'(x)<0$ when $\frac{1}{8}<t<\frac{1}{4}$.
    \end{itemize} 
    Let $g_0:\mtt$ be given as 
    \begin{itemize}
        \item $g_0|_{\TT^2\setminus U}=A|_{\TT^2\setminus U}$, 
        \item $g_0(x,y)=A(x,y)+(0,\lambda\phi(x)\phi(ky)y)$ for every $(x,y)\in U$,
    \end{itemize} 
    where $\lambda<0$ and $k\geqslant 1$ are constants and given later. 
    
    It is clear that $g_0$ is $C^{\infty}$. 
    Note that we just perturb $A$ along its unstable manifold $\mcl^u$, 
    hence, the foliation $\mcl^u$ is $g_0$-invariant 
    and gives a $Dg_0$-invariant bundle $L^u$ (may not be expanding under $Dg_0$). 
    To construct the desired $g_0$, we choose
    $\lambda$ and $k$ satisfying the following claims.
   
\begin{claim}\label{5 claim lambda}
    There exist $\lambda$ and $k$ such that
        \begin{enumerate}
            \item There exists fixed point $p\in U\setminus \{0\}$ with $g_0(p)=p$ and $p\in \mcl^u(0)$. 
            \item There exists a $Dg_0$-contracting cone field 
                    \[ \mathcal{C}_a (z):=\{(x,y) \in T_z\TT^2 \ |\ |y|\geqslant a|x| \},   \] 
                namely, $Dg_0\big( \overline{\mathcal{C}_a (z)}  \big) 
                    \subset {\rm int}\big(  \mathcal{C}_a (g_0(z))\big)$.
            \item For every point $z\in\TT^2$ and every vector $v\notin \mathcal{C}_a (z)$, 
                one has $\|D_zg_0v\|<\|v\|$.
    \end{enumerate}
\end{claim}
  
    Note that if $g_0$ satisfies the second item, 
    then there exists a $Dg_0$-invariant dominated splitting 
    $T\TT^2=E_{g_0}^1\oplus_< E_{g_0}^2$ (see \cite{CP15}),  
    in particular, $E_{g_0}^2=L^u$. 
    And $Dg_0$ contracts the bundle $E_{g_0}^1$ by the third item. 
    Hence, $g_0$ is a special $sc$-DA endomorphism,
    and we rewrite its partially hyperbolic splitting as 
        \[ T\TT^2=E_{g_0}^s\oplus_< E_{g_0}^c.\]
    According to Theorem \ref{main-thm-scu}, 
    $g_0$ is semi-conjugacy to $A$. Let surjection $h:\mtt$ homotopic to Id$_{\TT^2}$ 
    such that $h\circ f=A\circ h$. 
    Note that $h$ preserves the center foliation, 
    namely, $h(\mcf_{g_0}^c(0))=\mcl^u(0)$, 
    where $\mcf_{g_0}^c$ is the integral foliation of bundle $E_{g_0}^c$. 
    By the first item and the semi-conjugacy, it follows $h(p)=0$, 
    hence, $h$ is not injective. 
    Now we check that $g_0$ satisfies the last three properties. 

\begin{proof}[Proof of Claim \ref{5 claim lambda}]
   The second and third items automatically hold for $z\in \TT^2\setminus U$ 
    since $g_0|_{\TT^2\setminus U}=A|_{\TT^2\setminus U}$. 
    Hence, we just focus on the perturbed neighborhood $ U$. 
     Let $z=(x,y)\in U$. Then 
    \[ D_{(x,y)}g_0= 
        \begin{bmatrix}
            \mu_1,&0\\
            \lambda\phi'(x)\phi(ky)y,&\mu_2+\lambda\phi(x)\big(k\phi'(ky)y+\phi(ky)\big)
        \end{bmatrix} .\]
    To get another fixed point on $\mcl^u(0)$, 
    let $g_0(x,y)=(x,y)$ and $x=0$. 
    Then we get $\phi(ky)=\frac{1-\mu_2}{\lambda}$. 
    Hence, when $\lambda< 1-\mu_2<0$,
    the equation $\phi(ky)=\frac{1-\mu_2}{\lambda}\in [0,1]$ has a solution $y_0\in(0,\frac{1}{4k}]$. 
    Then $p=(0,y_0)\in U$ is a fixed point satisfying the first item.

    Let $(1,\pm a) \in \mathcal{C}_a(z)$. 
    To prove the second item, we will find $a\in\RR_{\geqslant 0}$ such that 
    \begin{align}
        D_zg_0(1,\pm a)\in {\rm int}\big(\mathcal{C}_a(g_0(z))\big),\quad \forall\ z\in U. 
        \label{eq. 5. cone-field}
    \end{align}
    By symmetry, we need only find $a$ when $x\geqslant 0$ and $y\geqslant 0$. 
    Let \[ D_zg_0(1,a)=(\mu_1, B(x,y)+C(x,y)a),\]
    where \[
    \begin{cases}
        B(x,y)=\lambda\phi'(x)\phi(ky)y,\\
        C(x,y)=\mu_2+\lambda\phi(x)\big(k\phi'(ky)y+\phi(ky)\big).
    \end{cases}\]
    It is easy to get that 
    $B(x,y)\in[0,-\frac{9}{4k}\lambda]$ and 
    $C(x,y)\in[\mu_2+\lambda,\mu_2-\frac{9}{4}\lambda]$,
    when $x\geqslant 0$ and $y\geqslant 0$. 
    Then \eqref{eq. 5. cone-field} follows from the following inequalities
    \begin{align}
         B(x,y)+C(x,y)a>\mu_1 a, \label{eq. 5. cone-field 1}
    \end{align}
    and 
    \begin{align}
         B(x,y)-C(x,y)a<-\mu_1 a. \label{eq. 5. cone-field 2}
    \end{align}
    Recall that $0<\mu_1<1$ and notice that $\lambda$ can be close to $1-\mu_2$. Thus, we can choose $\lambda$ such that $\mu_2+\lambda-\mu_1>0$, hence,
        $C(x,y)-\mu_1\geqslant \mu_2+\lambda-\mu_1>0$. 
    Then \eqref{eq. 5. cone-field 1} always holds when $a>0$. 
    On the other hand, 
    the function \[\frac{B(x,y)}{C(x,y)-\mu_1}\] has an upper bound $0<C_0(k)<+\infty$ 
    and $C_0(k)\to 0$ as $k\to +\infty$. 
    Hence, \eqref{eq. 5. cone-field 1} always holds when $a>C_0(k)$. 
    As a result, we take $a=2C_0(k)$, then $\mathcal{C}_a$ is the cone-field satisfying the second item.

    To get the third property of $g_0$, for $z=(x,y)\in U$ (we still assume that $x,y\geqslant 0$) 
    we just need to show that every $v=(1,b)\in T_z\TT^2$ not in $\mathcal{C}_a(z)$ 
    will always satisfy the following inequality
        \[1+b^2>\mu_1^2+B^2+C^2b^2+2BCb,\]
    where for short we denote $B(x,y),C(x,y)$ by $B,C$, respectively. Then 
        \[ (C^2-1)b^2+2BCb+B^2+\mu_1^2-1<0,\]
    holds for all $-2C_0(k) \leqslant b\leqslant 2C_0(k)$. 
    This can be done since $\mu_1\in(0,1)$, $C\in[\mu_2+\lambda,\mu_2-\frac{9}{4}\lambda]$, 
    $0\leqslant B\leqslant -\frac{9}{4k}\lambda \to 0$ and $C_0(k)\to 0$ 
    as $k\to +\infty$ for given $\lambda< 1-\mu_2$. 
\end{proof}

    From the choices of $\lambda$ in $k$ of Claim \ref{5 claim lambda}, $g_0$ is an $sc$-DA endomorphism.
    And we further choose $\lambda$ and $k$ such that $\Lambda_{g_0}$ is a uniformly hyperbolic repeller.
    
\begin{claim}\label{5 claim lambda2}
    There exist $\lambda$ and $k$ 
    such that 
    the set $\Lambda_{g_0}$ is a uniformly hyperbolic repeller.
\end{claim}
\begin{proof}[Proof of Claim \ref{5 claim lambda2}]

    It suffices to prove that $Dg_0|_{E^c(\Lambda_{g_0})}$ is uniformly expanding. 
    
    Recall that $E^c_{g_0}=L^u$, hence, the conclusion follows from that 
    there exists $C_1>1$ such that 
    \begin{align}
        C(x,y)=\mu_2+\lambda\phi(x)\big(k\phi'(ky)y+\phi(ky)\big)>C_1>1, 
        \quad \forall\ (x,y)\in\Lambda_{g_0}.
        \label{eq. 5. repeller-1}
    \end{align}
    Note that \eqref{eq. 5. repeller-1} always holds when $(x,y)\in \TT^2\setminus U$, 
    hence, we focus on $U$. 
    Let \[V:=\{(x,y)\in U\ |\ C(x,y)\leqslant 1 \}
        = \bigcup_{x\in[-{1}/{4},{1}/{4}]}\big(\{x\}\times V_{x}\big),\]
    where $V_x:=\{ y\in [-{1}/{4},{1}/{4}]\ |\ C(x,y)\leqslant 1\}$.
    By the compactness of $\Lambda_{g_0}$ and $V$, if $V\subset \TT^2\setminus \Lambda_{g_0}$, 
    then \eqref{eq. 5. repeller-1} holds.

    Let $\mu_1-\mu_2<\lambda<1-\mu_2$. Then $0$ is a sink of $g_0$. 
    It is easy to find $\lambda$ and $k$ such that 
    $(0,y_1)$ and $(0,y_2)$, with $\phi(ky_i)=\frac{1-\mu_2}{\lambda}\ (i=1,2)$, 
    are two saddle points on $\mcf^c_{g_0}(0)$.
    Moreover, we can assume that  
    $(0,y_1), (0,y_2),(0,0)$ are the only three fixed points on $\mcf^c_{g_0}(0)\cap U$, 
    hence, the only three periodic points by monotonicity of $g_0$ 
    along the center segment from $(0,y_1)$ to $(0,y_2)$. 
    Indeed, if $(0,y)\in \mcf^c_{g_0}(0)\cap U$ is a fixed point, 
    then there exists $n_0\in\NN$ such that 
        \[ \mu_2y+\lambda\phi(ky)y-y=n_0.\]
    However, for given $\mu_2, \lambda$, 
    we can assume that $U$ is small (let $\phi$ be suitable to $U$) and $k$ is big enough 
    such that $|\mu_2y-y|$ and $|\lambda\phi(ky)y|$ are small, 
    hence, $|\mu_2y+\lambda\phi(ky)y-y|<1$ and $n_0=0$. 

    Note that $C(x,y)\leqslant 1$ is equivalent to 
    \begin{align}
        \frac{d}{dy}\big( k\phi(ky)y \big) \geqslant \frac{1-\mu_2}{\lambda\phi(x)},
        \label{eq. 5. repeller}
    \end{align}
    when $\phi(x)> 0$. 
    And $(x,y)\notin V$ when $\phi(x)=0$. By \eqref{eq. 5. repeller}, one can get that 
    \begin{itemize}
        \item If $|x_1|<|x_2|$, then $V_{x_2}\subset V_{x_1}$;
        \item If $g_0(x,y)=(x',y')$ and $g_0\big(\{x\}\times V_x\big)=\{x'\}\times V_x'$, 
            then $V_x' \subset  V_{x'}$ by $C(x,y)|_{V_x}\leqslant 1$;
        \item $\{0\}\times V_0\subset \TT^2\setminus \Lambda_{g_0}$. 
            Indeed, $V_0$ is a compact subset of the interval $(y_1,y_2)$ 
            on the center leaf $\mcf^c_{g_0}(0)$,
            so there is a neighborhood $U_0$ of $0$ on $\mcf^s_{g_0}(0)$ 
            such that $\{x\}\times V_x\subset \TT^2\setminus \Lambda_{g_0}$ for every $x\in U_0$.
    \end{itemize}
    There exists $n\in\NN$ such that for every $x\in \mcf^s_{g_0}(0)\cap U=\mcl^s(0)\cap U$, 
    one has $g_0^n(x)\in U_0$, hence,
        \[g_0^n\big(\{x\}\times V_x\big) 
        \subset \{g_0^n(x)\}\times V_{g_0^n(x)} 
        \subset \TT^2\setminus \Lambda_{g_0}.\] 
    It follows that $g_0^n(V)\subset \TT^2\setminus \Lambda_{g_0}$ and $V\subset \TT^2\setminus \Lambda_{g_0}$ 
    as $\Lambda_{g_0}$ is invariant.

    Since we have proved that $g_0$ restricted on $\Lambda_{g_0}$ is uniformly hyperbolic,  
    the stable set of a point $z\in\Lambda_{g_0}$ is 
    the union of the stable manifold $\mcf^s_{g_0}$ of its generation set 
        \[P(z)=\left\{w\in\TT^2\ |\ g_0^k(z)=g_0^k(w),\ {\rm for}\ {\rm some}\ k\in\NN \right\},\] 
    that is, \[ W_{g_0}^s(z)=\bigcup_{w\in P(z)}\mcf^s_{g_0}(w).\]
    Since $\Lambda_{g_0}$ is $g_0^{\pm 1}$-invariant 
    and $\mcf^s_{g_0}$-saturated by Proposition \ref{2 prop injective point set},  
    we have $w\in P(z)\subset \Lambda_{g_0}$ and $\mcf^s_{g_0}(w)\subset \Lambda_{g_0}$. 
    Then we get that $\Lambda_{g_0}$ is a uniformly hyperbolic repeller.
\end{proof}

    Combing Claim \ref{5 claim lambda} and Claim \ref{5 claim lambda2}, 
    we obtain the existence of $g_0$ in Proposition \ref{5 prop counter-example}. 

    Now we are going to make another perturbation to obtain $g_1:\mtt$. 
    Recall that \[\Lambda_{g_0}:=\overline{ \{ x\in\TT^2\ |\ h^{-1}\circ h(x)=x \}},\]
    where $h$ is a semi-conjugacy between $g_0$ and $A$.
    Note that $\Lambda_{g_0}\neq \TT^2$ and $\TT^2\setminus \Lambda_{g_0}$ has non-empty interior. 
    By Theorem \ref{main-thm-sc} and Proposition \ref{2 prop injective point set}, 
    \[ \overline{{\rm Per}(g|_{\Lambda_{g_0}})}=\Lambda_{g_0}\] is a $g_0^{\pm 1}$-invariant 
    $\mcf^s_{g_0}$-saturated set and $\lambda^s(p,g_0)=\lambda^s(A)$ 
    for any $p\in {\rm Per}(g|_{\Lambda_{g_0}})$. 
    In fact, even for every $p\in {\rm Per}(f)$,  
    we can calculate directly that $\lambda^s(p,g_0)\equiv {\rm log}\mu_1$, since $Dg_0$ is of triangular form.

    Let $V\subset \TT^2\setminus \Lambda_{g_0}$ be an open set. 
    Let $g_1:\mtt$ be a perturbation of $g_0:\mtt$ such that 
    \begin{itemize}
        \item $g_1|_{\TT^2\setminus V}=g_0|_{\TT^2\setminus V}$;
        \item $g_1$ is an $sc$-DA endomorphism;
        \item $g_1$ is not special, namely, there is no $Dg_1$-invariant center bundle on $T\TT^2$.
    \end{itemize} 
    Note that $g_0$ is in fact a special Anosov endomorphism on $\Lambda_{g_0}$,
    the perturbation can be done by the same method of 
    perturbing a special Anosov endomorphism to be non-special, see \cite{Pr76}. 
    Then there exists a  semi-conjugacy $H:\RR^2\to\RR^2$ 
    between $G_1:\RR^2\to\RR^2$ and $A:\RR^2\to\RR^2$, the lifted maps of $g_1$ and $A$, respectively. 
    Let $H_0:\RR^2\to \RR^2$ be a lifted map of $h$ such that $H_0\circ G_0=A\circ H_0$ 
    where $G_0:\mrr$ is a lifted map of $g_0$.

    By Corollary \ref{main-cor-sc}, $H$ does not commute with the deck transformations. 
    Let \[\Gamma_{g_1}:=\{x\in\RR^2\ |\ H^{-1}\circ H(x)=\{x\} \}.\]
    Since $\Lambda_{g_0}$ is  $g_0^{\pm 1}$-invariant, 
    we have $\Lambda_{g_0}$ is  $g_1^{\pm 1}$-invariant. 
    Moreover, 
        \[\overline{\Gamma}_{g_1}=\widetilde{\Lambda}_{g_0}
        \quand
        H|_{\overline{\Gamma}_{g_1}}=H_0|_{\widetilde{\Lambda}_{g_0}},\]
    where $\widetilde{\Lambda}_{g_0}$ is the lift of $\Lambda_{g_0}$.
    Let $\Lambda_{g_1}:=\pi (\overline{\Gamma}_{g_1})=\Lambda_{g_0}$. 
    Then $g_1$ restricted on $\Lambda_{g_1}$ and $H$ restricted on $\overline{\Gamma}_{g_1}$ have
    the same properties as $g_0|_{\Lambda_{g_0}}$ and $H_0|_{\widetilde{\Lambda}_{g_0}}$. 
    This ends the proof of Proposition \ref{5 prop counter-example}.
\end{proof}

\section{DA endomorphisms with unstable cone-fields}\label{sec-DA-cu}

    In this section, we will prove Theorem \ref{main-thm-cu} and Corollary \ref{main-cor-cu} 
    which focus on the rigidity of center Lyapunov exponents of $cu$-DA endomorphisms on $\TT^2$. 
  
    Let $f$ be a $C^1$-smooth $cu$-DA endomorphism on $\TT^2$ and $A$ be its hyperbolic linearization.
    Comparing to the $sc$-DA case, there are more rigidity properties in the $cu$-DA case, 
    especially, the existence of semi-conjugacy $h$ between $f$ and $A$
    leads to $h$ being a topological conjugacy and $f$ being a special Anosov endomorphism.  
    This rigidity result mainly follows from the next two lemmas. When $f$ is semi-conjugate to $A$,  
    let $\Lambda\subset \TT^2$ be given in Proposition \ref{2 prop injective point set}. 
    The first one claims that 
    the infimum of center Lyapunov exponents of periodic points on $\Lambda$ is negative 
    so that we can apply Proposition \ref{4 prop s-rigidity}.
 
\begin{lem}\label{6 lem attractor}
    Assume that $f$ is semi-conjugate to $A$. 
    Then there exists $p\in {\rm Per}(f|_{\Lambda})$ such that $\lambda^c(p,f)<0$. 
    Moreover, $\lambda^c(q,f)=\lambda^c(p,f)$, for any $p,q\in {\rm Per}(f|_{\Lambda})$. 
    In particular, the set $\Lambda$ is a uniformly hyperbolic attractor.
\end{lem}

    In fact, the hyperbolic attractor $\Lambda$ is the whole torus.
\begin{lem}\label{6 lem Ansoov}
    Assume that $f$ is semi-conjugate to $A$.  
    The hyperbolic attractor $\Lambda$ guaranteed by Lemma \ref{6 lem attractor} is the whole $\TT^2$, 
    hence, $f$ is a special Anosov endomorphism.
\end{lem}  

    Combining Lemma \ref{6 lem attractor} and Lemma \ref{6 lem Ansoov},
    we can get Theorem \ref{main-thm-cu}  immediately.

\begin{proof}[Proof of Theorem \ref{main-thm-cu}]
    By  Lemma \ref{6 lem attractor} and Lemma \ref{6 lem Ansoov}, 
    the $C^{1}$-smooth $cu$-DA $f$ is indeed a special Anosov endomorphism. 
    Then, by the result of \cite[Section 3]{AGGS23}, one has that $f$ is conjugate to $A$ and
    \begin{align}
        \lambda^s(p,f)=\lambda^s(A),\quad \forall\ p\in {\rm Per}(f). 
        \label{eq. 6. s-rigidity}
    \end{align}
    Note that we can also prove \eqref{eq. 6. s-rigidity} directly. 
    Since we already have $\Lambda=\TT^2$, 
    we can apply Lemma \ref{6 lem attractor} again to get that 
    $\lambda^s(p,f)=\lambda^s(q,f)$, for any $p,q\in {\rm Per}(f)$.
    Thus, by the same argument of the proof of Theorem \ref{main-thm-sc}, 
    we also have $\lambda^s(p,f)\equiv\lambda^s(A)$, for any $p\in {\rm Per}(f)$.
\end{proof}

    In \cite{AGGS23}, it is further proved that if  $f$ is a $C^{1+}$-smooth Anosov endomorphism on $\TT^2$,  
    then $f$ is conjugate to its linearization if and only if 
    $f$ has same periodic Lyapunov exponents on stable bundle as the one of its linearization. 
    Combining this result and Theorem \ref{main-thm-cu}, we can complete the proof of Corollary \ref{main-cor-cu}. 
    For convenience, we recall the four items in Corollary \ref{main-cor-cu} as follows:
     \begin{enumerate}
      \item $f$ is semi-conjugate to $A$ via $h:\mtt$;
      \item $f$ is conjugate to $A$ via $h:\mtt$;
      \item $f$ is special;
      \item $f$ is Anosov and $\lambda^s(p,f)=\lambda^s(A)$, for any $p\in {\rm Per}(f)$.
  \end{enumerate}

\begin{proof}[Proof of Corollary \ref{main-cor-cu}]
    By  Theorem \ref{main-thm-scu} and Theorem \ref{main-thm-cu}, 
    we have $(3)\Longrightarrow (1)$ and $(1)\Longrightarrow (4)$. 
    It follows from \cite{AGGS23} that $(4)\Longrightarrow (2)$, where requires that $f$ is $C^{1+}$-smooth. 
    Finally, since $(2)\Longrightarrow (1)$ is trivial, 
    $(2)\Longrightarrow (1)\Longrightarrow (3)$ can be done by Theorem \ref{main-thm-cu}. 
    Finally, by the result of \cite[Sections 4 and 5]{AGGS23}, $h$ is smooth along each leaf of the stable foliation. 
\end{proof}

    Now we prove Lemma \ref{6 lem attractor} and Lemma \ref{6 lem Ansoov}. 
    For proving Lemma \ref{6 lem attractor}, 
    we need to use the following Ruelle inequality for endomorphisms which is given in \cite{Liu03}, and
    here we give an adaptation under the setting of this paper. 
    Let $f$ be a $C^{1}$-smooth $cu$-DA endomorphism and $\mu$ be an ergodic measure of $f$. 
    Then, by \cite[Remark 1.4]{Liu03},
    \begin{align}
        h_\mu(f) \leqslant F_\mu(f)- \min\{0,\lambda^c(\mu,f)\}, 
        \label{eq. 6. shulin}
    \end{align}
    where  $F_\mu(f)$ is the folding entropy of $(f,\mu)$ 
    whose definition and more properties can be found in \cite{Liu03, WZ21}. 
    We note that the folding entropy is invariant under an isomorphism, 
    namely, if $(f_1,\mu_1)$ is isomorphic to $(f_2,\mu_2)$, 
    then $F_{\mu_1}(f_1)=F_{\mu_2}(f_2)$. 

\begin{proof}[Proof of Lemma \ref{6 lem attractor}]
    First, we prove the existence of periodic points in $\Lambda$ with negative center Lyapunov exponents. 
    Let $h:\mtt$ be a semi-conjugacy between $f$ and $A$. 
    Let $\mu=h^*({\rm Leb})$ where Leb is the Lebesgue measure.
    Note that $(f,\mu)$ is isomorphic to $(A,{\rm Leb})$. 
    Indeed, the set 
        \[\Lambda':=\{x\in\TT^2\ |\ \# h^{-1}\circ h(x)>1 \}=\overline{\TT^2\setminus \Lambda},\] 
    has $\mu(\Lambda')={\rm Leb}(h(\Lambda'))=0$, 
    since the intersection of $\Lambda'$ 
    and every center leaf $\mcf^c_f(x)$ is the union of countable close intervals 
    whose image under $h$ is the set of countable points in $\mcl^s(h(x))$. 
    We also mention that the support of $\mu$ has supp$(\mu)=\Lambda$ 
    by the same method in \cite[Section 6]{U12}.

    Note that $(A,{\rm Leb})$ is ergodic, so is $(f,\mu)$ 
    since  $(f,\mu)$ is isomorphic to $(A,{\rm Leb})$. 
    Moreover, \[h_{\mu}(f)=h_{{\rm Leb}}(A) \quand F_{\mu}(f)=F_{\rm Leb}(A).\] 
    Then, by Formula \eqref{eq. 6. shulin}, we can obtain $\lambda^c(\mu,f)<0$. 
    Indeed, if it is false, i.e., $\lambda^c(\mu,f)\geqslant 0$, then 
    \begin{align*}
        \lambda^u(A)=h_{\rm top}(A)=h_{\rm Leb}(A)&=h_{\mu}(f)
        \leqslant F_{\mu}(f)=F_{\rm Leb}(A)={\rm log}\big(|{\rm det}(A)|\big).
    \end{align*}
    This contradicts the fact that $\lambda^u(A)>{\rm log}\big(|{\rm det}(A)|\big)$. 
    Now by Lemma \ref{4 lem measure shadowing}, 
    there exists a sequence of periodic points $p_n\in \Lambda$ 
    such that \[\lim_{n\to +\infty}\lambda^c(p_n,f)=\lambda^c(\mu,f)<0.\] 
    Hence, there exists $p\in {\rm Per}(f|_{\Lambda})$ such that $\lambda^c(p,f)<0$. 
    It follows from Proposition \ref{4 prop s-rigidity} that 
    $\lambda^c(p,f)$ is constant with respect to $p\in {\rm Per}(f|_{\Lambda})$ 
    and we denote this constant by $\lambda<0$.

    Now we are going to prove that $\Lambda$ is a uniformly hyperbolic attractor.  
    In fact, by Proposition \ref{4 prop adapted metric}, 
    there exists a Riemannian metric such that 
        \[ {\rm log}\|Df|_{E_f^c(x)}\|<\frac{\lambda}{2}<0, \quad \forall\ x\in\Lambda. \] 
    This implies that $f|_{\Lambda}$ is uniformly hyperbolic.  
    Moreover, $\Lambda$ is an attractor, 
    since the unstable set of a point $x\in\Lambda$ is exactly the unstable manifold $\mcf^u_f(x)$ 
    and $\Lambda$ is $\mcf^u_f$-saturated by Proposition \ref{2 prop injective point set}.
    This ends the proof of Lemma \ref{6 lem attractor}.
\end{proof}

    Finally, we prove $\Lambda=\TT^2$ which implies that $f$ is a special Anosov endomorphism. 
    Recall that this phenomenon does not happen in the $sc$-DA case, 
    see the counterexamples given in Proposition \ref{5 prop counter-example} 
    where $\Lambda$ is a hyperbolic repeller but a proper subset of $\TT^2$.  
    The following proof is inspired by \cite[Proposition 4.10]{HS21}.
    One can see the generation set 
        \[P(x)=\{y\in\TT^2\ |\ f^k(x)=f^k(y),\ {\rm for}\ {\rm some}\ k\in\NN \},\] 
    as the strong stable ``leaf'' of  $x\in\Lambda$.

\begin{proof}[Proof of Lemma \ref{6 lem Ansoov}]
    Note that we just need to prove $\Lambda=\TT^2$, 
    then by Lemma \ref{6 lem attractor} $f$ is Anosov and $f$ is special 
    since $f|_{\Lambda}$ is special (Proposition \ref{2 prop injective point set}) and $\Lambda=\TT^2$.

    By contradiction, we assume that $\Lambda\neq \TT^2$,
    that is, there is a point $x\in\TT^2$ with $h^{-1}\circ h(x)\neq \{x\}$. 
    It is convenient to assume that 
    $x$ is the center point of the segment $h^{-1}\circ h(x)\subset \mcf^c_f(x)$. 
    Let $\delta>0$ such that the ball of $x$, with radius $3\delta$ in $\TT^2$, has 
    \begin{align}
        B_{3\delta}(x)\cap \Lambda=\emptyset. \label{eq. 6. attrator delta}
    \end{align} 
    By the uniform continuity of $h$, 
    there exists $\e_0>0$ such that for any $y,z\in\TT^2$ if $d(y,z)<\e_0$, 
    then $d\big(h(y),h(z)\big)<{1}/{2}$. 
    And there exists $0<\e<\min\{\delta,\e_0\}$ such that 
    the $\e$-neighborhood of $\Lambda$ has 
    \begin{align}
         f^k\big(B_\e(\Lambda)\big)
        \subset B_\delta(\Lambda),\quad \forall\ k\in\NN,\label{eq. lem6.3.2}
    \end{align}
    since $\Lambda$ is a hyperbolic attractor. 
    It follows that 
    \begin{align}
          f^{-k}\big(B_\delta(x)\big)\cap B_\e(\Lambda)
        =\emptyset, \quad \forall\ k\in\NN. \label{eq. lem6.3.3}
    \end{align}
    Indeed, if there exist $k_0 \in\NN$ and $y_0\in \TT^2$ such that 
    $y_0\in f^{-k_0}\big(B_\delta(z)\big)\cap \big(B_\e(\Lambda)\big)$,
    then 
    \[f^{k_0}(y_0)\in B_\delta(x)\cap f^{k_0}\big(B_\e(\Lambda)\big)
        \subset B_\delta(x)\cap B_\delta(\Lambda).\] 
    This contradicts \eqref{eq. 6. attrator delta}. 
    Hence, for every $k\in\NN$, $x_{-k}\in f^{-k}(x)$, and $x_{-k}'=h(x_{-k})$, we have 
    \begin{align}
        B_\e(x_{-k})\cap \Lambda =\emptyset 
        \quand 
        h\big(B_\e(x_{-k})\big) \subset \mcl^u\big(x_{-k}',1\big), 
        \label{eq. 6. ball not intersect}
    \end{align}
    where $\mcl^u(x_{-k}',1)$ is the local leaf of $\mcl^u(x_{-k}')$ centering at $x_{-k}'$  
    with the length smaller than $1$.

    Note that we can choose infinitely many distinct preimages
    $x'_{-k_i}\ (i\in\NN)$ of $x'$ under $A$ satisfying 
       \begin{align}
             \mcl^u(x'_{-k_i},1)  \cap \mcl^u(x'_{-k_j},1) =\emptyset, \quad  \forall\  i\neq j\in \NN. \label{eq. lem6.3.0}
       \end{align}
    Indeed, we can get these points $x_{-k_i}'$ as follows. 
    Let $x'_{-k_1}=x'_{-1}\in A^{-1}(x')$. 
    There is small size $0<\eta\ll 1$ such that 
    \begin{align}
        \mcl^u(z,1)  \cap \mcl^u(w,1) =\emptyset, \quad \forall\ z,w\in B_{\eta}(x'_{-1})\ {\rm with}\ z\neq w.\label{eq. lem6.3.1}
    \end{align}
    By the denseness of preimage sets (Lemma \ref{2 lem preimage dense linear}), we can take 
    \[x_{-k_2}\in A^{-k_2}(x')\cap \big(B_{\eta}(x'_{-1})\setminus\{{x_{-k_1}}\}\big)\]  
    for some $k_2\in \NN$. By induction,  we can take 
    \[x_{-k_i}\in A^{-k_i}(x')\cap \big(B_{\eta}(x'_{-1})\setminus\{{x_{-k_{1}},\cdots, x_{-k_{i-1}}}\}\big)\] 
    for each $i\geqslant 2$ and some $k_i>k_{i-1}$. 
    By \eqref{eq. lem6.3.1}, these points $x_{-k_i}\ (i\in\NN)$ satisfy \eqref{eq. lem6.3.0}.
    
    Combining \eqref{eq. 6. ball not intersect} and \eqref{eq. lem6.3.0}, we obtain that 
        \[ B_\e(x_{-k_i})\cap B_\e(x_{-k_j}) =\emptyset, \quad  \forall\  i\neq j\in \NN. \]
    This is impossible, since the volume of a ball 
    $ B_\e(x_{-k_i})\ (i\in\NN)$
    with radius $\e$ has a lower bound. 
    
    Thus, $h$ is injective and $\Lambda=\TT^2$. 
    This ends the proof of Lemma \ref{6 lem Ansoov}.
\end{proof}

\begin{rmk}\label{rmk cs no hyper}
     We mention that the method in the proof of Lemma \ref{6 lem Ansoov} is invalid in the $sc$-DA case. In this case, the set $\Lambda$ may not be a repeller, 
     since one can get the information of $Df|_{E_f^s(\Lambda)}$ (see Theorem \ref{main-thm-sc})
     but has no idea about $Df|_{E_f^c(\Lambda)}$. 
     Hence, one cannot get a similar property as \eqref{eq. lem6.3.2}. 
     Furthermore, even if $\Lambda$ is a hyperbolic repeller (e.g., the $\Lambda_{g_0}$ constructed in Proposition \ref{5 prop counter-example}), one can have $B_\e(\Lambda)\subset f^k\big( B_\delta(\Lambda) \big)$ for all $k\in\NN$ instead of \eqref{eq. lem6.3.2}, this still does not imply that $f^{-k}\big(B_\e(\Lambda)\big)\subset B_\delta(\Lambda) $ for all $k\in\NN$. Thus, one cannot get $f^k\big(B_\delta(x)\big)\cap B_\e(\Lambda)=\emptyset$ for all $k\in\NN$ as a replacement of \eqref{eq. lem6.3.3} to show that $\Lambda$ is the whole tours.
\end{rmk}


\section{Applications for the volume-preserving case}\label{sec-volume-preserving}

    In this section, we prove Corollary \ref{main-cor-scvp} and Corollary \ref{main-cor-Jac-rigidity}, 
    and we also give an example of a volume-preserving Anosov endomorphism on the torus 
    whose Jacobian is not constant on periodic points. 

    First, we give the following property on the periodic boundaries. 
    {\color{black}A point $p$ is called a \textit{pre-periodic} point of $f$ if there is $n\in\NN$ such that $f^n(p)\in {\rm Per}(f)$.}
    Let $\Lambda\subset \TT^2$ be given in Proposition \ref{2 prop injective point set}. 
    Recall that the boundaries of sets $\Lambda$ and $\TT^2\setminus \Lambda$ are hyperbolic leaves. 
    Note that $x\in \partial(\TT^2\setminus \Lambda)$ if and only if 
    $h^{-1}\circ h(x)\neq \{x\}$ and $x\in \Lambda$.

\begin{lem}\label{7 lem periodic on buondary}
    Let $f:\mtt$ be a  $C^{1}$-smooth $sc$-DA endomorphism semi-conjugate to its linearization $A:\mtt$ via $h:\mtt$. 
    Let $\Lambda\subset \TT^2$ be given in Proposition \ref{2 prop injective point set} 
    and $\Lambda\neq \TT^2$. 
    If one of the following conditions holds:
    \begin{enumerate}
    \item There exists $\lambda>0$ such that $\lambda^c(p,f)>\lambda$ for any $p\in {\rm Per}(f|_{\Lambda})$.
        \item The non-wondering set $\Omega(f)=\TT^2$.
    \end{enumerate}
    Then each stable leaf of the boundary $\partial(\TT^2\setminus \Lambda)$ admits a unique \textcolor{black}{pre-periodic or periodic} point.
\end{lem}

\begin{proof}
    By Theorem \ref{main-thm-sc}, $f$ is special and admits an invariant center bundle $E^c_f$ which is integrable to foliation $\mcf^c_F$. Without loss of generality, let $f$ be orientation-preserving when restricted on $\mcf^c_F$.
   
    We consider the first condition.  Let $\lambda>0$ such that $\lambda^c(p,f)>\lambda$ for all $p\in {\rm Per}(f|_{\Lambda})$.

\begin{claim}\label{7 claim uniform size of unstable leaf}
    There exists $\eta>0$ such that for any $x\in\Lambda$,  
    the length of any component $I_{k}$ of $f^{-k}\big(\mcf^c_f(x,\eta)\big)$  goes to $0$, 
    i.e., \[ |I_{k}|\to 0,\quad {\rm as}\ k\to +\infty.\]
    In particular, the center leaf of a point in $\Lambda$ with size $\eta$ is in fact an unstable manifold.
\end{claim}
\begin{proof}[Proof of Claim \ref{7 claim uniform size of unstable leaf}]
    According to Lemma \ref{4 lem measure shadowing}, for every ergodic measure $\mu$ supported on $\Lambda$, 
    one has $\lambda^c(\mu,f)\geqslant \lambda$. 
    By the same construction of the adapted metric in Proposition \ref{4 prop adapted metric},
    there exists a Riemannian metric such that  
        \[ {\rm log}\|Df|_{E_f^c(x)}\|>\frac{\lambda}{2}>0, \quad \forall\ x\in\Lambda. \]
    Then there exists $\eta>0$ such that 
    for any $x\in\Lambda$ and any $y\in \mcf^c_f(x,\eta)$, one has 
        \[ {\rm log}\|Df|_{E_f^c(y)} \|>\frac{\lambda}{4}>0.  \]
    Recall that the set $\Lambda$ is $f^{\pm}$-invariant.
    Hence, the constant $\eta$ is the size of unstable manifolds.
\end{proof}

    Let $a=a^-\in \partial(\TT^2\setminus\Lambda)$ 
    and $\eta>0$ be given in Claim \ref{7 claim uniform size of unstable leaf}. 
    Let $J$ be an open interval on $\mcf^c_f(a,\eta)$ such that 
    it has an endpoint $a^-$ and $J\subset \TT^2\setminus\Lambda$. 
    Considering the stable manifold of $J$,
        \[ \mcf^s_f(J)=\bigcup_{x\in J}\mcf^s_f(x)\subset \TT^2\setminus\Lambda,\]
    there exists $C_0>0$ such that for every $k\in\NN$
    the volume of the set $f^k\big(\mcf^s_f(J)\big)$ is bigger than $C_0$. Indeed, on the one hand, the length of curve $f^k(J)$ has a lower bound by  Claim \ref{7 claim uniform size of unstable leaf}.  
    On the other hand,  the stable bundle $E^s_f$ is uniformly transverse to the center bundle $E^c_f$. 
    Thus, we can take a curve $J_k\subset f^k(J)$ with uniform length and a stable foliation box of $J_k$ with uniform size. 
    Then the volume of this foliation box gives the lower bound $C_0$.

    \textcolor{black}{If $\mcf^s_f(a)$ is neither a pre-periodic leaf nor a periodic leaf,}
    then one has
    \begin{align}
      f^{i}\big(\mcf^s_f(J)\big)\cap f^{j}\big(\mcf^s_f(J)\big)=
        \emptyset,\quad \forall\  i\neq j\in\NN. \label{eq. notintersect}  
    \end{align}
    Indeed, since $f$ preserves the orientation of $\mcf^c_f$, the point
    $$f^i(a)=f^i(a^-)=(f^i(a))^-$$  
    is the negative endpoint of the curve $h^{-1}\circ h (f^i(a))$, for all $i\in \NN$.  
    If \eqref{eq. notintersect} is not true, that is, there are $i_0\neq j_0\in \NN$ such that 
    $$  
    f^{i_0}\big(\mcf^s_f(J)\big)\cap f^{j_0}\big(\mcf^s_f(J)\big)\neq \emptyset,
    $$ 
    then we have 
    $$
    f^{i_0}\big(\mcf^s_f(a)\big)\cap f^{j_0}\big(\mcf^s_f(J)\big)\neq\emptyset\ \ {\rm or}\ \ f^{j_0}\big(\mcf^s_f(a)\big)\cap f^{i_0}\big(\mcf^s_f(J)\big)\neq\emptyset.
    $$
    This contradicts the fact that 
    $$
    f^i\big(\mcf^s_f(a)\big)\cap f^j\big(\mcf^s_f(J)\big)=\emptyset, \quad \forall\  i\neq j\in\NN,
    $$ 
    since the sets $\Lambda$ (including $\mcf^s_f(a)$) and $\TT^2\setminus\Lambda$ (including $\mcf^s_f(J)$) are $f^{\pm}$-invariant. 

    Combining \eqref{eq. notintersect} and the fact that the volume of $f^k\big(\mcf^s_f(J)\big)$ has lower boundary $C_0$, 
    the volume of the union of sets $f^k\big(\mcf^s_f(J)\big)$ is infinite. However, this is impossible. 
    \textcolor{black}{Thus, $\mcf^s_f(a)$ is a pre-periodic leaf or periodic leaf and admits a unique pre-periodic or periodic point.}

    Now, we consider the other case. Let $\Omega(f)=\TT^2$. 
    In this case, the proof is inspired by \cite[Proposition A.5]{HHU08}. 
    Let $a=a^-\in \partial(\TT^2\setminus\Lambda)$. 
    Then there exists a close interval $I$ on $\mcf^c_f(a)$ with 
    $I\setminus\{a\}\subset \TT^2\setminus\Lambda$. 
    Take a stable foliation box $U_I$ of $I$. 
    Note that $U_I\subset \TT^2\setminus\Lambda$ is an open set in $\TT^2$, 
    there are $x\in U_I$ and $k\in\NN$ such that $f^k(x)\in U_I$. 
    Let $y$ be the unique intersection point of local leaves $\mcf^s_f(a)$ and $\mcf^c_f(x)$. 
    Then one has $y\in \Lambda$, since $\Lambda$ is $\mcf^s_f$-saturated. 

    We claim $f^k(y)\in \mcf^s_f(a)$. 
    Indeed,  one has $f^k(y) \notin U_I$ as $\Lambda$ is invariant. 
    If $f^k(y)\notin \mcf^s_f(a)$, 
    then the interval $ f^k(I\setminus\{a\})$ on the leaf $\mcf^c_f(f^k(x))$ 
    intersects with $\mcf^s_f(a)$ at least one point $z\in \Lambda$. 
    However, this is impossible, since $z\in f^k(I\setminus\{a\})$ and $f^k(I\setminus\{a\})\cap \Lambda=\emptyset$. 
    So one has $y, f^k(y)\in \mcf^s_f(a)$.
    Finally, by the contraction of stable leaf, there exists exactly one periodic point on $\mcf^s_f(a)$.
\end{proof}

\begin{rmk}
    Note that the center leaf in Claim \ref{4 claim topological contracting leaf} 
    is just a half leaf of a boundary point, 
    so we cannot apply it to the proof of Lemma \ref{7 lem periodic on buondary} directly.
\end{rmk}

\begin{prop}\label{7 prop}
    Let $f:\mtt$ be a $C^{1}$-smooth $sc$-DA endomorphism semi-conjugate to its linearization $A:\mtt$. 
    If $f$ satisfies one of the following conditions:
    \begin{enumerate}
    \item There exists $\lambda>0$ 
            such that $\lambda^c(p,f)>\lambda$ for any $p\in {\rm Per}(f)$.
        \item $f$ is volume-preserving.
    \end{enumerate}
    Then $f$ is a special Anosov endomorphism and $\lambda^s(p,f)=\lambda^s(A)$, for any $p\in {\rm Per}(f)$.
\end{prop}

\begin{proof}
    First, we prove that 
    the set $\Lambda$ given in Theorem \ref{main-thm-sc} (see also Proposition \ref{2 prop injective point set}) is in fact $\TT^2$. 
    
    Note that when $f$ is volume-preserving, the non-wondering set of the $sc$-DA will be $\TT^2$. 
    So, by Lemma \ref{7 lem periodic on buondary}, both conditions imply that each stable leaf of the boundary $\partial(\TT^2\setminus \Lambda)$ admits \textcolor{black}{a unique pre-periodic or periodic point.} 
    For convenience, let $p^+\in \partial(\TT^2\setminus \Lambda)$ be fixed.
    Then $p^-\neq p^+$ is the other endpoint of $h^{-1}\circ h(p^+)$. 
    Denote by $I_p\subset \TT^2\setminus \Lambda$ the open interval on $\mcf^c_f(p^+)$ with endpoints $p^\pm$,  
    then we have $f(\Bar{I}_p)=\bar{I}_p$. Now we deal with these two cases, respectively.

    \begin{enumerate}
        \item Let $f$ satisfy the first condition, namely, 
        there exists $\lambda>0$ such that $\lambda^c(p,f)>\lambda$ for any $p\in {\rm Per}(f)$.  
        Note that $f:\Bar{I}_p\to \Bar{I}_p$ is a $C^1$-smooth orientation-preserving map with 
            \[{\rm log}\|Df|_{E^c_f(p^\pm)}\|> \lambda>0.\]
        This implies that there exists a fixed point $z_1\in I_p$. 
        By ${\rm log}\|Df|_{E^c_f(z_1)}\|> \lambda>0$, 
        there exists another fixed point $z_2\in I_p$ between $p^-$ and $z_1$. 
        Repeating this, we get infinitely many fixed points $z_n$ on $I_p$ with an accumulation point $z\in \Bar{I}_p$. 
        Thus, $\|Df|_{E^c_f(z)}\|=1$ by the definition of derivative.
        It contradicts the assumption ${\rm log}\|Df|_{E^c_f(z)}\|> \lambda>0$. 
        
        \item  Let $f$ be volume-preserving. 
        Consider the open interval $I_p\subset \TT^2\setminus \Lambda$, let 
            \[U=\bigcup_{x\in \Bar{I}_p}\mcf^s_f(x,\e)\]
        be a stable foliation box of $\bar{I}_p$ with the volume $C_\e$.  
        Then, by the contraction of stable leaves, the component of $f^{-1}(U)$ containing $I_p$ has the volume bigger than $C_\e$. This contradicts \textcolor{black}{the assumption of the volume-preserving property of $f$.}
    \end{enumerate}
    Hence, we get that $\Lambda=\TT^2$. In particular, $f$ is indeed conjugate to $A$. 
    By Theorem  \ref{main-thm-sc}, $f$ is a special $sc$-DA endomorphism and $\lambda^s(p,f)=\lambda^s(A)$, for any $p\in {\rm Per}(f)$. 
    
    Next, we prove that $f$ is Anosov. 
    
    Note that if $f$ satisfies the first condition, namely, the infimum of the periodic center Lyapunov exponents is strictly bigger than $0$, one can choose a Riemannian metric such that $f$ is expanding on $E^c_f$ by a similar argument of Proposition \ref{4 prop adapted metric}. Notice that $f$ has the shadowing property by inheriting from $A$ via the conjugacy, this situation is actually easier than that of the non-uniformly shadowing property (Lemma \ref{4 lem measure shadowing}) which we used in the proof of Proposition \ref{4 prop adapted metric}.
    Thus, the first condition further implies that $f$ is Anosov.
    
{\color{black}
    The rest of the proof is to show that if $f$ preserves a volume measure $m$, then $f$ is also Anosov. Let $\rho:\TT^2\to \RR_+$ be the density of $m$, that is, $dm(x)=\rho(x)d{\rm Leb}(x)$. 
    Let $\lambda^s(A)=-\mu<0$. Recall that for any $p\in {\rm Per}(f)$, $\lambda^s(p,f)=\lambda^s(A)=-\mu$. 
    It suffices to prove 
    \begin{align}
       \lambda^c(p,f)\geqslant \mu>0,\quad  
        \forall\ p\in {\rm Per}(f).  \label{eq. c-exp positive}
    \end{align}
    So we will reduce the volume-preserving condition to the first case and conclude that $f$ is Anosov. 

    For any point $x\in\TT^2$,  let $B_\e(p)$ be the open ball centered at $p$ with small radius $\e>0$. 
    Then for each periodic point $p\in {\rm Per}(f)$ of period $k$, one has
    \begin{align*}
        m\big(f^k(B_\e(p)) \big)&=\int_{f^k(B_\e(p))}\rho(y)d{\rm Leb}(y)\\
        &=\int_{B_\e(p)} {\rm Jac}(Df^k(x))\frac{\rho(f^k(x))}{\rho(x)}dm(x)={\rm Jac}(Df^k(x_\e))\frac{\rho(f^k(x_\e))}{\rho(x_\e)}m\big(B_\e(p)\big),
    \end{align*}
    for some $x_\e\in B_\e(p)$. Since $f$ is volume-preserving,  
    $$  
    m\big(B_\e(p)\big)\leqslant m\big(f^{-k}\circ f^{k}(B_\e(p))\big)=m\big(f^{k}(B_\e(p))\big).
    $$ 
    Hence, ${\rm Jac}(Df^k(x_\e))\frac{\rho(f^k(x_\e))}{\rho(x_\e)}\geqslant 1$. 
    Let $\e\to 0$, then $x_\e\to p$ and ${\rm Jac}(Df^k(p))\geqslant 1$. 
    Namely, $\lambda^s(p,f)+\lambda^c(p,f)\geqslant 0$. 
    This gives us \eqref{eq. c-exp positive}. 
    Thus, $f$ is Anosov when it is volume-preserving.}
\end{proof}

    Now we can obtain Corollary \ref{main-cor-scvp} from Proposition \ref{7 prop} immediately.
\begin{proof}[Proof of Corollary \ref{main-cor-scvp}]
    The proof is similar to the proof of Corollary \ref{main-cor-cu}. In fact, we just need to replace Theorem \ref{main-thm-cu} by Proposition \ref{7 prop}.
\end{proof}

    Now we prove Corollary \ref{main-cor-Jac-rigidity}.
    It should be mentioned here that though 
    the constant periodic Jacobian data imply the existence of an invariant volume measure 
    for toral Anosov endomorphisms, 
    we cannot use this result under the assumption of Corollary \ref{main-cor-Jac-rigidity} 
    where $f$ is not \emph{a priori} Anosov.
    This indicates the important role of Proposition \ref{7 prop} as well.

\begin{proof}[Proof of Corollary \ref{main-cor-Jac-rigidity}]
    Note that for any $n\in\NN$ and $p\in{\rm Fix}(f^n)$, one has
        $${\rm Jac}(Df^n(p))=|{\rm det}(A)|^n,$$
    and $Df|_{E^s_f}$ is uniformly contracting, 
    so there exists $\lambda>0$ with $\lambda^c(p,f)>\lambda$. 
    Then, by Proposition \ref{7 prop}, $f$ is a special Anosov endomorphism. 
    Thus, according to \cite[Theorem 1.3]{GS22}, $f$ is smoothly conjugate to $A$ with the regularity we desired. 
    It follows from the smooth conjugacy that $f$ inherits the volume-preserving property from $A$.
\end{proof}

    Finally, we give a volume-preserving Anosov endomorphism which has no Jacobian rigidity.

{\color{black}   
\begin{example}\label{7 example}
    A volume-preserving Anosov endomorphism $f:\TT^2\to \TT^2$ without Jacobian rigidity.
\end{example}   

We start with a linear hyperbolic endomorphism. Let
    \[
    A = \begin{bmatrix} 3 & 1 \\ 1 & 1 \end{bmatrix}: \mathbb{T}^2 \to \mathbb{T}^2.
    \]

It is clear that $\deg(A) = \det(A) = 2$ and $A$ has a unique fixed point $O_1 = (0,0)^T$. 
Then $A^{-1}(O_1)$ consists of two points $O_1$ and $O_2 = (\frac{1}{2}, \frac{1}{2})^T$. 
Let $U_1$ and $U_2$ be two neighborhoods of points $O_1$ and $O_2$, respectively, such that
    \[
    A(U_1) = A(U_2) = U_0,
    \]
where $U_0$ is a small neighborhood of $O_1$ such that $A|_{U_1}$ and $A|_{U_2}$ are bijections onto $U_0$ 
and there exists a periodic orbit of $A$, say $\operatorname{Orb}_A(p)$, disjoint from $U_i$, for all $i = 0,1,2$. 

Let $v_u$ be the unit eigenvector corresponding to the eigenvalue $\lambda_u = 2 + \sqrt{2}$. 
Let $\eta: U_0 \to \mathbb{R}$ be a bump function on $U_0$ such that $\eta(O_1) = 0$ and $\nabla\eta(O_1) = v_u \neq 0$. 
We define a vector field $X: U_0 \to \mathbb{R}^2$ strictly aligned with the unstable manifold of $A$:
\[
    X(x) = \lambda_u \eta(x)v_u.
\]
Note that $O_1$ is a singularity of $X$, and its divergence is given by 
$\operatorname{div}X(x) = \lambda_u (\nabla\eta(x))^T v_u$. 
Then we perturb the linear inverse branches of $A$ with $B_i=(A|_{U_i})^{-1}:U_0\to U_i$ for $i=1,2$,
using the induced maps of $X$. 
Specifically, we consider two maps
\[
    g_i: U_0 \to U_i, \quad g_i(x) = B_i\left(x + (-1)^{i-1} \cdot \varepsilon \cdot X(x)\right),
\]
where $\varepsilon > 0$ is small and $i = 1,2$. Let $f: \mathbb{T}^2 \to \mathbb{T}^2$ be such that
\[
\begin{cases}
    g_1 \circ f(x) = x, & x \in U_1, \\
    g_2 \circ f(x) = x, & x \in U_2, \\
    f(x) = A(x),        & x \notin U_1 \cup U_2.
\end{cases}
\]

When $\varepsilon > 0$ is small enough,
$f$ is a smooth Anosov endomorphism with a $Df$-invariant unstable bundle $E_f^u = E_A^u$. 
Note that $f$ is topologically conjugate to $A$, while there is no smooth conjugacy between $f$ and $A$.
To the end, we show that $f$ is volume-preserving but has no Jacobian rigidity. 

By Sylvester's determinant identity for rank-1 updates (i.e., $\det(I + uv^T) = 1 + v^T u$), 
the Jacobian of $g_i$ can be explicitly computed:
Using the chain rule on $g_i(x)$, we have
\begin{align*}
    Dg_i(x) 
    &= DB_i\left(x + (-1)^{i-1} \cdot \varepsilon \cdot X(x)\right) \cdot \left( I + (-1)^{i-1} \varepsilon DX(x) \right) \\
    &= A^{-1} \cdot \left( I + (-1)^{i-1} \varepsilon \lambda_u v_u (\nabla\eta(x))^T \right),
\end{align*}
where $A$ is regarded as a matrix, then, taking the determinant of both sides, we can get
\begin{align*}
\det(Dg_i(x)) &= \det(A^{-1}) \cdot \det\left( I + (-1)^{i-1} \varepsilon \lambda_u v_u (\nabla\eta(x))^T \right) \\
&= \frac{1}{2} \left( 1 + (-1)^{i-1} \varepsilon \lambda_u (\nabla\eta(x))^T v_u \right).
\end{align*}
Note that the divergence of the vector field is exactly $\operatorname{div}X(x) = \operatorname{Tr}(DX(x)) = \lambda_u (\nabla\eta(x))^T v_u$, the expression simplifies to
\[
    \det(Dg_i(x)) = \frac{1}{2}\big(1 + (-1)^{i-1}\varepsilon \operatorname{div}X(x)\big).
\]

Moreover, by $\operatorname{div}X(O_1) = \lambda_u \neq 0$, we obtain
\[
\det(Df(O_1)) = \frac{1}{\det Dg_1(O_1)} = \frac{2}{1 + \varepsilon \operatorname{div}X(O_1)} \neq 2,
\]
while we have $g_1(O_1) = B_1(O_1) = O_1$ and $f(O_1) = O_1$ by $X(O_1) = 0$.
However, since $\operatorname{Orb}_A(p)$ is disjoint from the perturbation regions, 
we have $p \in \operatorname{Per}(f)$ and $\det(Df|_{\operatorname{Orb}_f(p)}) \equiv \det(A) = 2$. 
Thus, $f$ has different period-averaged Jacobians at $O_1$ and $p$. 

Furthermore, for any point $x \in \mathbb{T}^2$, let $f^{-1}(x) = \{x_1, x_2\}$. 
On the one hand, for $x \notin U_0$, 
\begin{equation}\label{eq. example 1}
    \frac{1}{\det(Df(x_1))} + \frac{1}{\det(Df(x_2))} = \frac{1}{2} + \frac{1}{2} = 1. 
\end{equation}
On the other hand, for $x \in U_0$, 
\begin{equation}\label{eq. example 2}
\begin{aligned}
\frac{1}{\det(Df(x_1))} + \frac{1}{\det(Df(x_2))} 
&= \det(Dg_1(x)) + \det(Dg_2(x)) \\
&= \frac{1}{2}(1 + \varepsilon \operatorname{div}X(x)) + \frac{1}{2}(1 - \varepsilon \operatorname{div}X(x)) = 1.
\end{aligned} 
\end{equation}
Hence, by \eqref{eq. example 1} and \eqref{eq. example 2}, $f$ is volume-preserving.
}

\section*{Acknowledgement}
	The authors would like to thank Yi Shi for suggesting this rigidity problem, 
	which is motivated by a question of J\'{e}r\^{o}me Buzzi 
	at the conference of Beyond Uniform Hyperbolicity. 
    The authors also would like to thank Ra\'{u}l Ures for comments and suggestions on the preliminary version of this paper, 
    and anonymous referees for comments and suggestions which are helpful to improve this paper.
    This work was supported by National Key R\&D Program of China (2021YFA1001900). 
    R. Gu was partially supported by the New Cornerstone Science Foundation through the New Cornerstone Investigator Program, 
    and M. Xia was partially supported by NSFC (Nos. 12161141002, 12501238) and
    the Fundamental Research Funds for the Central Universities (DUT24RC(3)112).

\bibliographystyle{plain}
\bibliography{Bib-DE2024}

\end{CJK}
\end{document}